\documentclass[11pt,a4paper,leqno]{amsart}

\title
[Microlocal analysis for Gelfand--Shilov spaces]
{Microlocal analysis for Gelfand--Shilov spaces}

\author[L. Rodino]{Luigi Rodino}

\address{Department of Mathematics, Universit\`a di Torino, Via Carlo Alberto 10,
10123 Torino, Italy}

\email{luigi.rodino[AT]unito.it}

\author[P. Wahlberg]{Patrik Wahlberg}

\address{Dipartimento di Scienze Matematiche, Politecnico di Torino, Corso Duca degli Abruzzi 24,
10129 Torino, Italy}

\email{patrik.wahlberg[AT]polito.it}

\usepackage{latexsym}
\usepackage[a4paper]{geometry}
\usepackage{amsmath}
\usepackage{amssymb}
\usepackage{amsthm}
\usepackage{bbm}
\usepackage{amsfonts}
\usepackage{mathrsfs}
\usepackage{calc}
\usepackage{cite}
\usepackage{color}
\usepackage{graphicx}
\usepackage{enumerate}

\numberwithin{equation}{section}          

\newtheorem{thm}{Theorem}
\numberwithin{thm}{section}

\newcommand{\rubrik}{}
\newtheorem{prop}[thm]{Proposition}
\newtheorem{cor}[thm]{Corollary}
\newtheorem{lem}[thm]{Lemma}

\theoremstyle{definition}

\newtheorem{defn}[thm]{Definition}

\theoremstyle{remark}

\newtheorem{rem}[thm]{Remark}              

\newcommand{\Ker}{\operatorname{Ker}}
\newcommand{\scal}[2]{\langle #1,#2\rangle}
\newcommand{\pd}[1] {\partial ^#1}
\newcommand{\pdd}[2] {\partial_{#1} ^{#2}}

\newcommand{\ro}{\mathbf R}
\newcommand{\no}{\mathbf N}
\newcommand{\rr}[1]{\mathbf R^{#1}}
\newcommand{\sr}[1]{\mathbf S^{#1}}
\newcommand{\sro}[1]{\mathbf S}
\newcommand{\nn}[1]{\mathbf N^{#1}}

\newcommand{\zo}{\mathbf Z}

\newcommand{\co}{\mathbf C}
\newcommand{\cc}[1]{\mathbf C^{#1}}

\newcommand{\dd}{\mathrm {d}}

\newcommand{\nm}[2]{\Vert #1\Vert _{#2}}

\newcommand{\sgn}{\operatorname{sgn}}

\newcommand{\ep}{\varepsilon}
\newcommand{\fy}{\varphi}

\newcommand{\cdo}{\, \cdot \, }

\newcommand{\supp}{\operatorname{supp}}

\newcommand{\eabs}[1]{\langle #1\rangle}

\newcommand{\Sp}{\operatorname{Sp}}

\newcommand{\GL}{\operatorname{GL}}

\newcommand{\rB}{\operatorname{B}}

\newcommand{\WF}{\mathrm{WF}}
\newcommand{\WFg}{\mathrm{WF_g}}

\newcommand{\cS}{\mathscr{S}}
\newcommand{\cE}{\mathscr{E}}

\newcommand{\cD}{\mathscr{D}}
\newcommand{\cF}{\mathscr{F}}

\newcommand{\J}{\mathcal{J}}

\newcommand{\wh}{\widehat}
\newcommand{\re}{{\rm Re}}
\newcommand{\im}{{\rm Im}}
\def\la{\langle}
\def\ra{\rangle}
\newcommand{\leqs}{\leqslant}
\newcommand{\geqs}{\geqslant}

\begin{document}

\begin{abstract}
We introduce an anisotropic global wave front set of Gelfand--Shilov ultradistributions with different indices
for regularity and decay at infinity. 
The concept is defined by the lack of super-exponential decay along power type curves in the phase space of 
the short-time Fourier transform. 
This wave front set captures the phase space behaviour of oscillations of power monomial type, 
a k a chirp signals.
A microlocal result is proved with respect to pseudodifferential operators 
with symbol classes that give rise to continuous operators on Gelfand--Shilov spaces. 
We determine the wave front set of certain series of derivatives of the Dirac delta, 
and exponential functions. 
\end{abstract}


\maketitle

\section{Introduction}\label{sec:intro}

Gelfand--Shilov spaces, for $t > 0$ and $s > 0$, are defined by
\begin{equation}\label{eq:GelfandShilovest}
|x^\alpha D ^\beta f(x)|
\leqs C h^{|\alpha + \beta |} \alpha !^t \, \beta !^s
\end{equation}
which we assume to be valid for every $h > 0$ and a suitable $C > 0$ depending on $h$
(spaces of Beurling type $\Sigma_t^s(\rr d)$), 
or else for some $h > 0$ and some $C > 0$ (Roumieu type $\mathcal S_t^s(\rr d)$). 
The ultradistributions $(\Sigma_t^s)'(\rr d)$, $(\mathcal S_t^s)'(\rr d)$ are defined as their respective topological duals. 
Attention in our paper will be limited to the Beurling case 
under the assumption $t + s > 1$ granting $\Sigma_t^s(\rr d) \neq \{ 0\}$. 
The definition was introduced in \cite{Gelfand2}, and then analyzed in various contexts, 
with application to linear and nonlinear partial differential equations, in connection also
with problems in Mathematical Physics. 
The literature on the subject is extremely wide, see for example \cite{Pilipovic1,Debrouwere1,Teofanov1}
for recent contributions to the general theory, and \cite{Cappiello0a,Carypis1,Morimoto1,Morimoto2}
concerning travelling waves, Boltzmann and Schr\"odinger equations. 
In particular, Gelfand--Shilov spaces have been considered in the framework of 
pseudodifferential operators. 
Namely, classes of pseudodifferential operators were introduced, 
with symbols satisfying suitable factorial and exponential estimates, 
acting continuously on Gelfand--Shilov spaces,  
see for example \cite{Abdeljawad1,Cappiello2}. 

In our paper we shall refer to the class of symbols satisfying 
\begin{equation}\label{eq:symbolest}
|\pdd  x \alpha  \pdd \xi\beta a(x,\xi)|
\leqs C h^{|\alpha + \beta |} \alpha !^s \, \beta !^t
e^{\mu \left( |x|^{\frac1t} + |\xi|^{\frac1s} \right)}
\end{equation}
for some $\mu > 0$ and all $h > 0$, with $C > 0$ depending on $h$. 
This symbol class was introduced in \cite{Abdeljawad1}. 
The corresponding Weyl operators $a^w(x,D)$ were proved to act continuously on $\Sigma_t^s(\rr d)$
and on $(\Sigma_t^s)'(\rr d)$
in \cite[Theorem~3.15]{Abdeljawad1}. 

Our attention will be actually addressed to another ingredient of the microlocal analysis: the wave front set. 
The classical definition of H\"ormander \cite{Hormander0} in the setting of Schwartz distributions
was extended in different ways. 
In particular H\"ormander \cite{Hormander1} introduced for $u \in \cS'(\rr d)$ the notion of $\WFg(u)$
adapted to the study of global regularity in $T^* \rr d \setminus 0$. 
Let us recall the definition by using the short-time Fourier transform (Gabor transform) 
with window $\fy \in \cS(\rr d) \setminus 0$, cf. \cite{Rodino2}: 
\begin{equation*}
V_\fy u (x,\xi) = (2\pi )^{-\frac d2} \int_{\rr d} e^{- i \scal y \xi} u(y) \overline{\fy(y-x)} \dd y. 
\end{equation*}
We have $z_0 = (x_0,\xi_0) \notin \WFg(u)$, $z_0 \neq 0$, if
\begin{equation}\label{eq:conicdecay1}
\sup_{z \in \Gamma} \eabs{z}^N |V_\fy u (z)| < \infty \quad \forall N \geqs 0 
\end{equation}
for a suitable conic neighborhood $\Gamma$ of $z_0$ in $\rr {2d} \setminus 0$. 

Looking for a counterpart of \eqref{eq:conicdecay1} in the Gelfand--Shilov setting, 
we may start with the equivalent definition of the $\Sigma_t^s(\rr d)$ regularity 
of $u \in (\Sigma_t^s)'(\rr d)$ given by the estimates, with window $\fy \in \Sigma_t^s(\rr d) \setminus 0$, 
\begin{equation}\label{eq:STFTGFstequiv}
| V_\fy u (x,\xi)| \lesssim e^{-r (|x|^{\frac1t} + |\xi|^{\frac1s})} \quad \forall r > 0. 
\end{equation}
For the equivalence with \eqref{eq:GelfandShilovest} see for example \cite{Toft1}. 

Hence in the case $s = t$ we may define as $\Sigma_s^s(\rr d)$ regularity at $z_0 \in T^* \rr d \setminus 0$
\begin{equation}\label{eq:conicdecay2}
\sup_{z \in \Gamma} e^{r |z|^{\frac1s}} |V_\fy u (z)| < \infty \quad \forall r > 0 
\end{equation}
where again $\Gamma$ is a conic neighborhood of $z_0$ in $\rr {2d} \setminus 0$. 
Based on \eqref{eq:conicdecay2}, the Gelfand--Shilov wave front set for $s = t$ was recently defined and used
in applications to partial differential equations \cite{Boiti1,Cappiello1,Carypis1}. 
Let us address for some early ideas to \cite{Hormander1}, and to the theory of Fourier hyperfunctions
\cite{Kaneko1,Kaneko2}. 

If $s \neq t$, cones $\Gamma \subseteq T^* \rr d \setminus 0$ are not anymore appropriate to micro-localize the decay of the Gabor transform 
in \eqref{eq:STFTGFstequiv}. 
The natural idea is to replace the standard cones with anisotropic cones, 
namely we replace the straight lines through $(x_0,\xi_0) \in T^* \rr d \setminus 0$ 
with the curves $\{x = \lambda^t x_0, \ \xi = \lambda^s \xi_0, \ \lambda > 0 \}$
and we define the anisotropic cone as the union of such curves through a neighborhood $U \subseteq T^* \rr d \setminus 0$ of $(x_0,\xi_0)$. 
The required decay to define $(x_0,\xi_0) \notin \WF^{t,s} (u)$ can then be expressed by 
\begin{equation*}
\sup_{\lambda > 0, \ (x,\xi) \in U} e^{r \lambda} |V_\fy u(\lambda^t x, \lambda^s \xi)| < \infty, \quad \forall r > 0. 
\end{equation*}

Let us describe in short the contents of the paper. 
Section \ref{sec:prelim} is devoted to some preliminaries. 
We give in particular a new proof of the celebrated Peetre inequality;
the optimality of the constant in our formula seems new in the literature, surprisingly. 
The definition of $\WF^{t,s} (u)$ is reported in Section \ref{sec:GelfandShilovWF}.
We give there examples about $\WF^{s,s} (u)$, i.e. the case $s=t$, 
and then prove invariance properties under change of window and the action 
of certain metaplectic operators. 

Section \ref{sec:chirp} is devoted to chirp signals, providing an interesting example
of anisotropic wave front set. 
Namely in dimension $d=1$, for 
\begin{equation}\label{eq:chirp1}
u (x) = e^{i c x^{m}}, \quad m \in \no \setminus \{ 0,1 \}, \quad c \in \ro \setminus 0, 
\end{equation}
we obtain if $t (m-1) > 1$
\begin{equation}\label{eq:WFchirp1}
\WF^{t, t (m-1)}(u) = \{ (x, \xi = c m x^{m-1} ) \in \rr 2, \ x \neq 0 \}.  
\end{equation}

Section \ref{sec:GSGevrey} is addressed to the relations between the Gelfand--Shilov wave front set
and the Gevrey wave front set $\WF_s (u)$ for $u \in (\Sigma_t^s)'(\rr d)$, $s > 1$. 
We shall refer to \cite{Rodino1}, results given there for the Roumieu case being easily translated
to the present Beurling framework. 

The main result of the paper is in Section \ref{sec:microlocal}, where we prove the microlocal inclusion
\begin{equation}\label{eq:microlocal0}
\WF^{t,s}( a^w(x,D) u ) \subseteq \WF^{t,s}(u), \quad u \in (\Sigma_t^s)'(\rr d), 
\end{equation}
for symbols satisfying \eqref{eq:symbolest}. 
Several examples are then given. 
Namely in Section \ref{sec:polynomials} we compute $\WF^{t,s} (u)$ for polynomials
and finite linear combinations of derivatives of the delta distribution $\delta_0$. 
The analysis extends to ultradistributions of the form
\begin{equation*}
u = \sum_{\alpha \in \nn d} c_\alpha D^\alpha \delta_0
\end{equation*}
under suitable bounds on the coefficients $c_\alpha \in \co$, and their Fourier transforms. 

In Section \ref{sec:exponential} we first consider $e^{\la \cdot, z \ra} \in (\Sigma_t^s)'(\rr d)$,
with $z \in \cc d$ fixed, $t \leqs 1$. 
From \eqref{eq:microlocal0} we obtain
\begin{equation*}
\WF^{t,s} ( e^{\la \cdot, z \ra}  ) = (\rr d \setminus 0) \times \{ 0 \}. 
\end{equation*}
Combining with the example \eqref{eq:chirp1} in dimension $d = 1$, we then consider
\begin{equation*}
u(x) = e^{z x + i c x^{m}}
\end{equation*}
and we deduce for $\WF^{t,s} (u)$ the same identity \eqref{eq:WFchirp1}. 

In conclusion, we would like to observe that anisotropic cones are not a novelty in microlocal analysis. 
They were used as a partition of the space $\rr d$ of the dual variables by \cite{Lascar1} and \cite{Parenti1}, 
soon followed by other authors, see for more recent contributions \cite{Garello1} and its references.  
In these papers the anisotropic cones in $\rr d$ are used as a suitable option in the microlocal study
of equations of parabolic type, whereas in our case the anisotropy in $T^* \rr d$ is forced by the very structure of the function spaces. 

As background we mention recent works of ours (written after this paper) concerning anisotropic global wave front sets and their propagation for certain evolution equations \cite{Rodino3,Wahlberg4,Wahlberg5}. 

Applications to partial differential equations will be given in a sequel of this paper. 
We are then inspired by \cite{Cappiello0b}, where the authors prove Gelfand--Shilov regularity for operators of the type
\begin{equation*}
P = - \Delta + |x|^{2m}, \quad m \in \no \setminus 0. 
\end{equation*}
We aim for microlocal versions of this result, as well as 
propagation of singularities for Schr\"odinger operators of the form
\begin{equation*}
Q = i \partial_t - P = i \partial_t + \Delta - |x|^{2m}. 
\end{equation*}
%

\section{Preliminaries}\label{sec:prelim}

An open ball in $\rr d$ of radius $r > 0$ centered at $x \in \rr d$ is denoted $\rB_r(x)$, and $\rB_r(0) = \rB_r$. 
The unit sphere is denoted $\sr {d-1} \subseteq \rr d$. 
The group of invertible matrices in $\rr {d \times d}$ is $\GL(d,\ro)$, 
and the determinant of $A \in \rr {d \times d}$ is $|A|$. 
The transpose of $A \in \rr {d \times d}$ is denoted $A^T$ and the inverse transpose 
of $A \in \GL(d,\ro)$ is $A^{-T}$. 
The derivative $D_j = - i \partial_j$ is used extended to multi-indices. 
We write $f (x) \lesssim g (x)$ provided there exists $C>0$ such that $f (x) \leqs C \, g(x)$ for all $x$ in the domain of $f$ and of $g$. 
We use the bracket $\eabs{x} = (1 + |x|^2)^{\frac12}$ for $x \in \rr d$. 
Peetre's inequality is usually stated as
\begin{equation*}
\eabs{x+y}^s \leqs 2^{\frac{|s|}{2}} \eabs{x}^s\eabs{y}^{|s|}\qquad x,y \in \rr d,
\qquad s \in \ro, 
\end{equation*}
but in fact the constant can be improved as follows. 

\begin{lem}\label{Peetresharp}
We have 
\begin{equation*}
\eabs{x+y}^s \leqs \left( \frac{2}{\sqrt{3}} \right)^{|s|} \eabs{x}^s\eabs{y}^{|s|}\qquad x,y \in \rr d, \quad s \in \ro, 
\end{equation*}
where the constant is optimal. 
\end{lem}

\begin{proof}
It suffices to show 
\begin{equation*}
\sup_{x,y \in \rr d} \frac{1 + |x + y|^2}{ (1 + |x|^2) (1 + |y|^2) } = \frac43.
\end{equation*}

If $|x| = 2^{- \frac12}$ and $y = x$ then 
\begin{equation*}
\frac{1 + |x + y|^2}{ (1 + |x|^2) (1 + |y|^2) }  = \frac{1 + 4 |x|^2}{ (1 + |x|^2)^2} = \frac43
\end{equation*}
so it remains to show
\begin{equation*}
3 (1 + |x + y|^2 ) \leqs 4 (1 + |x|^2) (1 + |y|^2), \quad x, y \in \rr d. 
\end{equation*}
The latter inequality can be written
\begin{equation*}
4 \la x, y \ra \leqs 1 + | x - y |^2 + 4 |x|^2 |y|^2
\end{equation*}
whose truth is a consequence of $(2 |x| |y| - 1)^2 \geqs 0$ and the Cauchy--Schwarz inequality. 
\end{proof}

The normalization of the Fourier transform is
\begin{equation*}
 \cF f (\xi )= \widehat f(\xi ) = (2\pi )^{-\frac d2} \int _{\rr
{d}} f(x)e^{-i\scal  x\xi }\, dx, \qquad \xi \in \rr d, 
\end{equation*}
for $f\in \cS(\rr d)$ (the Schwartz space), where $\scal \cdo \cdo$ denotes the scalar product on $\rr d$. 
The conjugate linear action of a (ultra-)distribution $u$ on a test function $\phi$ is written $(u,\phi)$, consistent with the $L^2$ inner product $(\cdo ,\cdo ) = (\cdo ,\cdo )_{L^2}$ which is conjugate linear in the second argument. 

Denote translation by $T_x f(y) = f( y-x )$ and modulation by $M_\xi f(y) = e^{i \scal y \xi} f(y)$ 
for $x,y,\xi \in \rr d$ where $f$ is a function or distribution defined on $\rr d$. 
The composition is denoted $\Pi(x,\xi) = M_\xi T_x$. 
Let $\fy \in \cS(\rr d) \setminus \{0\}$. 
The short-time Fourier transform (STFT) \cite{Cordero1} of a tempered distribution $u \in \cS'(\rr d)$ is defined by 
\begin{equation*}
V_\fy u (x,\xi) = (2\pi )^{-\frac d2} (u, M_\xi T_x \fy) = \cF (u T_x \overline \fy)(\xi), \quad x,\xi \in \rr d. 
\end{equation*}
Then $V_\fy u$ is smooth and polynomially bounded \cite[Theorem~11.2.3]{Grochenig1}. 
When $u \in \cS(\rr d)$ it is instead superpolynomially decreasing, that is 
\begin{equation*}
|V_\fy u (x,\xi)| \lesssim \eabs{(x,\xi)}^{-N}, \quad (x,\xi) \in T^* \rr d, \quad \forall N \geqs 0.  
\end{equation*}
The inverse transform is given by
\begin{equation}\label{eq:STFTinverse}
u = (2\pi )^{-\frac d2} \iint_{\rr {2d}} V_\fy u (x,\xi) M_\xi T_x \fy \, \dd x \, \dd \xi
\end{equation}
provided $\| \fy \|_{L^2} = 1$, with action under the integral understood, that is 
\begin{equation}\label{eq:moyal}
(u, f) = (V_\fy u, V_\fy f)_{L^2(\rr {2d})}
\end{equation}
for $u \in \cS'(\rr d)$ and $f \in \cS(\rr d)$, cf. \cite[Theorem~11.2.5]{Grochenig1}.

\subsection{Spaces of functions and ultradistributions}

Let $s,t, h > 0$. 
The space denoted $\mathcal S_{t,h}^s(\rr d)$
is the set of all $f\in C^\infty (\rr d)$ such that
\begin{equation}\label{eq:seminormSigmas}
\nm f{\mathcal S_{t,h}^s}\equiv \sup \frac {|x^\alpha D ^\beta
f(x)|}{h^{|\alpha + \beta |} \alpha !^t \, \beta !^s}
\end{equation}
is finite, where the supremum is taken over all $\alpha ,\beta \in
\mathbf N^d$ and $x\in \rr d$.
The function space $\mathcal S_{t,h}^s$ is a Banach space which increases
with $h$, $s$ and $t$, and $\mathcal S_{t,h}^s \subseteq \cS$.
The topological dual $(\mathcal S_{t,h}^s)'(\rr d)$ is
a Banach space such that $\cS'(\rr d) \subseteq (\mathcal S_{t,h}^s)'(\rr d)$.

The Beurling type \emph{Gelfand--Shilov space}
$\Sigma _t^s(\rr d)$ is the projective limit 
of $\mathcal S_{t,h}^s(\rr d)$ with respect to $h$ \cite{Gelfand2}. This means
\begin{equation}\label{GSspacecond1}
\Sigma _t^s(\rr d) = \bigcap _{h>0} \mathcal S_{t,h}^s(\rr d)
\end{equation}
and the Fr{\'e}chet space topology of $\Sigma _t^s (\rr d)$ is defined by the seminorms $\nm \cdot{\mathcal S_{t,h}^s}$ for $h>0$.

If $s + t > 1$ then $\Sigma _t^s(\rr d)\neq \{ 0\}$ \cite{Petersson1}. 
The topological dual of $\Sigma _t^s(\rr d)$ is the space of (Beurling type) \emph{Gelfand--Shilov ultradistributions} \cite[Section~I.4.3]{Gelfand2}
\begin{equation}\tag*{(\ref{GSspacecond1})$'$}
(\Sigma _t^s)'(\rr d) =\bigcup _{h>0} (\mathcal S_{t,h}^s)'(\rr d).
\end{equation}

The dual space $(\Sigma _t^s)'(\rr d)$ may be equipped with several topologies: the weak$^*$ topology, the strong topology, the Mackey topology, and the topology defined by the union \eqref{GSspacecond1}$'$ as an inductive limit topology \cite{Schaefer1}. The latter topology is the strongest topology such that the inclusion $(\mathcal S_{t,h}^s)'(\rr d) \subseteq (\Sigma _t^s)'(\rr d)$ is continuous for all $h > 0$.  

The Roumieu type Gelfand--Shilov space is the union 
\begin{equation*}
\mathcal S_t^s(\rr d) = \bigcup _{h>0}\mathcal S_{t,h}^s(\rr d)
\end{equation*}
equipped with the inductive limit topology \cite{Schaefer1}, that is 
the strongest topology such that each inclusion $\mathcal S_{t,h}^s(\rr d) \subseteq\mathcal S_t^s(\rr d)$
is continuous. 
Then $\mathcal S _t^s(\rr d)\neq \{ 0\}$ if and only if $s+t \geqs 1$ \cite{Gelfand2}. 
The corresponding (Roumieu type) Gelfand--Shilov ultradistribution space is 
\begin{equation*}
(\mathcal S_t^s)'(\rr d) = \bigcap _{h>0} (\mathcal S_{s,h}^t)'(\rr d). 
\end{equation*}
For every $s,t > 0$ such that $s+t > 1$, and for any $\ep > 0$ we have
\begin{equation*}
\Sigma _t^s (\rr d)\subseteq \mathcal S_t^s(\rr d)\subseteq
\Sigma _{t+\ep}^{s+\ep}(\rr d).
\end{equation*}
We will not use the Roumieu type spaces in this article but mention them as a service to a reader interested in a wider context. 

We write $\Sigma _s^s (\rr d) = \Sigma_s (\rr d)$ and $(\Sigma _s^s)' (\rr d) = \Sigma_s' (\rr d)$. 
Then $\Sigma_s(\rr d) \neq \{ 0 \}$ if and only if $s > \frac12$.

The Gelfand--Shilov (ultradistribution) spaces enjoy invariance properties, with respect to 
translation, dilation, tensorization, coordinate transformation and (partial) Fourier transformation.
The Fourier transform extends 
uniquely to homeomorphisms on $\mathscr S'(\rr d)$, from $(\mathcal
S_t^s)'(\rr d)$ to $(\mathcal
S_s^t)'(\rr d)$, and from $(\Sigma _t^s)'(\rr d)$ to $(\Sigma _s^t)'(\rr d)$, and restricts to 
homeomorphisms on $\mathscr S(\rr d)$, from $\mathcal S_t^s(\rr d)$ to $\mathcal S_s^t(\rr d)$, 
and from $\Sigma _t^s(\rr d)$ to $\Sigma _s^t(\rr d)$, and to a unitary operator on $L^2(\rr d)$.
Likewise \eqref{eq:moyal} holds when $u \in (\Sigma_t^s)'(\rr d)$, $f \in \Sigma_t^s(\rr d)$, $\fy \in \Sigma_t^s(\rr d)$ and $\| \fy \|_{L^2} = 1$. 

At one occasion we will need Gelfand--Shilov spaces defined on $\rr {2d}$ which has possibly different behavior with respect to the two $\rr d$ coordinates \cite{Abdeljawad1,Cappiello2,Gelfand2}. Then the seminorms \eqref{eq:seminormSigmas} 
are generalized into 
\begin{equation}\label{eq:seminormSigmas2}
\nm f{\mathcal S_{t_1,t_2,h}^{s_1,s_2}}\equiv \sup \frac {|x_1^{\alpha _1} x_2^{\alpha _2} D_{x_1} ^{\beta_1} D_{x_2} ^{\beta_2}
f( x_1, x_2 )|}{h^{|\alpha_1 + \alpha_2 + \beta_1 + \beta_2 |} \alpha_1 !^{t_1} \, \alpha_2 !^{t_2} \beta_1 !^{s_1} \beta_2 !^{s_2}}
\end{equation}
for $t_j, s_j > 0$, $j = 1,2$. The spaces $\Sigma_{t_1, t_2}^{s_1,s_2}(\rr {2d})$ and $(\Sigma_{t_1, t_2}^{s_1,s_2}) '(\rr {2d})$ are defined as above.

Working with Gelfand--Shilov spaces we will often need the inequality (cf. \cite{Cappiello2})
\begin{equation*}
|x+y|^{\frac1s} \leqs \kappa(s^{-1} ) ( |x|^{\frac1s} + |y|^{\frac1s}), \quad x,y \in \rr d, \quad s > 0, 
\end{equation*}
where 
\begin{equation*}
\kappa (t)
= 
\left\{
\begin{array}{ll}
1 & \mbox{if} \quad 0 < t \leqs 1 \\
2^{t-1} & \mbox{if} \quad t > 1
\end{array}
\right. ,
\end{equation*}
which implies 
\begin{equation*}
\begin{aligned}
e^{r |x+y|^{\frac1s} } & \leqslant e^{ \kappa(s^{-1} ) r |x|^{\frac1s}} e^{ \kappa(s^{-1} )r |y|^{\frac1s}},  \quad x,y \in \rr d, \quad r >0, \\
e^{- r \kappa(s^{-1} ) |x+y|^{\frac1s} } & \leqslant e^{- r |x|^{\frac1s}} e^{ \kappa(s^{-1} ) r |y|^{\frac1s}}, \quad x,y \in \rr d, \quad r >0. 
\end{aligned}
\end{equation*}

We will often use the following estimate where we use $|\alpha|! \leqs \alpha! d^{|\alpha|}$ for $\alpha \in \nn d$ \cite[Eq.~(0.3.3)]{Nicola1}. For any $s > 0$, $h > 0$ and any $\alpha \in \nn d$ we have 
\begin{equation}\label{eq:expestimate0}
\alpha!^{-s} h^{- |\alpha|} 
= \left( \frac{h^{- \frac{|\alpha|}{s}}}{\alpha!} \right)^s
\leqs  \left( \frac{ \left( d h^{- \frac{1}{s}} \right)^{|\alpha|}}{|\alpha|!} \right)^s
\leqs e^{s d h^{- \frac{1}{s}}}. 
\end{equation}
%

\subsection{Weyl pseudodifferential operators}

Finally we need some elements from the calculus of pseudodifferential operators \cite{Folland1,Hormander0,Nicola1,Shubin1}. 
Let $a \in C^\infty (\rr {2d})$ and $m \in \ro$. Then $a$ is a \emph{Shubin symbol} of order $m$, denoted $a\in \Gamma^m$, if for all $\alpha,\beta \in \nn d$ there exists a constant $C_{\alpha,\beta}>0$ such that
\begin{equation}\label{eq:shubinineq}
|\partial_x^\alpha \partial_\xi^\beta a(x,\xi)| \leqs C_{\alpha,\beta} \langle (x,\xi)\rangle^{m-|\alpha + \beta|}, \quad x,\xi \in \rr d.
\end{equation}
The Shubin symbols $\Gamma^m$ form a Fr\'echet space where the seminorms are given by the smallest possible constants in \eqref{eq:shubinineq}.

For $a \in \Gamma^m$ a pseudodifferential operator in the Weyl quantization is defined by
\begin{equation}\label{eq:weylquantization}
a^w(x,D) f(x)
= (2\pi)^{-d}  \int_{\rr {2d}} e^{i \langle x-y, \xi \rangle} a \left(\frac{x+y}{2},\xi \right) \, f(y) \, \dd y \, \dd \xi, \quad f \in \cS(\rr d),
\end{equation}
when $m<-d$. The definition extends to general $m \in \ro$ if the integral is viewed as an oscillatory integral.
The operator $a^w(x,D)$ then acts continuously on $\cS(\rr d)$ and extends uniquely by duality to a continuous operator on $\cS'(\rr d)$.
By Schwartz's kernel theorem the Weyl quantization procedure may be extended to a weak formulation which yields continuous linear operators $a^w(x,D):\cS(\rr{d}) \to \cS'(\rr{d})$, even if $a$ is only an element of $\cS'(\rr{2d})$.
Likewise $a^w(x,D):\Sigma_s(\rr{d}) \to \Sigma_s'(\rr{d})$ if $a \in \Sigma_s'(\rr {2d})$ and $s > \frac12$. 

If $s > \frac12$ and $a \in \Sigma_s'(\rr {2d})$ the Weyl quantization extends a continuous operator $\Sigma_s(\rr d) \rightarrow \Sigma_s'(\rr d)$
that satisfies
\begin{equation}\label{eq:wignerweyl}
(a^w(x,D) f, g) = (2 \pi)^{-d} (a, W(g,f) ), \quad f, g \in \Sigma_s(\rr d), 
\end{equation}
where the cross-Wigner distribution is defined as 
\begin{equation*}
W(g,f) (x,\xi) = \int_{\rr d} g (x+y/2) \overline{f(x-y/2)} e^{- i \la y, \xi \ra} \dd y, \quad (x,\xi) \in \rr {2d}. 
\end{equation*}
We have $W(g,f) \in \Sigma_s(\rr {2d})$ when $f,g \in \Sigma_s(\rr d)$. 

The real phase space $T^* \rr d \simeq \rr d \oplus \rr d$ is a real symplectic vector space equipped with the 
canonical symplectic form
\begin{equation*}
\sigma((x,\xi), (x',\xi')) = \langle x' , \xi \rangle - \langle x, \xi' \rangle, \quad (x,\xi), (x',\xi') \in T^* \rr d. 
\end{equation*}
This form can be expressed with the inner product as $\sigma(X,Y) = \la \J X, Y \ra$ for $X,Y \in T^* \rr d$
where 
\begin{equation}\label{eq:Jdef}
\J =
\left(
\begin{array}{cc}
0 & I_d \\
-I_d & 0
\end{array}
\right) \in \rr {2d \times 2d}. 
\end{equation}
The real symplectic group $\Sp(d,\ro)$ is the set of matrices in $\GL(2d,\ro)$ that leaves $\sigma$ invariant. 
Hence $\J \in \Sp(d,\ro)$. 

To each symplectic matrix $\chi \in \Sp(d,\ro)$ is associated an operator $\mu(\chi)$ that is unitary on $L^2(\rr d)$, and determined up to a complex factor of modulus one, such that
\begin{equation}\label{symplecticoperator}
\mu(\chi)^{-1} a^w(x,D) \, \mu(\chi) = (a \circ \chi)^w(x,D), \quad a \in \cS'(\rr {2d})
\end{equation}
(cf. \cite{Folland1,Hormander0}).
The operator $\mu(\chi)$ is a homeomorphism on $\mathscr S$ and on $\mathscr S'$.
The same conclusions hold if $a \in \Sigma_s'(\rr {2d})$ in the functional framework $\Sigma_s$, $\Sigma_s'$ if $s > \frac12$. 
In fact $\mu(\chi)$ is a homeomorphism on $\Sigma_s(\rr d)$ which extends uniquely to a homeomorphism on $\Sigma_s'(\rr d)$ \cite[Proposition~4.4]{Carypis1}. 

The mapping $\Sp(d,\ro) \ni \chi \rightarrow \mu(\chi)$ is called the metaplectic representation \cite{Folland1}.
It is in fact a representation of the so called $2$-fold covering group of $\Sp(d,\ro)$, which is called the metaplectic group.  
The metaplectic representation satisfies the homomorphism relation modulo a change of sign:
\begin{equation*}
\mu( \chi \chi') = \pm \mu(\chi ) \mu(\chi' ), \quad \chi, \chi' \in \Sp(d,\ro).
\end{equation*}
%

\section{The Gabor and the $t,s$-Gelfand--Shilov wave front sets}\label{sec:GelfandShilovWF}

First we define the Gabor wave front set $\WFg$ introduced in \cite{Hormander1} and further elaborated in \cite{Rodino2}. 

\begin{defn}\label{def:Gaborwavefront}
Let $\varphi \in \cS(\rr d) \setminus 0$, $u \in \cS'(\rr d)$ and $z_0 \in T^* \rr d \setminus 0$. 
Then $z_0 \notin \WFg(u)$ if there exists an open conic set $\Gamma \subseteq T^* \rr d \setminus 0$ such that $z_0 \in \Gamma$ and 
\begin{equation}\label{eq:conedecay}
\sup_{z \in \Gamma} \eabs{z}^N | V_\varphi u(z)| < \infty, \quad N \geqs 0. 
\end{equation}
\end{defn}

This means that $V_\varphi u$ decays rapidly (super-polynomially) in $\Gamma$. 
The condition \eqref{eq:conedecay} is independent of $\varphi \in \cS(\rr d) \setminus 0$, in the sense that super-polynomial decay will hold also for $V_\psi u$ if $\psi \in \cS(\rr d) \setminus 0$, in a possibly smaller cone containing $z_0$. 
The Gabor wave front set is a closed conic subset of $T^*\rr d  \setminus 0$. 
By \cite[Proposition~2.2]{Hormander1} it is symplectically invariant in the sense of 
\begin{equation}\label{eq:metaplecticWFG}
\WFg( \mu(\chi) u) = \chi \WFg (u), \quad \chi \in \Sp(d, \ro), \quad u \in \cS'(\rr d).
\end{equation}

The Gabor wave front set is naturally connected to the definition of the $C^\infty$ wave front set \cite[Chapter~8]{Hormander0}, often called just the wave front set and denoted $\WF$. 
For $u \in \cD'(\rr d)$ a point in the phase space $(x_0,\xi_0) \in T^* \rr d$ such that $\xi_0 \neq 0$ satisfies $(x_0,\xi_0) \notin \WF(u)$ 
if there exists $\varphi \in C_c^\infty(\rr d)$ such that $\varphi(0) \neq 0$, 
an open conical set $\Gamma_2 \subseteq \rr d \setminus 0$ such that $\xi_0 \in \Gamma_2$, and 
\begin{equation*}
\sup_{\xi \in \Gamma_2} \eabs{\xi}^N | V_\varphi u(x_0,\xi)| < \infty, \quad N \geqs 0. 
\end{equation*}

The difference compared to $\WFg (u)$ is that the $C^\infty$ wave front set $\WF(u)$ is defined in terms of super-polynomial decay in the frequency variable, for $x_0 \in \rr d$ fixed, instead of super-polynomial decay in an open cone in the phase space $T^* \rr d$ containing the point of interest. 

Pseudodifferential operators with Shubin symbols are microlocal with respect to the Gabor wave front set. 
In fact we have by  \cite[Proposition~2.5]{Hormander1} 
\begin{equation*}
\WFg (a^w(x,D) u) \subseteq  \WFg (u)
\end{equation*}
provided $a \in \Gamma^m$ and $u \in \cS'(\rr d)$.

Let $u \in (\Sigma_t^s)' (\rr d)$ with $s + t > 1$. 
If $\psi \in \Sigma_t^s (\rr d) \setminus 0$ then 
\begin{equation}\label{eq:STFTGFstdistr}
| V_\psi u (x,\xi)| \lesssim e^{r (|x|^{\frac1t} + |\xi|^{\frac1s})}
\end{equation}
for some $r > 0$. We have $u \in \Sigma_t^s (\rr d)$ if and only if 
\begin{equation}\label{eq:STFTGFstfunc}
| V_\psi u (x,\xi)| \lesssim e^{-r (|x|^{\frac1t} + |\xi|^{\frac1s})}
\end{equation}
for all $r > 0$. See e.g. \cite[Theorems~2.4 and 2.5]{Toft1}.

For $u \in (\Sigma_t^s)' (\rr d)$ we define the $t,s$-Gelfand--Shilov wave front set $\WF^{t,s} (u)$ as a closed subset of the phase space $T^* \rr d \setminus 0$ as follows. 

\begin{defn}\label{def:wavefrontGFst}
Let $s,t > 0$ satisfy $s + t > 1$, and suppose $\psi \in \Sigma_t^s(\rr d) \setminus 0$ and $u \in (\Sigma_t^s)'(\rr d)$. 
Then $(x_0,\xi_0) \in T^* \rr d \setminus 0$ satisfies $(x_0,\xi_0) \notin \WF^{t,s} (u)$ if there exists an open set $U \subseteq T^*\rr d \setminus 0$ containing $(x_0,\xi_0)$ such that 
\begin{equation}\label{eq:notinWFGFst1}
\sup_{\lambda > 0, \ (x,\xi) \in U} e^{r \lambda} |V_\psi u(\lambda^t x, \lambda^s \xi)| < \infty, \quad \forall r > 0. 
\end{equation}
\end{defn}

Due to \eqref{eq:STFTGFstdistr} it is clear that it suffices to check \eqref{eq:notinWFGFst1} for $\lambda \geqs L$
where $L > 0$ can be arbitrarily large, for each $r > 0$. 

A consequence of Definition \ref{def:wavefrontGFst} is that we have the scaling invariance (here we assume 
$(x,\xi) \in T^* \rr d \setminus 0$)
\begin{equation}\label{eq:WFstscalinv}
(x,\xi) \in \WF^{t,s} (u) \quad \Longleftrightarrow \quad (\lambda^t x, \lambda^s \xi) \in \WF^{t,s} (u) \quad \forall \lambda > 0. 
\end{equation}

Another immediate consequence of Definition \ref{def:wavefrontGFst} is 
\begin{equation}\label{eq:WFstsublinear}
\WF^{t,s} (u + v) \subseteq \WF^{t,s} (u) \cup \WF^{t,s} (v), \quad u,v \in (\Sigma_t^s)'(\rr d).
\end{equation}

If $t = s > \frac12$ and $u \in \Sigma_s' (\rr d)$ then $\WF^{s,s} (u) = \WF^s(u)$, that is we recapture 
the $s$-Gelfand--Shilov wave front set $\WF^s (u)$ (which is a slightly modified version of Cappiello's and Schulz's \cite[Definition~2.1]{Cappiello1}),
as defined originally in \cite[Definition~4.1]{Carypis1}: 

\begin{defn}\label{def:wavefrontGFs}
Let $s > 1/2$, $\psi \in \Sigma_s(\rr d) \setminus 0$ and $u \in \Sigma_s'(\rr d)$. 
Then $z_0 \in T^*\rr d \setminus 0$ satisfies $z_0 \notin \WF^s (u)$ if there exists an open conic set $\Gamma_{z_0} \subseteq T^*\rr d \setminus 0$ containing $z_0$ such that 
\begin{equation*}
\sup_{z \in \Gamma_{z_0}} e^{r | z |^{\frac1s}} |V_\psi u(z)| < \infty, \quad \forall r > 0. 
\end{equation*}
\end{defn}

In Definition \ref{def:wavefrontGFst} we ask for exponential decay with arbitrary parameter $r > 0$ (super-exponential) of $V_\psi u$ along the curve $C_{x,\xi} \in T^* \rr d$
defined by $\ro_+ \ni  \lambda \to (\lambda^t x, \lambda^s \xi)$ which passes through $(x,\xi) \in T^* \rr d \setminus 0$. 
This power type curve reduces to a straight line if $t=s$.
By \eqref{eq:STFTGFstdistr} a generic point $(x,\xi) \in T^* \rr d \setminus 0$ has an exponential growth upper bound along the curve $C_{x,\xi}$. 
Due to \eqref{eq:STFTGFstfunc} we have $\WF^{t,s} (u) = \emptyset$ if and only if $u \in \Sigma_t^s(\rr d)$. 
Thus $\WF^{t,s} (u) \subseteq T^* \rr d \setminus 0$ can be seen as a measure of singularities of $u \in (\Sigma_t^s)'(\rr d)$: It records the phase space points $(x,\xi) \in T^* \rr d \setminus 0$ such that $V_\psi u$ does not decay super-exponentially along the curve $C_{x,\xi}$, 
that is, does not behave like an element in $\Sigma_t^s(\rr d)$ there. 

We will soon show that Definition \ref{def:wavefrontGFst} does not depend on the window function $\psi \in \Sigma_t^s(\rr d) \setminus 0$ (see Proposition \ref{prop:windowinvariance}). 
If $\check u(x) = u(-x)$ then 
\begin{equation}\label{eq:evensteven0}
V_{\check \psi} \check u(x,\xi)
= V_\psi u(-x,-\xi). 
\end{equation}
If $u$ is even or odd we thus have the following symmetry:
\begin{equation}\label{eq:evensteven1}
\check u = \pm u \quad \Longrightarrow \quad \WF^{t,s} (u) = - \WF^{t,s} (u). 
\end{equation}

We also have 
\begin{equation}\label{eq:evensteven2}
V_{\psi} \overline{u}(x,\xi)
= \overline{V_{\overline \psi} u(x,-\xi)}. 
\end{equation}

\begin{rem}\label{rem:WFinclusion}
Suppose $s_j,t_j > 0$, $j=1,2$, $s_1 + t_1 > 1$, 
and $t_2/t_1 = s_2/s_1 = a \geqs 1$. 
Then we have for $u \in (\Sigma_{t_2}^{s_2})'(\rr d) \subseteq (\Sigma_{t_1}^{s_1})'(\rr d)$
\begin{equation*}
\WF^{t_2,s_2} (u) \subseteq \WF^{t_1,s_1} (u). 
\end{equation*}
In fact this follows directly from 
Definition \ref{def:wavefrontGFst} with $\psi \in \Sigma_{t_1}^{s_1}(\rr d) \setminus 0$, and 
$\lambda^{t_2} = (\lambda^ a)^{t_1}$, 
$\lambda^{s_2} = (\lambda^ a)^{s_1}$, 
and $\lambda \leqs \lambda^a$ for $\lambda \geqs 1$. 
\end{rem}

\subsection{Examples of Gabor and $s$-Gelfand--Shilov wave front sets}

In this subsection we compile known and deduce a few new results on the $t,s$-Gelfand--Shilov wave front set. 

We have 
\begin{equation}\label{eq:GaborGSinclusion}
\WFg ( u ) \subseteq \WF^s(u), \quad \forall s > \frac12, \quad u \in \cS'(\rr d). 
\end{equation}

If $\frac12 < s_1 < s_2$ then 
\begin{equation}\label{eq:strictinclusion1}
\Sigma_{s_1}(\rr d) \subsetneq \Sigma_{s_2}(\rr d), 
\end{equation}
\begin{equation}\label{eq:strictinclusion2}
\Sigma_{s_2}'(\rr d) \subsetneq \Sigma_{s_1}'(\rr d) 
\end{equation}
and 
\begin{equation*}
\WF^{s_2}(u) \subseteq \WF^{s_1}(u), \quad u \in \Sigma_{s_2}'(\rr d). 
\end{equation*}
The strictness of the inclusions \eqref{eq:strictinclusion1} and \eqref{eq:strictinclusion2} 
can be seen for instance from the  Hilbert sequence space characterizations of $\Sigma_{s}(\rr d)$ and $\Sigma_{s}'(\rr d)$
for series expansions in Hermite functions (cf. e.g. \cite{Wahlberg3}). 

If $u \in \Sigma_{s_2}(\rr d) \setminus \Sigma_{s_1}(\rr d)$ then $u \in \Sigma_{s_2}'(\rr d)$ and
\begin{equation*}
\WF^{s_2}(u) = \emptyset \neq \WF^{s_1}(u). 
\end{equation*}
So given $s_2 > s_1 > \frac12$ there exists $u \in \Sigma_{s_2}'(\rr d)$ such that $\WF^{s_2}(u) \neq \WF^{s_1}(u)$. 
This gives some motivation for the interest of the scale of wave front sets $\WF^{s}(u)$ for $s > \frac12$. 
In the given example it is a measure of very fine singularities within $\cS$. 

If on the other hand $u \in \Sigma_{s_1}'(\rr d) \setminus \Sigma_{s_2}'(\rr d)$ then 
$\WF^{s_1}(u)$ is well defined, and $\WF^{s_1}(u) \neq \emptyset$ since $u \in \Sigma_{s_1}$ would imply
$u \in \Sigma_{s_2}'$. But $\WF^{s_2}(u)$ is not well defined so we cannot compare $\WF^{s_1}(u)$ and $\WF^{s_2}(u)$. 

It is also clear that if $u \in \cS(\rr d) \setminus \Sigma_s(\rr d)$ for some $s > \frac12$
then $u \in \Sigma_s'(\rr d)$ and 
\begin{equation*}
\emptyset = \WFg (u) \neq  \WF^s(u). 
\end{equation*}

Nevertheless it seems that for most ultradistributions $u$ for which $\WF^s(u)$ can be determined we have 
\begin{equation*}
\WFg (u) = \WF^s(u) \quad \mbox{for all} \quad s > \frac12
\end{equation*}
(cf. \cite{Carypis1,PRW1}). 
We collect a few examples. 
For any $x \in \rr d$ we have
\begin{equation}\label{eq:diracWF}
\WFg (\delta_x) = \WF^s(\delta_x) = \{ 0 \} \times (\rr d \setminus 0) \quad \forall s > \frac12. 
\end{equation}
For any $\xi \in \rr d$ we have 
\begin{equation}\label{eq:planewaveWF}
\WFg ( e^{i \la \cdot, \xi \ra} ) = \WF^s( e^{i \la \cdot, \xi \ra} ) =  (\rr d \setminus 0) \times \{ 0 \} \quad \forall s > \frac12. 
\end{equation}
For any $A \in \rr {d \times d}$ symmetric we have 
\begin{equation}\label{eq:chirpWF1}
\WFg ( e^{i \la x, A x \ra/2} ) = \WF^s( e^{i \la x, A x \ra/2} ) =  \{ (x,Ax): \ x \in \rr d \setminus 0 \} \quad \forall s > \frac12. 
\end{equation}

The latter formula can be generalized, by combining \cite[Example~7.1]{PRW1} (generalized to the Gelfand--Shilov framework) and \cite[Corollary~9.2]{Carypis1}. 
This gives the following formula when $A \in \cc {d \times d}$ is symmetric and $\im A \geqs 0$: 
\begin{equation}\label{eq:chirpWF2}
\WFg ( e^{i \la x, A x \ra/2} ) = \WF^s( e^{i \la x, A x \ra/2} ) =  \{ (x, \re A \, x): \ x \in \rr d \cap \Ker (\im A) \setminus 0 \} \quad \forall s > \frac12. 
\end{equation}

If $d = 1$ then (cf. \cite[Section~8]{PRW1}) for $k \geqs 2$ we have
\begin{equation}\label{eq:evenosc1}
\WFg ( e^{i x^{2k}} ) = \{ 0 \} \times (\ro \setminus 0)
\end{equation}
and for $k \geqs 1$ we have 
\begin{equation}\label{eq:oddosc1}
\WFg ( e^{i x^{2k+1}} ) = \{ 0 \} \times \ro_+. 
\end{equation}
It also follows from the proof of \cite[Proposition~8.2]{PRW1} that 
\begin{equation}\label{eq:modulusosc1}
\WFg ( e^{i |x|^p} ) = \{ 0 \} \times (\ro \setminus 0)
\end{equation}
if $p > 2$. 

We obtain from  \eqref{eq:GaborGSinclusion} if $k \geqs 2$
\begin{equation*}
\{ 0 \} \times (\ro \setminus 0) 
\subseteq \WF^s( e^{i x^{2k}} )
\quad \forall s > \frac12
\end{equation*}
In the same way we obtain from \eqref{eq:oddosc1} if $k \geqs 1$
\begin{equation*}
\{ 0 \} \times \ro_+  
\subseteq \WF^s( e^{i x^{2k+1}} ) \quad \forall s > \frac12
\end{equation*}
and from \eqref{eq:modulusosc1}
\begin{equation*}
\{ 0 \} \times (\ro \setminus 0)  
\subseteq \WF^s( e^{i |x|^p} ) \quad \forall s > \frac12
\end{equation*}
if $p > 2$.

In \cite[Theorem~6.1]{Schulz1} we prove that given any closed conic set $\Gamma \subseteq T^*\rr d \setminus 0$
there exists $u \in \cS'(\rr d)$ such that $\WFg (u) = \Gamma$. 
By a careful examination of the proof it follows that $\WFg (u) = \WF^s(u) = \Gamma$ for all $s > \frac12$. 

A similar result is given in \cite[Proposition~3.5]{Cappiello1} for a wave front set that is similar to $\WF^s(u)$ albeit with the Roumieu choice of behaviour instead of Beurling.

\subsection{Invariances of the $t,s$-Gelfand--Shilov wave front set}

In \cite[Proposition~4.3]{Carypis1} it is shown that $\WF^s (u)$ does not depend on the chosen window function $\psi \in \Sigma_s(\rr d) \setminus 0$. 
The following result generalizes this statement to $\WF^{s,t} (u)$ with $t \neq s$. 

\begin{prop}\label{prop:windowinvariance}
Let $s,t > 0$ satisfy $s + t > 1$, 
and let 
$u \in (\Sigma_t^s)'(\rr d)$. 
Suppose $z_0 \in T^* \rr d \setminus 0$. 
If $\psi \in \Sigma_t^s(\rr d) \setminus 0$ and \eqref{eq:notinWFGFst1} holds for an open set $U \subseteq T^*\rr d \setminus 0$ containing $z_0$, 
and $\fy \in \Sigma_t^s(\rr d) \setminus 0$ then there exists an open set $V \subseteq U$ such that $z_0 \in V$ and 
\begin{equation}\label{eq:notinWFGFst2}
\sup_{\lambda > 0, \ (x,\xi) \in V} e^{r \lambda } |V_\fy u(\lambda^t x, \lambda^s \xi)| < \infty, \quad \forall r > 0. 
\end{equation}
\end{prop}

\begin{proof}
Since $z_0 \in U \subseteq \rr {2d}$ where $U$ is open we may pick an open set $V \subseteq U$ 
such that $z_0 \in V$ and $V + \rB_\ep \subseteq U$ for some $0 < \ep \leqs 1$, and we may assume 
\begin{equation}\label{eq:Vbounds}
\sup_{z \in V} |z| \leqs |z_0| + 1 := \mu. 
\end{equation}

By \eqref{eq:STFTGFstdistr} we have 
\begin{equation}\label{eq:STFT1}
| V_\fy u (x,\xi)| \lesssim e^{r_1 (|x|^{\frac1t} + |\xi|^{\frac1s})}
\end{equation}
for some $r_1 > 0$. 
By \cite[Lemma~11.3.3]{Grochenig1} we have 
\begin{equation*}
|V_\fy u (z)| \leqs (2 \pi)^{-\frac{d}{2}} \| \psi \|_{L^2}^{-2} \,  |V_\psi u| * |V_\fy \psi | (z), \quad z \in \rr {2d}, 
\end{equation*}
and according to \eqref{eq:STFTGFstfunc} we have 
\begin{equation}\label{eq:STFTGFstwindow1}
| V_\fy \psi (x,\xi)| \lesssim e^{-r_2 (|x|^{\frac1t}+ |\xi|^{\frac1s})}
\end{equation}
for any $r_2 > 0$. 

Let $r > 0$ and $\lambda > 0$. We have
\begin{align*}
& e^{r \lambda} |V_\fy u (\lambda^t x, \lambda^s \xi)| \\
& \lesssim \iint_{\rr {2d}} e^{r \lambda } |V_\psi u ( \lambda^t (x- \lambda^{-t} y), \lambda^s (\xi - \lambda^{-s} \eta))| \ |V_\fy \psi (y,\eta) | \, \dd y \, \dd \eta \\
& = I_1 + I_2
\end{align*}
where we split the integral into the two terms 
\begin{align*}
I_1 = & \iint_{\rr {2d} \setminus \Omega_\lambda} e^{r \lambda} |V_\psi u ( \lambda^t (x- \lambda^{-t} y), \lambda^s (\xi - \lambda^{-s} \eta))| \ |V_\fy \psi (y,\eta) | \, \dd y \, \dd \eta, \\
I_2 = & \iint_{\Omega_\lambda} e^{r \lambda} |V_\psi u ( \lambda^t (x- \lambda^{-t} y), \lambda^s (\xi - \lambda^{-s} \eta))| \ |V_\fy \psi (y,\eta) | \, \dd y \, \dd \eta
\end{align*}
where  
\begin{equation*}
\Omega_\lambda = \{(y,\eta) \in \rr {2d}: |y|^{\frac1t} + |\eta|^{\frac1s} < 2^{-\frac{1}{2v}} \ep^{\frac1v} \lambda \} 
\subseteq \rr {2d}
\end{equation*}
with $v = \min(s,t)$. 

First we estimate $I_1$ when $(x,\xi) \in V$. 
Set $\kappa = \max(\kappa(t^{-1}), \kappa(s^{-1}))$. 
From \eqref{eq:Vbounds}, \eqref{eq:STFT1} and \eqref{eq:STFTGFstwindow1} we obtain for some $r_1 > 0$ and any $r_2 > 0$
\begin{equation}\label{eq:estimateI1a}
\begin{aligned}
I_1 
& \lesssim e^{r \lambda } \iint_{\rr {2d} \setminus \Omega_\lambda} e^{r_1 \lambda |x- \lambda^{-t} y|^{\frac1t} + r_1 \lambda |\xi- \lambda^{-s} \eta|^{\frac1s}}  \ |V_\fy \psi (y,\eta) | \, \dd y \, \dd \eta \\
& \leqs e^{ r \lambda+ \kappa r_1 \lambda |x|^{\frac1t} + \kappa r_1 \lambda |\xi|^{\frac1s} } \iint_{\rr {2d} \setminus \Omega_\lambda} e^{ r_1 \kappa (| y|^{\frac1t} +  | \eta|^{\frac1s}) }  \ |V_\fy \psi (y,\eta) | \, \dd y \, \dd \eta \\
& \lesssim e^{ \lambda \left( r + 2 r_1 \kappa \mu^{\frac1v} \right) } \iint_{\rr {2d} \setminus \Omega_\lambda} e^{ (\kappa r_1 - \kappa r_1 - 1 -  r_2) (|y|^{\frac1t} + |\eta|^{\frac1s})}  \, \dd y \, \dd \eta \\
& \leqs e^{ \lambda \left( r + 2 r_1 \kappa \mu^{\frac1v} \right)  - \lambda \, r_2 2^{-\frac{1}{2v}}\ep^{\frac1v} } \iint_{\rr {2d}} e^{ - (|y|^{\frac1t} + |\eta|^{\frac1s})}  \, \dd y \, \dd \eta \\
& \lesssim e^{ \lambda \left( r + 2 r_1 \kappa \mu^{\frac1v}  -  r_2 2^{-\frac{1}{2v}} \ep^{\frac1v } \right) }\leqs C_{r}
\end{aligned}
\end{equation}
for any $\lambda > 0$, 
provided we pick $r_2 \geqs 2^{\frac{1}{2v}} \ep^{- \frac1v} \left( r + 2 r_1 \kappa \mu^{\frac1v}\right)$. 
Here $C_{r} > 0$ is a constant that depends on $r > 0$ but not on $\lambda > 0$. 
Thus we have obtained the requested estimate for $I_1$. 

It remains to estimate $I_2$. 
From $|y|^{\frac1t} + |\eta|^{\frac1s} < 2^{-\frac{1}{2v}} \ep^{\frac1v} \lambda$ 
we obtain
\begin{align*}
& \lambda^{-t} |y| < \ep^{\frac{t}{v}} \, 2^{-\frac{t}{2v}} \leqs \ep \, 2^{-\frac{1}{2}}, \\
& \lambda^{-s} |\eta| < \ep^{\frac{s}{v}} \, 2^{-\frac{s}{2v}} \leqs \ep \, 2^{-\frac{1}{2}}
\end{align*}
which gives $(\lambda^{-t} y, \lambda^{-s} \eta ) \in \rB_{\ep}$.  
Hence if $(x,\xi) \in V$ then $( x- \lambda^{-t} y, \xi - \lambda^{-s} \eta) \in U$ and we may use the estimate 
\eqref{eq:notinWFGFst1}. 
This gives for a constant $C_{r} > 0$, using \eqref{eq:STFTGFstwindow1} 
\begin{equation}\label{eq:estimateI2a}
\begin{aligned}
I_2 = & \iint_{\Omega_\lambda} e^{r \lambda} |V_\psi u ( \lambda^t (x- \lambda^{-t} y), \lambda^s (\xi - \lambda^{-s} \eta))| \ |V_\fy \psi (y,\eta) | \, \dd y \, \dd \eta \\
& \leqs C_{r} \iint_{\rr {2d}} |V_\fy \psi (y,\eta) | \, \dd y \, \dd \eta \\
& \leqs C_{r}' 
\end{aligned}
\end{equation}
for all $\lambda > 0$, 
Thus we have obtained the requested estimate for $I_2$. 
The statement follows from \eqref{eq:estimateI1a} and \eqref{eq:estimateI2a}. 
\end{proof}

\subsection{Metaplectic properties}

The $s$-Gelfand--Shilov wave front set is symplectically invariant as (cf. \cite[Corollary~4.5]{Carypis1})
\begin{equation}\label{eq:metaplecticWFs}
\WF^s( \mu(\chi) u) = \chi \WF^s(u), \quad \chi \in \Sp(d, \ro), \quad u \in \Sigma_s'(\rr d), \quad s > \frac12.
\end{equation}

When $t \neq s$ the $t,s$-Gelfand--Shilov wave front set $\WF^{t,s} (u)$ is not symplectically invariant. 
Nevertheless, two of the generators of the symplectic group behave invariantly in certain individual senses which we now describe. 
By \cite[Proposition~4.10]{Folland1} each matrix $\chi \in \Sp(d,\ro)$ is a finite product of matrices in $\Sp(d,\ro)$ of the form
\begin{equation*}
\J, \quad 
\left(
  \begin{array}{cc}
  A^{-1} & 0 \\
  0 & A^{T}
  \end{array}
\right), 
\quad
\left(
  \begin{array}{cc}
  I & 0 \\
  B & I
  \end{array}
\right), 
\end{equation*}
for $A \in \GL(d,\ro)$ and $B \in \rr {d \times d}$ symmetric. 
The corresponding metaplectic operators are 
$\mu(\J) = \cF$, 
\begin{equation*}
\mu \left(
  \begin{array}{cc}
  A^{-1} & 0 \\
  0 & A^{T}
  \end{array}
\right) f(x)
= |A|^{\frac12} f(Ax), 
\end{equation*}
if $A \in \GL(d,\ro)$, and 
\begin{equation*}
\mu 
\left(
  \begin{array}{cc}
  I & 0 \\
  B & I
  \end{array}
\right) f(x)
= e^{\frac{i}{2} \la B x, x \ra} f(x), 
\end{equation*}
if $B \in \rr {d \times d}$ is symmetric. 

\begin{prop}\label{prop:WFstsymplectic}
Let $s,t > 0$ satisfy $s + t > 1$, and suppose 
$u \in (\Sigma_t^s)'(\rr d)$. Then we have 

\begin{enumerate}[\rm(i)]

\item 
\begin{equation*}
\WF^{s,t} (\wh u) = \J \WF^{t,s} (u);  
\end{equation*}

\item if $A \in \GL(d,\ro)$ and $u_A (x) = |A|^{\frac12} u(Ax)$ then 
\begin{equation*}
\WF^{t,s} (u_A) = 
\left(
  \begin{array}{cc}
  A^{-1} & 0 \\
  0 & A^{T}
  \end{array}
\right) 
\WF^{t,s} (u). 
\end{equation*}

\end{enumerate}

\end{prop}

\begin{proof}
Let $\psi \in \Sigma_t^s(\rr d) \setminus 0$.  
We have from the proof of \cite[Corollary~4.5]{Carypis1}
\begin{equation}\label{eq:STFTmetaplectic}
|V_{\mu (\chi) \psi} (\mu(\chi) u)( \chi(x,\xi))| = |V_{\psi} u( x, \xi)|
\end{equation}
for all $\chi \in \Sp(d, \ro)$. If $\chi = \J$ we obtain
\begin{equation*}
|V_{\wh \psi} \wh u(\J (x,\xi))| = |V_{\wh \psi} \wh u(\xi, - x)| = |V_{\psi} u( x,\xi )|. 
\end{equation*}
Note that $\wh \psi \in \Sigma_s^t(\rr d) \setminus 0$ and $\wh u \in (\Sigma_s^t)'(\rr d)$.
From this it follows that $(x,\xi) \notin \WF^{t,s} (u)$ if and only if $\J(x,\xi) \notin \WF^{s,t} (\wh u)$ which proves claim (i). 

Next we insert $u_A$ for $A \in \GL(d,\ro)$ into \eqref{eq:STFTmetaplectic} which gives
\begin{equation*}
|V_{\psi_A} u_A ( A^{-1}x, A^T \xi )| = |V_\psi u( x,\xi) |. 
\end{equation*}
Note that $\psi_A \in \Sigma_t^s(\rr d) \setminus 0$ and $u_A \in (\Sigma_t^s)'(\rr d)$.
We obtain $(x,\xi) \notin \WF^{t,s} (u)$ if and only if $(A^{-1}x, A^T \xi) \notin \WF^{t,s} (u_A)$ which shows claim (ii). 
\end{proof}

\begin{rem}\label{rem:propagator}
Proposition \ref{prop:WFstsymplectic} implies that $\WF^{v,s}$ when $s \neq v$ does not behave as $\WF^{s}$ with respect to Schr\"odinger type propagators, in the case of quadratic potential. 
In fact let $Q \in \rr {2d \times 2d}$ be symmetric, let 
\begin{equation*}
q(x,\xi) = \la (x,\xi), Q (x,\xi) \ra, \quad x, \ \xi \in \rr d, 
\end{equation*}
and consider the initial value Cauchy problem
\begin{equation}\label{eq:schrodeq}
\left\{
\begin{array}{rl}
\partial_t u(t,x) + i q^w(x,D_x) u (t,x) & = 0, \\
u(0,\cdot) & = u_0, 
\end{array}
\right.
\end{equation}
where $q^w(x,D_x)$ acts on the $x \in \rr d$ variable. 
If $u_0 \in D (q^w(x,D)) \subseteq L^2(\rr d)$, the domain of the closure of $q^w(x,D)$ considered as an unbounded operator in $L^2(\rr d)$, the equation is solved by 
\begin{equation*}
u(t,x) = e^{- i t q^w(x,D)} u_0
\end{equation*}
where $e^{- i t q^w(x,D)}$ is the propagator one-parameter group of unitary operators indexed by $t \in \ro$ (cf. e.g.\cite{Carypis1,Hormander2}). 
The propagator is the metaplectic operator $e^{- i t q^w(x,D)} = \mu( e^{2 t \J Q})$ \cite{Folland1}, 
which extends to a continuous operator on $\Sigma_s'(\rr d)$ for $s > \frac12$ and the equation \eqref{eq:schrodeq} admits initial datum $u_0 \in \Sigma_s'(\rr d)$ \cite{Carypis1,Wahlberg3}. 

By the metaplectic invariance \eqref{eq:metaplecticWFs} we thus have the propagation of singularities equality
\begin{equation}\label{eq:propschrodWFs}
\WF^s ( e^{- i t q^w(x,D)} u_0) = e^{2 t \J Q} \WF^s(u_0), \quad t \in \ro, \quad u_0 \in \Sigma_s'(\rr d), \quad s > \frac12. 
\end{equation}

If $Q = I_{2d}$ then 
\begin{equation*}
e^{2 t \J Q} =
\left(
\begin{array}{ll}
\cos 2t & \sin 2t \\
- \sin 2t & \cos 2t
\end{array}
\right)
\end{equation*}
so
\begin{equation}\label{eq:WFsFourier}
\WF^s ( e^{- i \frac{\pi}{4} q^w(x,D)} u_0)
= \WF^s ( \wh u_0) 
= \J \WF^s(u_0). 
\end{equation}

If $s \neq v$ then the equality \eqref{eq:propschrodWFs} cannot hold for $\WF^{v,s}$, 
since \eqref{eq:WFsFourier} for $\WF^{v,s} (u)$ would contradict Proposition \ref{prop:WFstsymplectic} (i).  
\end{rem}

The next result reveals 
that if $\WF^t (u)$ has empty intersection with the frequency axis $\{ 0 \} \times (\rr d \setminus 0)$
then $\WF^{t,s}$ is contained in the space axis $(\rr d \setminus 0) \times \{ 0 \}$ if $s > t$. 

\begin{prop}\label{prop:WFsvsWFst1}
If $s > t > \frac12$, $u \in (\Sigma_t^s)'(\rr d)$ and
\begin{equation}\label{eq:WFsfreqaxis}
\WF^t (u) \cap \{ 0 \} \times (\rr d \setminus 0) = \emptyset
\end{equation}
then 
\begin{equation}\label{eq:WFstconclusion1}
\WF^{t,s} (u) \subseteq (\rr d \setminus 0) \times \{ 0 \}. 
\end{equation}
\end{prop}

\begin{proof}
We have $(\Sigma_t^s)'(\rr d) \subseteq \Sigma_t'(\rr d)$ since $\Sigma_t(\rr d) \subseteq \Sigma_t^s(\rr d)$. 
By the assumption \eqref{eq:WFsfreqaxis} there exists $C > 0$ such that
for the open conic set
\begin{equation*}
\Gamma = \{ (x,\xi) \in T^* \rr d \setminus 0: |\xi| > C |x| \} \subseteq T^* \rr d
\end{equation*}
we have
\begin{equation*}
\sup_{z \in \Gamma} e^{r | z |^{\frac1t}} |V_\psi u(z)| < \infty \quad \forall r > 0
\end{equation*}
where $\psi \in \Sigma_t(\rr d) \setminus 0 \subseteq \Sigma_t^s (\rr d) \setminus 0$.

Let $(x_0,\xi_0) \in T^* \rr d \setminus 0$ where $\xi_0 \neq 0$. 
If $x_0 = 0$ we pick $U \subseteq \Gamma$ as an open set containing $(0,\xi_0)$.
Then if $(x,\xi) \in U$ we have $(\lambda^t x, \lambda^s \xi ) \in \Gamma$ for $\lambda \geqs 1$, 
since $|\xi| > C |x|$ implies $\lambda^{s-t} |\xi| > C |x|$. 
If instead $x_0 \neq 0$ then we pick as $U \subseteq \rr {2d}$ an open set containing $(x_0,\xi_0)$ such that $\ep < |x| < 2 |(x_0,\xi_0)|$
and $\ep < |\xi| < 2 |(x_0,\xi_0)|$ when $(x,\xi) \in U$ where $\ep > 0$. 
If $(x,\xi) \in U$ then 
\begin{equation*}
C |x| |\xi|^{-1}
< 2 | (x_0,\xi_0) | C \ep^{-1} \leqs \lambda^{s-t} 
\end{equation*}
if $\lambda \geqs L > 0$ provided $L$ is sufficiently large. 
This gives $(\lambda^t x, \lambda^s \xi ) \in \Gamma$ for $\lambda \geqs L$. 

If necessary we increase $L > 0$ such that $|(x, \lambda^{s-t} \xi)| \geqs 1$ when $\lambda \geqs L$ and $(x,\xi) \in U$.
This gives for any $r > 0$
\begin{align*}
\sup_{\lambda \geqs L, \ (x,\xi) \in U} e^{r \lambda } |V_\psi u(\lambda^t x, \lambda^s \xi)|
& \leqs \sup_{\lambda \geqs L, \ (x,\xi) \in U}  e^{r \lambda |(x, \lambda^{s-t} \xi)|^{\frac1t}} |V_\psi u(\lambda^t x, \lambda^s \xi)| \\
& \leqs \sup_{\lambda \geqs L, \ (x,\xi) \in U}  e^{r |(\lambda^t x, \lambda^s \xi)|^{\frac1t}} |V_\psi u(\lambda^t x, \lambda^s \xi)| \\
& \leqs \sup_{z \in \Gamma} e^{r | z |^{\frac1t}} |V_\psi u(z)| < \infty. 
\end{align*}
We have shown $(x_0, \xi_0) \notin \WF^{t,s} (u)$ which proves \eqref{eq:WFstconclusion1}. 
\end{proof}

\section{The $t,s$-Gelfand--Shilov wave front set of oscillatory functions}\label{sec:chirp}

A main reason for the introduction of the wave front set $\WF^{t,s} (u)$ is that it describes accurately the phase space singularities of oscillatory functions of the form 
\begin{equation}\label{eq:chirpdef1}
u(x) = e^{i c x^m}, \quad x \in \ro, \quad m \in \no \setminus \{0, 1\} 
\end{equation}
or
\begin{equation}\label{eq:chirpdef2}
u(x) = e^{i c |x|^\alpha}, \quad x \in \ro, \quad \alpha \in \ro \setminus 2 \no, \quad \alpha > 1
\end{equation}
where 
$c \in \ro \setminus 0$ in both cases. 
These functions are known as chirp signals. 
Here we work in dimension $d = 1$. 
In \eqref{eq:chirpdef2} we ask $\alpha \notin 2 \no$ since $\alpha \in 2 \no$ is covered by \eqref{eq:chirpdef1}.

If $u$ is defined by \eqref{eq:chirpdef1}, 
and $s$ is chosen adapted to $t$ and $m$, 
we will see that $\WF^{t,s} (u)$ is the curve in phase space described by the instantaneuos frequency of $u$, that is the 
derivative of the phase function. 
 
We will need a lemma. 

\begin{lem}\label{lem:chirplemma}
Suppose $s, t, \ep > 0$, $U \subseteq \rr {2d} \setminus 0$ is open
and $f \in C^\infty(\rr {2d})$. 
If the estimate
\begin{equation*}
\sup_{(x,\xi) \in U} \lambda^{s k} \ep^{2k} | f (\lambda^t x, \lambda^s \xi) | 
\leqs C_h \lambda^t h^k k!^s
\end{equation*}
holds for all $h > 0$, all $\lambda \geqs 1$ and all $k \in \no$, then 
for any $r > 0$ and any $\lambda \geqs 1$ we have
\begin{equation*}
\sup_{(x,\xi) \in U}  e^{r \lambda} \left| f (\lambda^t x,\lambda^s\xi) \right| \leqs C_{r,\ep,t}. 
\end{equation*}
\end{lem}

\begin{proof}
Let $r > 0$. We have if $(x,\xi) \in U$
\begin{align*}
e^{\frac{r \lambda}{s} \ep^{\frac{2}{s}}} \left| f (\lambda^t x,\lambda^s \xi) \right|^{\frac{1}{s}}
& = 
\sum_{k=0}^{\infty} 2^{-k} k!^{-1} \left( \frac{2 r}{s} ( \lambda^s \ep^2 )^{\frac{1}{s}} \right)^k
\left| f (\lambda^t x,\lambda^s \xi) \right|^{\frac{1}{s}} \\
& \leqs 
2 \left( \sup_{k \geqs 0}  k!^{-s} \left( \left( \frac{2 r}{ s} \right)^s \lambda^{s} \ep^2 \right)^k  \left| f ( \lambda^t x,\lambda^s \xi ) \right| \right)^{\frac{1}{s}} \\
& \leqs 2 \, C_h^{\frac1s} \lambda^{\frac{t}{s}} \sup_{k \geqs 0} \left( \left( \frac{2r}{s} \right)^{s} h \right)^{\frac{k}{s}} \\
& \leqs C_r^{\frac1s} \lambda^{\frac{t}{s}}
\end{align*}
for all $\lambda \geqs 1$,
provided we pick 
\begin{equation*}
0 < h \leqs \left( \frac{s}{2r} \right)^{s}. 
\end{equation*}

Thus 
for any $r > 0$, $(x,\xi) \in U$ and $\lambda \geqs 1$
\begin{equation*}
e^{r \lambda \ep^{\frac{2}{s}}} \left| f ( \lambda^t x , \lambda^s \xi) \right|
\leqs C_r \lambda^t
\end{equation*}
which gives finally 
\begin{equation}\label{eq:WFstnonmembership1q}
\begin{aligned}
e^{r \lambda} \left| f (\lambda^t x,\lambda^s\xi) \right|
& = e^{- r \lambda } e^{ 2 r \ep^{-\frac2s} \lambda \ep^{\frac2s}} \left| f (\lambda^t x,\lambda^s \xi) \right| \\
& \leqs C_{r,\ep} \lambda^t e^{- r \lambda } \\
& \leqs C_{r,\ep,t}
\end{aligned}
\end{equation}
for all $\lambda \geqs 1$ and $(x,\xi) \in U$. 
\end{proof}

The next result generalizes \eqref{eq:chirpWF1} for $d=1$.

\begin{thm}\label{thm:chirpWFst}
Suppose 
$c \in \ro \setminus 0$. 

\begin{enumerate}[\rm (i)]

\item If $u$ is defined by \eqref{eq:chirpdef1} and $t > \frac1{m-1}$ then 
\begin{equation}\label{eq:conclusion1a}
\WF^{t,t(m-1)} (u) 
= \{ (x, c m x^{m-1} ) \in \rr 2: \ x \neq 0 \}. 
\end{equation}

\item If $u$ is defined by \eqref{eq:chirpdef2} and $t > \frac1{\alpha-1}$ then 
\begin{equation}\label{eq:conclusion2}
\{ 0 \} \times (\ro \setminus 0) \subseteq 
\WF^{t,t(\alpha-1)} (u) 
\subseteq \{ (x, c \alpha \sgn(x) |x|^{\alpha-1} ) \in \rr 2: \ x \neq 0 \}
\cup \{ 0 \} \times (\ro \setminus 0) . 
\end{equation}

\end{enumerate}

\end{thm}

\begin{proof}
Case (i): 
Set $s = t(m-1) > 1$. 
This implies that there are compactly supported Gevrey functions \cite{Rodino1} of order $s$ in the space $\Sigma_t^s(\ro)$. 
Set 
\begin{align*}
W 
& = \{ (x, c m x^{m-1} ) \in \rr 2: \ x \neq 0 \} \subseteq \rr 2 \setminus 0.
\end{align*}

Suppose $(x_0, \xi_0) \in \rr 2 \setminus 0$ and $(x_0, \xi_0) \notin W$. 
Then there exists an open set $U$ such that $(x_0,\xi_0) \in U$, and $0 < \ep \leqs 1$, $\delta > 0$,
such that
\begin{align*}
(x,\xi) \in U,  \quad |x-y| \leqs \delta & \quad \Longrightarrow \quad  |\xi - c m x^{m-1} | \geqs 2 \ep, \quad
m \, |c| \, | x^{m-1} - y^{m-1}| \leqs \ep.
\end{align*}
Then if $(x,\xi) \in U$ and $|x-y| \leqs \delta$
\begin{equation}\label{eq:lowerboundfas0}
|\xi - c m y^{m-1}| 
\geqs |\xi - c m x^{m-1}| - m \, |c| \, |  y^{m-1} - x^{m-1} \,| 
\geqs \ep.
\end{equation}

Let $\psi \in \Sigma_t^s(\ro) \setminus 0$ be such that $\supp \psi \subseteq \rB_\delta$. 
From the stationary phase theorem \cite[Theorem~7.7.1]{Hormander0}
this gives for any $k \in \no$, any $h > 0$ and any $\lambda \geqs 1$, if $(x,\xi) \in U$, using \eqref{eq:lowerboundfas0} and \eqref{eq:expestimate0},
\begin{equation}\label{eq:estimateWF0}
\begin{aligned}
|V_\psi u ( \lambda^t x, \lambda^s \xi)|
& = (2 \pi)^{-\frac12} \left| \int_{\ro} e^{i (c y^{m} - y \lambda^s \xi )} \overline{\psi( \lambda^t (\lambda^{-t} y-x) )} \, \dd y \right| \\
& = (2 \pi)^{-\frac12} \lambda^t \left| \int_{\ro} e^{i \lambda^{m t} (c y^{m} - y \xi ) )} \overline{\psi( \lambda^t (y-x) )} \, \dd y \right| \\
& \leqs C \lambda^t \sum_{n = 0}^k \lambda^{n t} \sup_{|x-y| \leqs \delta} |(D^n\psi)( \lambda^t (y-x) )| \, |\xi - c m y^{m-1}|^{n - 2k} \lambda^{m t (n-2k)} \\
& \leqs C \lambda^t \ep^{- 2 k} \sum_{n = 0}^k \sup_{|x-y| \leqs \delta} |(D^n\psi)( \lambda^t (y-x) )| \lambda^{- t k (m-1)} \lambda^{t(1+m) (n-k)} \\
& \leqs C \lambda^{t} \ep^{- 2 k}  \lambda^{- s k}  \sum_{n = 0}^k \sup_{|x-y| \leqs \delta} |(D^n\psi)( \lambda^t (y-x) )| \\
& \leqs C_h \lambda^{t} \ep^{- 2 k} \lambda^{- s k} \sum_{n = 0}^k   h^n n!^s  \\
& = C_h \lambda^{t} \ep^{- 2 k} \lambda^{- s k} h^k \sum_{n = 0}^k   
h^{-(k-n)} n!^s \\
& \leqs C_h \lambda^{t} \ep^{- 2 k} \lambda^{- s k} h^k e^{s h^{-\frac1s}} \sum_{n = 0}^k  (n! (k-n)!)^s \\
& \leqs C_{s,h}  \lambda^{t} \ep^{- 2 k} \lambda^{- s k} h^k  k!^s  \sum_{n = 0}^k   \\
& \leqs C_{s,h}  \lambda^{t} \ep^{- 2 k} \lambda^{- s k} (2 h)^k  k!^s . 
\end{aligned}
\end{equation}

Since $h > 0$ is arbitrary
we obtain
\begin{equation}\label{eq:STFTestimate0}
\lambda^{s k} \ep^{2k} | V_\psi u (\lambda^t x, \lambda^s \xi) | 
\leqs C_h \lambda^t h^k k!^s, \quad (x,\xi) \in U, 
\end{equation}
for all $h > 0$, all $\lambda \geqs 1$ and all $k \in \no$. 
Applying Lemma \ref{lem:chirplemma}
it follows that   
\begin{equation*}
(x_0, \xi_0) \notin \WF^{t,t(m-1)} (u)
\end{equation*}
and we may conclude 
\begin{equation}\label{eq:chirpinclusion1}
\WF^{t,t(m-1)} (u) \subseteq W.
\end{equation}
In order to prove \eqref{eq:conclusion1a} for Case (i) it hence remains to strengthen the above inclusion into an equality.

If $m$ is even then $u$ is even and $W = -W$, so by \eqref{eq:evensteven1} we have either $\WF^{t,t(m-1)} (u) =\emptyset$ or $\WF^{t,t(m-1)} (u) = W$. 
The former is not true since $u \notin \Sigma_t^s(\ro)$. Thus we have proved \eqref{eq:conclusion1a} for Case (i) and $m$ even. 

If $m$ is odd then $\check u (x) = \overline{ u(x) } = e^{- i  c x^m}$. 
Again $\WF^{t,t(m-1)} (u) =\emptyset$ cannot hold since $u \notin \Sigma_t^s(\ro)$. 
If we assume that the inclusion \eqref{eq:chirpinclusion1} is strict we  get a contradiction from 
\eqref{eq:evensteven0} and \eqref{eq:evensteven2}. 
Indeed suppose e.g. 
\begin{equation*}
\WF^{t,t(m-1)} (u) 
= \{ (x, c m x^{m-1} ) \in \rr 2: \ x > 0 \}. 
\end{equation*}
By \eqref{eq:evensteven0} and \eqref{eq:evensteven2} we then get the contradiction 
\begin{align*}
\WF^{t,t(m-1)} ( \check u) 
& = \{ (x, - c m x^{m-1} ) \in \rr 2: \ x < 0 \} \\
& = \{ (x, - c m x^{m-1} ) \in \rr 2: \ x > 0 \}
= \WF^{t,t(m-1)} ( \overline{u} ).   
\end{align*}
This proves \eqref{eq:conclusion1a} for Case (i) when $m$ is odd.

Case (ii):
In this case $u(x) = e^{i c |x|^\alpha}$ is not smooth at $x=0$ which causes some problems. 
Set again $s = t(\alpha-1) > 1$, and
\begin{align*}
W = \{ (x, c \alpha \sgn(x) |x|^{\alpha-1} ) \in \rr 2: \ x \neq 0 \} \subseteq \rr 2 \setminus 0. 
\end{align*}

Suppose $(x_0, \xi_0) \in \rr 2 \setminus 0$, $(x_0, \xi_0) \notin W$ and $(x_0, \xi_0) \neq \{ 0 \} \times (\ro \setminus 0)$. 
There exists an open set $U$ such that $(x_0,\xi_0) \in U$, and $0 < 2 \delta \leqs \ep \leqs 1$,
such that 
\begin{align*}
(x,\xi) \in U, \quad  |x-y| \leqs \delta \quad \Longrightarrow 
& \quad  |\xi - c \alpha \sgn(x) |x|^{\alpha-1}| \geqs 2 \ep, \quad |x| \geqs \ep, \quad \\
& \quad \alpha \, |c| \, | \, \sgn(y) |y|^{\alpha-1} - \sgn(x) |x|^{\alpha-1}| \leqs \ep.
\end{align*}

Then if $(x,\xi) \in U$ and $|x-y| \leqs \delta$
\begin{equation}\label{eq:lowerboundfas1}
|\xi - c \alpha \sgn(y) |y|^{\alpha-1}| 
\geqs |\xi - c \alpha \sgn(x) |x|^{\alpha-1}| - \alpha \, |c| \, | \sgn(y) |y|^{\alpha-1} - \sgn(x) |x|^{\alpha-1}| \geqs \ep. 
\end{equation}

Let $\psi \in \Sigma_t^s(\ro) \setminus 0$ be such that $\supp \psi \subseteq \rB_\delta$. 
Then if $\lambda \geqs 1$, $\lambda^t(y-x) \in \supp \psi$ and $|x| \geqs \ep$ we have $|y| \geqs \ep/2$. 
From the stationary phase theorem \cite[Theorem~7.7.1]{Hormander0}
this gives for any $k \in \no$, any $h > 0$ and any $\lambda \geqs 1$, if $(x,\xi) \in U$, using \eqref{eq:lowerboundfas1}
and the final estimates in \eqref{eq:estimateWF0},
\begin{equation*}
\begin{aligned}
|V_\psi u ( \lambda^t x, \lambda^s \xi)|
& = (2 \pi)^{-\frac12} \left| \int_{\ro} e^{i (c |y|^\alpha - y \lambda^s \xi )} \overline{\psi( \lambda^t (\lambda^{-t}y-x) )} \, \dd y \right| \\
& = (2 \pi)^{-\frac12} \lambda^t \left| \int_{|y| \geqs \ep/2} e^{i \lambda^{t \alpha} (c |y|^\alpha - y \xi ) )} \overline{\psi( \lambda^t (y-x) )} \, \dd y \right| \\
& \leqs C \lambda^t \sum_{n = 0}^k \lambda^{n t} \sup_{|x-y| \leqs \delta} |(D^n\psi)( \lambda^t (y-x) )| \, |\xi - c \alpha \sgn(y)  |y|^{\alpha-1}|^{n - 2k} \lambda^{t \alpha(n-2k)} \\
& \leqs C \lambda^{t} \ep^{- 2 k} \sum_{n = 0}^k \sup_{|x-y| \leqs \delta} |(D^n\psi)( \lambda^t (y-x) )|  \lambda^{- t k (\alpha-1)} \lambda^{t(1+\alpha) (n-k)} \\
& \leqs C \lambda^{t} \ep^{- 2 k} \lambda^{- s k}  \sum_{n = 0}^k \sup_{|x-y| \leqs \delta} |(D^n\psi)( \lambda^t (y-x) )| \\
& \leqs C_{s,h}  \lambda^{t} \ep^{- 2 k} \lambda^{- s k} (2 h)^k  k!^s . 
\end{aligned}
\end{equation*}

Appealing to Lemma \ref{lem:chirplemma} it follows that
\begin{equation*}
(x_0, \xi_0) \notin \WF^{t,t(\alpha-1)} (u)
\end{equation*}
and we may conclude 
\begin{equation*}
\WF^{t,t(\alpha-1)} (u) \subseteq W \cup \{ 0 \} \times (\ro \setminus 0)
\end{equation*}
which is the right inclusion in \eqref{eq:conclusion2} for Case (ii). 

It remains to show the left inclusion in \eqref{eq:conclusion2}, that is
\begin{equation}\label{eq:conclusion2a}
\{ 0 \} \times (\ro \setminus 0) \subseteq 
\WF^{t,t(\alpha-1)} (u). 
\end{equation}

We have for $\xi > 0$
\begin{equation*}
|V_\psi u ( 0, \pm \lambda^s \xi)|
= (2 \pi)^{-\frac12} \left| \int_{\ro} e^{i (c |y|^\alpha \mp y \lambda^s \xi )} \overline{\psi( y )} \, \dd y \right| 
= |\cF( \psi \, e^{ - i c |\cdot|^\alpha } )(\mp \lambda^s \xi)|. 
\end{equation*}
Let $\psi$ be even and satisfy $\psi(0) \neq 0$. 
Then $\cF( \psi \, e^{ - i c |\cdot|^\alpha } )$ is also even. 
If we assume $(0, \xi ) \notin \WF^{t,t(\alpha-1)} (u)$ or $(0, -\xi ) \notin \WF^{t,t(\alpha-1)} (u)$ then 
\begin{equation*}
|\cF( \psi \, e^{ - i c |\cdot|^\alpha } )(\xi)|
\lesssim e^{- r |\xi|^{\frac1s}}, \quad \xi \in \ro, 
\end{equation*}
for all $r > 0$. But this implies $\psi \, e^{ - i c |\cdot|^\alpha } \in C^\infty$
which is a contradiction as $\alpha \notin 2 \no \setminus 0$ and $\psi(0) \neq 0$.  
This shows \eqref{eq:conclusion2a} and thus \eqref{eq:conclusion2} for Case (ii) has been proved. 
\end{proof}

\begin{rem}\label{rem:weakassumption}
The wave front set $\WF^{t,t(\alpha-1)} (u)$
is well defined if $t + t(\alpha-1) = t \alpha > 1$ for $u \in (\Sigma_t^{t(\alpha-1)})' (\ro)$. 
If we weaken the assumption $t > \frac{1}{m-1}$ ($t > \frac{1}{\alpha-1}$) into $t > \frac1m$ ($t > \frac{1}{\alpha}$)
in Theorem \ref{thm:chirpWFst}, 
then we obtain from Theorem \ref{thm:chirpWFst} and Remark \ref{rem:WFinclusion} 
if $m \in \no \setminus \{ 0, 1 \}$ 
\begin{equation}\label{eq:conclusion1c}
\{ (x, c \, m x^{m-1} ) \in \rr 2: \ x \neq 0 \} 
\subseteq \WF^{t,t(m-1)} (u) 
\end{equation}
and if $\alpha \in \ro \setminus 2 \no$, $\alpha > 1$ 
\begin{equation}\label{eq:conclusion2c}
\{ 0 \} \times (\ro \setminus 0) \subseteq \WF^{t,t(\alpha-1)} (u).  
\end{equation}
Thus \eqref{eq:conclusion1c} has been weakened into an inclusion instead of the equality \eqref{eq:conclusion1a},  
and \eqref{eq:conclusion2c} gives a lower bound only as compared to \eqref{eq:conclusion2}. 
\end{rem}

\begin{rem}\label{rem:fourierchirp}
The Fourier transform $\widehat u$ of a chirp \eqref{eq:chirpdef1} with $m \in \no \setminus \{ 0, 1 \}$ is known explicitly for $m = 2$. 
It is $\wh u(\xi) = (2 |c|)^{-\frac12} e^{i \frac{\pi}{4} \sgn(c)} e^{- \frac{i}{4c} \xi^2}$
\cite[Theorem~7.6.1]{Hormander0}. 
For larger $m$ one has $\widehat u \in \cS'(\ro)$. 
From the discussion concerning the Airy function ($m=3$, $c = \frac13$) \cite[Chapter~7.6]{Hormander0}
it can be seen that $\wh u$ is actually real analytic provided $m$ is odd, and extends to an entire function on $\co$. 
But if $m$ is even it seems difficult to obtain 
explicit information about $\widehat u$. 
Nevertheless, combining Theorem \ref{thm:chirpWFst} with Proposition \ref{prop:WFstsymplectic}, 
we obtain the following identity for its anisotropic Gelfand--Shilov wave front set when $t > \frac{1}{m-1}$: 
\begin{equation*}
\WF^{t(m-1),t} (\wh u) 
= \{ ( (-1)^{m-1} c m x^{m-1}, x ) \in \rr 2: \ x \neq 0 \}. 
\end{equation*}

If $m = 3$ and $c = 1/3$ then $u (x) = e^{i x^3/3}$ and 
$v (\xi) = (2 \pi)^{\frac12} \cF^{-1} u(\xi) = (2 \pi)^{\frac12} \wh u (-\xi)$ is the Airy function \cite{Hormander0}. 
Using \eqref{eq:evensteven0} we conclude
\begin{equation*}
\WF^{2t,t} (v) 
= -\WF^{2t,t} ( \wh u ) 
= \{ ( - x^2, x ) \in \rr 2: \ x \neq 0 \}
\end{equation*}
when $t > \frac12$. 
\end{rem}

We would also like to determine $\WF^{t,s} (u)$ when $s \neq t (\alpha-1)$ for the chirp functions. 
The following two results treat this question and show that $\WF^{t,s} (u)$ does not give a meaningful result then. 

\begin{prop}\label{prop:chirpnegative1}
Suppose 
$c \in \ro \setminus 0$. 

\begin{enumerate}[\rm (i)]

\item If $u$ is defined by \eqref{eq:chirpdef1} and $s > t (m-1) > 1$ then 
\begin{equation}\label{eq:conclusion3a}
\WF^{t,s} (u) = (\ro \setminus 0) \times \{ 0 \}. 
\end{equation}

\item If $u$ is defined by \eqref{eq:chirpdef2} and $s > t (\alpha-1) > 1$ then
\begin{equation}\label{eq:conclusion4}
\{ 0 \} \times (\ro \setminus 0) \subseteq 
\WF^{t,s} (u) 
\subseteq (\ro \setminus 0) \times \{ 0 \} 
\cup \{ 0 \} \times (\ro \setminus 0) . 
\end{equation}

\end{enumerate}
\end{prop}

\begin{proof}
Case (i):
Suppose $(x_0, \xi_0) \in \rr 2$ and $\xi_0 \neq 0$. 
There exists $U \subseteq \rr {2}$ such that $(x_0, \xi_0) \in U$, 
and $0 < \ep \leqs 1$, 
$L \geqs 1$ such that 
\begin{equation*}
| \xi - c m \lambda^{ t (m-1) - s} y^{m-1}| \geqs \ep
\end{equation*}
when $(x,\xi) \in U$, $| x - y| \leqs 1$ and $\lambda \geqs L$, due to the assumption $t(m-1) - s < 0$. 

Let $\psi \in \Sigma_t^s(\ro) \setminus 0$ be such that $\supp \psi \subseteq \rB_1$. 
From the stationary phase theorem \cite[Theorem~7.7.1]{Hormander0}
we have for any $k \in \no$, any $h > 0$ and any $\lambda \geqs L$, if $(x,\xi) \in U$, 
again using \eqref{eq:estimateWF0}, 
\begin{align*}
|V_\psi u ( \lambda^t x, \lambda^s \xi)|
& = (2 \pi)^{-\frac12} \left| \int_{\ro} e^{i (c y^{m} - y \lambda^s \xi )} \overline{\psi( \lambda^t (\lambda^{-t} y-x) )} \, \dd y \right| \\
& = (2 \pi)^{-\frac12} \lambda^t \left| \int_{\ro} e^{i \lambda^{t+s} ( \lambda^{t (m - 1)-s} c y^{m} - y \xi ) )} \overline{\psi( \lambda^t (y-x) )} \, \dd y \right| \\
& \leqs C \lambda^t \sum_{n = 0}^k  \lambda^{nt} \sup_{|x-y| \leqs 1} |(D^n\psi)( \lambda^t (y-x) )| \, |\xi - c m \lambda^{t ( m - 1)-s} y^{m-1}|^{n - 2k} \lambda^{(t+s)(n-2k)} \\
& \leqs C \lambda^{t} \ep^{-2 k} \sum_{n = 0}^k \sup_{|x-y| \leqs 1} |(D^n\psi)( \lambda^t (y-x) )|  \lambda^{- s k} \lambda^{ s (n-k) + 2t( n - k)} \\
& \leqs C \lambda^{t} \ep^{-2 k} \lambda^{- s k} \sum_{n = 0}^k \sup_{|x-y| \leqs 1} |(D^n\psi)( \lambda^t (y-x) )| \\
& \leqs C_{s,h}  \lambda^{t} \ep^{- 2 k} \lambda^{- s k} (2 h)^k  k!^s . 
\end{align*}
Lemma \ref{lem:chirplemma} gives 
\begin{equation*}
\WF^{t,s} (u) \subseteq  (\ro \setminus 0) \times \{ 0 \} 
\end{equation*}
which shows the inclusion ``$\subseteq$'' in \eqref{eq:conclusion3a}. 
Equality in \eqref{eq:conclusion3a}
again follows from \eqref{eq:evensteven0}, \eqref{eq:evensteven2}, $u \notin \Sigma_t^s(\ro)$, 
and $\check u = \overline u$ if $m$ is odd.

Case (ii):
Suppose $(x_0, \xi_0) \in \rr 2$, $x_0 \neq 0$ and $\xi_0 \neq 0$. 
Then there exists $U \subseteq \rr {2d}$ such that $(x_0, \xi_0) \in U$, 
and $0 < \ep \leqs 1$, 
$L \geqs 1$,
such that
\begin{equation*}
\inf_{(x,\xi) \in U}  |x| = \ep
\end{equation*}
and
\begin{equation*}
| \xi - c \alpha \sgn(y) \lambda^{ t (\alpha-1) - s} |y|^{\alpha-1}| \geqs \ep
\end{equation*}
when $(x,\xi) \in U$, $|x - y| \leqs \ep/2$ and $\lambda \geqs L$.

Pick $\psi \in \Sigma_t^s(\ro) \setminus 0$ such that $\supp \psi \subseteq \rB_{\ep/2}$. 
Then if $\lambda \geqs L$, $\lambda^t(y-x) \in \supp \psi$ and $|x| \geqs \ep$ we have $|y| \geqs \ep/2$. 
From the stationary phase theorem \cite[Theorem~7.7.1]{Hormander0}
this gives for any $k \in \no$, any $h > 0$ and any $\lambda \geqs L$, if $(x,\xi) \in U$,
using \eqref{eq:estimateWF0},
\begin{align*}
|V_\psi u ( \lambda^t x, \lambda^s \xi)|
& = (2 \pi)^{-\frac12} \left| \int_{\ro} e^{i (c |y|^\alpha - y \lambda^s \xi )} \overline{\psi( \lambda^t (\lambda^{-t}y-x) )} \, \dd y \right| \\
& = (2 \pi)^{-\frac12} \lambda^t \left| \int_{|y| \geqs \ep/2} e^{i \lambda^{t+s} (c \lambda^{t (\alpha-1) - s} |y|^\alpha - y \xi ) )} \overline{\psi( \lambda^t (y-x) )} \, \dd y \right| \\
& \leqs C \lambda^t \sum_{n = 0}^k \lambda^{n t} \sup_{|x-y| \leqs \ep/2} |(D^n\psi)( \lambda^t (y-x) )| \, |\xi - c \alpha \sgn(y) \lambda^{t (\alpha-1) - s} |y|^{\alpha-1}|^{n - 2k} \\
& \qquad \qquad \times \lambda^{(t+s) (n-2k)} \\
& \leqs C_{s,h} \lambda^{t} \ep^{- 2k} \lambda^{- s k}  (2 h)^k k!^s. 
\end{align*}
Using Lemma \ref{lem:chirplemma} we obtain
\begin{equation*}
\WF^{t,s} (u) \subseteq  (\ro \setminus 0) \times \{ 0 \} \cup \{ 0 \} \times (\ro \setminus 0). 
\end{equation*}
Finally $\{ 0 \} \times (\ro \setminus 0) \subseteq \WF^{t,s} (u)$ follows recycling the argument at the end of the proof of 
Theorem \ref{thm:chirpWFst}. 
\end{proof}

\begin{prop}\label{prop:chirpnegative2}
Suppose 
$c \in \ro \setminus 0$. 

\begin{enumerate}[\rm (i)]

\item If $u$ is defined by \eqref{eq:chirpdef1} and $t (m-1) > s > 1$ then
\begin{equation}\label{eq:conclusion5a}
\WF^{t,s} (u) \subseteq \{ 0 \} \times (\ro \setminus 0) 
\end{equation}
and if $m$ is even then 
\begin{equation}\label{eq:conclusion5b}
\WF^{t,s} (u) = \{ 0 \} \times (\ro \setminus 0). 
\end{equation}

\item If $u$ is defined by \eqref{eq:chirpdef2} and $t (\alpha-1) > s > 1$ then
\begin{equation}\label{eq:conclusion6}
\{ 0 \} \times (\ro \setminus 0) \subseteq 
\WF^{t,s} (u) 
\subseteq (\ro \setminus 0) \times \{ 0 \} 
\cup \{ 0 \} \times (\ro \setminus 0) . 
\end{equation}

\end{enumerate}
\end{prop}

\begin{proof}
Case (i):
Suppose $(x_0, \xi_0) \in \rr 2$ and $x_0 \neq 0$. 
There exists $U \subseteq \rr {2d}$ such that $(x_0, \xi_0) \in U$, 
and $0 < \ep \leqs 1$, 
$L \geqs 1$, 
such that 
\begin{equation*}
| c m  y^{m-1} - \lambda^{ s - t (m-1)} \xi| \geqs \ep
\end{equation*}
when $(x,\xi) \in U$, $| x - y| \leqs \ep$ and $\lambda \geqs L$, due to the assumption $s - t(m-1) < 0$. 

If $0 \leqs n \leqs k$ we have 
\begin{equation*}
s k + n t + t m(n-2k)
< t ( k(m-1) + n - m k ) \\
\leqs 0. 
\end{equation*}

Let $\psi \in \Sigma_t^s(\ro) \setminus 0$ be such that $\supp \psi \subseteq \rB_\ep$. 
From the stationary phase theorem \cite[Theorem~7.7.1]{Hormander0}
we have for any $k \in \no$, 
any $h > 0$ and any $\lambda \geqs L$, if $(x,\xi) \in U$, 
again reusing \eqref{eq:estimateWF0}, 
\begin{align*}
|V_\psi u ( \lambda^t x, \lambda^s \xi)|
& = (2 \pi)^{-\frac12} \left| \int_{\ro} e^{i (c y^{m} - y \lambda^s \xi )} \overline{\psi( \lambda^t (\lambda^{-t} y-x) )} \, \dd y \right| \\
& = (2 \pi)^{-\frac12} \lambda^t \left| \int_{\ro} e^{i \lambda^{t m} ( c y^{m} - \lambda^{s - t (m - 1)}  y \xi ) )} \overline{\psi( \lambda^t (y-x) )} \, \dd y \right| \\
& \leqs C \lambda^t \sum_{n = 0}^k \lambda^{n t} \sup_{|x-y| \leqs \ep} |(D^n\psi)( \lambda^t (y-x) )| \, |c m  y^{m-1} - \lambda^{s - t (m - 1)} \xi |^{n - 2k} \lambda^{t m(n-2k) } \\
& \leqs C \lambda^{t} \ep^{-2 k} \lambda^{- s k} \sum_{n = 0}^k \sup_{|x-y| \leqs \ep} |(D^n\psi)( \lambda^t (y-x) )| \\
& \leqs C_{s,h}  \lambda^{t} \ep^{- 2 k} \lambda^{- s k} (2h)^k  k!^s . 
\end{align*}

Lemma \ref{lem:chirplemma} gives
\begin{equation*}
\WF^{t,s} (u) \subseteq \{ 0 \} \times (\ro \setminus 0) 
\end{equation*}
which is \eqref{eq:conclusion5a}. 
The equality \eqref{eq:conclusion5b} when $m$ is even follows from \eqref{eq:evensteven1} 
and $u \notin \Sigma_t^s(\ro)$.

Case (ii):
Suppose $(x_0, \xi_0) \in \rr 2$, $x_0 \neq 0$ and $\xi_0 \neq 0$. 
Then there exists $U \subseteq \rr {2d}$ such that $(x_0, \xi_0) \in U$, 
and $0 < \ep \leqs 1$, $L \geqs 1$, such that 
\begin{equation*}
\inf_{(x,\xi) \in U}  |x| = \ep
\end{equation*}
and
\begin{equation*}
| \xi - c \alpha \sgn(y) \lambda^{ t (\alpha-1) - s} |y|^{\alpha-1}| \geqs \ep
\end{equation*}
when $(x,\xi) \in U$, $|x - y| \leqs \ep/2$ and $\lambda \geqs L$. 

If $n \leqs k$ then  
\begin{equation*}
s k + n t + (t+s) (n-2k) 
\leqs s k  + n t - (t+s) k 
\leqs 0. 
\end{equation*}

Let $\psi \in \Sigma_t^s(\ro) \setminus 0$ be such that $\supp \psi \subseteq \rB_{\ep/2}$. 
Then if $\lambda \geqs L$, $\lambda^t(y-x) \in \supp \psi$ and $|x| \geqs \ep$ we have $|y| \geqs \ep/2$. 
From the stationary phase theorem \cite[Theorem~7.7.1]{Hormander0}
this gives for any $k \in \no$, any $h > 0$ and any $\lambda \geqs L$, if $(x,\xi) \in U$
and the final estimates in \eqref{eq:estimateWF0},
\begin{align*}
|V_\psi u ( \lambda^t x, \lambda^s \xi)|
& = (2 \pi)^{-\frac12} \left| \int_{\ro} e^{i (c |y|^\alpha - y \lambda^s \xi )} \overline{\psi( \lambda^t (\lambda^{-t}y-x) )} \, \dd y \right| \\
& = (2 \pi)^{-\frac12} \lambda^t \left| \int_{|y| \geqs \ep/2} e^{i \lambda^{t+s} (c \lambda^{t (\alpha-1) - s} |y|^\alpha - y \xi ) )} \overline{\psi( \lambda^t (y-x) )} \, \dd y \right| \\
& \leqs C \lambda^t \sum_{n = 0}^k \lambda^{n t} \sup_{|x-y| \leqs \ep/2} |(D^n\psi)( \lambda^t (y-x) )| \, |\xi - c \alpha \sgn(y) \lambda^{t (\alpha-1) - s} |y|^{\alpha-1}|^{n - 2k} \\
& \qquad \qquad \times \lambda^{(t+s) (n-2k)} \\
& \leqs C_{s,h} \lambda^{t} \ep^{- 2k} \lambda^{- s k}  (2 h)^k k!^s. 
\end{align*}
Lemma \ref{lem:chirplemma} gives again
\begin{equation*}
\WF^{t,s} (u) \subseteq  (\ro \setminus 0) \times \{ 0 \} \cup \{ 0 \} \times (\ro \setminus 0). 
\end{equation*}
Finally $\{ 0 \} \times (\ro \setminus 0) \subseteq \WF^{t,s} (u)$ follows again using the argument at the end of the proof of 
Theorem \ref{thm:chirpWFst}. 
\end{proof}

\begin{rem}\label{rem:schrodingernonquadratic}
By using Theorem \ref{thm:chirpWFst}, Propositon \ref{prop:chirpnegative1} and Propositon \ref{prop:chirpnegative2}
we may now give a counterpart of Remark \ref{rem:propagator}, 
showing that the anisotropic wave front set turns out to be needed 
when treating Schr\"odinger propagators in the case of non-quadratic potentials. 

Consider the Cauchy problem for the anisotropic free particle equation 
in dimension $d = 1$ 
\begin{equation}\label{eq:schrodeq2}
\left\{
\begin{array}{rl}
\partial_t u(t,x) + i D_x^{m} u (t,x) & = 0, \quad m \in \no \setminus \{ 0, 1 \}, \\
u(0,\cdot) & = u_0.  
\end{array}
\right.
\end{equation}
The Hamilton flow, along which we expect propagation of microlocal singularities, is given by
\begin{equation}\label{eq:hamiltonflow1}
(x,\xi) = \chi_t (x_0, \xi_0)
= (x_0 + m t \xi_0^{m-1}, \xi_0), \quad t \in \ro,
\end{equation}
and we are looking for parameters $v,s > 0$ such that $v + s > 1$ and  
\begin{equation}\label{eq:propagationschrodnonq1}
\WF^{v,s} (e^{- i t D_x^{m}} u_0 ) = \chi_t (\WF^{v,s} (u_0) ) . 
\end{equation}

The explicit solution to \eqref{eq:schrodeq2} is given by
\begin{equation}\label{eq:solutionschrod2}
u (t,x) 
= e^{- i t D_x^{m}} u_0 
= (2 \pi)^{- \frac12} \int_{\ro} e^{i x \xi - i t \xi^{m}} \wh u_0 (\xi) \dd \xi. 
\end{equation}

For simplicity let us test \eqref{eq:propagationschrodnonq1} on the case $u_0 = \delta_0$, 
and denote by $w_t$ the solution to \eqref{eq:schrodeq2}. 
It is easy to prove that 
\begin{equation*}
\WF^{v,s} ( w_0 ) = \WF^{v,s} ( \delta_0 ) = \{ 0 \} \times (\ro \setminus 0)
\end{equation*}
for any $v,s > 0$ with $v + s > 1$, cf. \eqref{eq:diracWF} for $v = s$ and Proposition \ref{prop:WFstelementary}, 
and from \eqref{eq:solutionschrod2}
\begin{equation}\label{eq:solutionschrod3}
\wh w_t (\xi) = (2 \pi)^{- \frac12}  e^{-i t \xi^{m}}.
\end{equation}
Hence from \eqref{eq:hamiltonflow1} and \eqref{eq:propagationschrodnonq1} we expect
\begin{equation}\label{eq:propagationschrodnonq2}
\WF^{v,s} ( w_t ) = \chi_t ( \{ 0 \} \times (\ro \setminus 0) ) = \{ (m t \xi^{m-1}, \xi) \in \rr 2, \ \xi \neq 0 \}. 
\end{equation}

This shows that the correct choice is $v = s (m-1) > 1$, $s > 0$. 
In fact from Theorem \ref{thm:chirpWFst} applied to \eqref{eq:solutionschrod3} we have for $v (m-1) > 1$ and $t \neq 0$
\begin{equation}\label{eq:solutionschrod4}
\WF^{v, v (m-1)} (\wh w_t) = \{ (\xi, - m t \xi^{m-1} ) \in \rr 2, \ \xi \neq 0 \}
\end{equation}
and hence, in view of Proposition \ref{prop:WFstsymplectic}, swapping the roles of $s$ and $v$, 
\begin{equation}\label{eq:solutionschrod5}
\WF^{s(m-1),s} (w_t) = \{ (m t \xi^{m-1}, \xi ) \in \rr 2, \ \xi \neq 0 \}
\end{equation}
if $s (m-1) > 1$, 
as expected from \eqref{eq:propagationschrodnonq2}. 

Other choices of $v > 0$ do not work. In fact by applying Proposition \ref{prop:chirpnegative1}
to \eqref{eq:solutionschrod3} we have if $v > s(m-1) > 1$
\begin{equation*}
\WF^{s, v} ( \wh w_t) = (\ro \setminus 0) \times \{ 0 \} 
\end{equation*}
and hence
\begin{equation*}
\WF^{v,s} ( w_t) =  \{ 0 \} \times (\ro \setminus 0)
\end{equation*}
for every $t \in \ro$. 

Whereas by applying Proposition \ref{prop:chirpnegative2}
to \eqref{eq:solutionschrod3} we have if $1 < v < s(m-1)$, 
in particular if $v = s > 1$, we obtain
\begin{equation*}
\WF^{s, v} ( \wh w_t ) \subseteq \{ 0 \} \times (\ro \setminus 0), 
\end{equation*}
hence
\begin{equation*}
\WF^{v,s} ( w_t ) \subseteq  (\ro \setminus 0) \times \{ 0 \}, \quad t \neq 0.
\end{equation*}
(These inclusions are equalities if $m$ is even.)
This is not consistent with \eqref{eq:propagationschrodnonq2}. 
\end{rem}

\begin{rem}\label{rem:comment}
\textit{Addendum at revision.}
After finishing this work we have proved a generalization of the conjecture \eqref{eq:propagationschrodnonq2} with $v = s(m-1) > 1$, 
see \cite[Theorem~7.1]{Wahlberg4}.
\end{rem}

\section{Relations between the $t,s$-Gelfand--Shilov wave front set and the $s$-Gevrey wave front set}\label{sec:GSGevrey}

Next we show a few results that are valid when $s > 1$. 
Then Gevrey functions of order $s$ and of compact support exist \cite{Rodino1}. 
We define Gevrey functions of order $s > 1$ slightly differently from \cite{Rodino1}, using again Beurling instead of Roumieu type. 
Let $\Omega \subseteq \rr d$ be open. Then $f \in G^s(\Omega)$ provided $f \in C^\infty(\Omega)$
and for each compact $K \subseteq \Omega$ we have
\begin{equation*}
|\pd \alpha f (x)| \leqs C_{K,h} h^{|\alpha|} \alpha!^s, \quad x \in K, \quad \alpha \in \nn d, \quad \forall h > 0. 
\end{equation*}
The topology on $G^s(\Omega)$ is defined first as the projective limit with respect to $h > 0$, and then as the inductive limit with respect to an exhaustive increasing sequence of compact sets $K \subseteq \Omega$.
In the sequel we limit attention to $\Omega = \rr d$. 

The space of compactly supported Gevrey functions is embedded in the usual test function space as 
$G_c^s(\rr d) \subseteq C_c^\infty(\rr d)$. 
The topological duals therefore satisfy the embedding $\cD'(\rr d) \subseteq \cD_s'(\rr d)$
where $\cD_s'(\rr d)$ is the space of Gevrey ultradistributions of order $s > 1$. 

With small modifications of the proof of \cite[Theorem~1.6.1]{Rodino1} we obtain that for $f \in C_c^\infty(\rr d)$ we have $f \in G_c^s(\rr d)$ if and only if the Fourier transform satisfies
\begin{equation*}
|\wh f(\xi)| \lesssim e^{- r |\xi|^{\frac1s}} \quad \forall r > 0. 
\end{equation*}
Denoting $\cE_s' (\rr d)$
the subspace of $\cD_s'(\rr d)$ of ultradistributions of compact support, 
we also have $f \in \cE_s' (\rr d)$ if and only if 
\begin{equation*}
\exists r > 0: \quad |\wh f(\xi)| \lesssim e^{r |\xi|^{\frac1s}}
\end{equation*}
cf. \cite{Komatsu1,Sobak1} and \cite[Theorems~1.6.1 and 1.6.7]{Rodino1}. 

This is the basis of the definition of the Gevrey wave front set $\WF_s( u)$ of $u \in \cD_s'(\rr d)$ \cite{Rodino1}. 
A phase space point $(x_0, \xi_0) \in \rr d \times (\rr d \setminus 0)$ 
satisfies $(x_0, \xi_0) \notin \WF_s( u)$ if there exists $\fy \in G_c^s(\rr d)$ such that $\fy(x_0) = 1$
and an open conical neighborhood $\Gamma \subseteq \rr d \setminus 0$ containing $\xi_0$ such that 
\begin{equation*}
\sup_{\xi \in \Gamma} e^{r |\xi|^{\frac1s}} |\wh{ u \fy} (\xi)| < \infty  \quad \forall r > 0. 
\end{equation*}

Hence $\WF_s( u) = \emptyset$ if and only if $u \in G^s(\rr d)$. 
Note that for every $s > 1$ and any $t > 0$ we have
\begin{align*}
& G_c^s(\rr d) \subseteq \Sigma_t^s(\rr d) \subseteq G^s(\rr d), \\
& \cE_s' (\rr d) \subseteq ( \Sigma_t^s )' (\rr d) \subseteq \cD_s'(\rr d). 
\end{align*}

Inspired by the proofs in \cite{Wahlberg2} we obtain the following results. 
Here $\pi_2(x,\xi) = \xi$ for $(x,\xi) \in T^* \rr d$. 

\begin{prop}\label{prop:WFGevreyWFst}
If $t \geqs s > 1$ and 
$u \in (\Sigma_t^s)'(\rr d)$ then 
\begin{equation*}
\{ 0 \} \times \pi_2 \WF_s(u) \subseteq \WF^{t,s}(u). 
\end{equation*}
\end{prop}

\begin{proof}
Suppose $\xi_0 \in \rr d \setminus 0$ and $(0,\xi_0) \notin \WF^{t,s}(u)$. 
By \eqref{eq:WFstscalinv} we may assume that $|\xi_0| = 1$. 
Let $\fy \in G_c^s(\rr d) \subseteq \Sigma_t^s (\rr d)$ satisfy $\fy (0) = 1$.  
We have for some $\ep > 0$, for any $r > 0$
\begin{equation*}
e^{r \lambda} |V_\fy u ( \lambda^t x, \lambda^s (\xi_0 + \xi) )|  \leqs C_r < \infty
\end{equation*}
if $(x,\xi) \in \rB_\ep$ and $\lambda > 0$. 
Define the open set
\begin{equation*}
\Gamma = \{ ( \lambda^t x, \lambda^s (\xi_0 + \xi) ) \in \rr {2d}: \, (x,\xi) \in \rB_\ep, \, \lambda > 0 \} \subseteq \rr {2d}. 
\end{equation*}

We have to show that $(x_0,\xi_0) \notin \WF_s(u)$ for all $x_0 \in \rr d$. 
Let $x_0 \in \rr d$.
Define for $\delta > 0$ the open conic set containing $\xi_0$
\begin{equation*}
\Gamma_{\delta} = \left\{ \xi \in \rr d \setminus 0: \, \left| \frac{\xi}{|\xi|} - \xi_0 \right| < \delta \right\} \subseteq \rr d \setminus 0. 
\end{equation*}

Pick $\delta > 0$ sufficiently small so that $\delta (1 + |x_0|) \leqs 1$ and 
\begin{equation*}
\delta^2 \left(1 +  \frac{| x_0 |^2}{(1 - \delta |x_0| )^2}  \right) \leqs \ep^2.
\end{equation*}

Then we have 
\begin{equation}\label{eq:gammainclusion1}
(\{ x_0 \} \times \Gamma_{\delta}) \setminus \rB_{\delta^{-1}} \subseteq \Gamma. 
\end{equation}
In fact let $\eta \in \Gamma_{\delta}$ and $|(x_0, \eta)| \geqs \delta^{-1}$. 
Then $|\eta| \geqs \delta^{-1}  -  |x_0| \geqs 1$. 
We write for $\lambda > 0$
\begin{equation*}
(x_0, \eta) = ( \lambda^t x, \lambda^s (\xi_0 + \xi) )
\end{equation*}
that is $x = \lambda^{-t} x_0$ and $\xi = \lambda^{-s} \eta - \xi_0$. 
In order to show \eqref{eq:gammainclusion1} we have to show that $(x,\xi) \in \rB_\ep$ for some $\lambda > 0$. 

If we set $\lambda = |\eta|^{\frac1s} > 0$ then $|\xi| < \delta$ and we obtain using the assumption $t \geqs s$
\begin{align*}
| x |^2 + | \xi |^2
& < \lambda^{-2t} |x_0|^2 + \delta^2
= |\eta|^{- \frac{2t}s} |x_0|^2 + \delta^2
\leqs |\eta|^{-2} |x_0|^2 + \delta^2 \\
& \leqs (  \delta^{-1}  -  |x_0| )^{-2} |x_0|^2 +  \delta^2
=  \delta^2 \left( 1 + \frac{|x_0|^2}{(1 - \delta |x_0|)^2} \right)
\leqs \ep^2. 
\end{align*}
Thus $(x,\xi) \in \rB_\ep$ and we have shown \eqref{eq:gammainclusion1}. 

Finally let $\eta \in \Gamma_{\delta}$ and
$|\eta| \geqs \delta^{-1} + |x_0|$, which implies $|(x_0, \eta)| \geqs \delta^{-1}$. 
By \eqref{eq:gammainclusion1} we have $(x_0, \eta) \in \Gamma$, that is 
$(x_0, \eta) = (\lambda^t x, \lambda^s (\xi_0 + \xi) )$ for some $\lambda > 0$ and some $(x,\xi) \in \rB_\ep$. 
Since
\begin{equation*}
|\eta|^{\frac1s} = 
\lambda |\xi_0 + \xi|^{\frac1s} 
\leqs \lambda \kappa (s^{-1}) \left( |\xi_0|^{\frac1s} + \ep^{\frac1s} \right)
\end{equation*} 
we obtain for any $r > 0$
\begin{align*}
\sup_{\eta \in \Gamma_{\delta}, \ |\eta| \geqs \delta^{-1} + |x_0|} e^{r | \eta |^{\frac1s}}  |V_\fy u(x_0,\eta)| 
& \leqs \sup_{(x,\xi) \in \rB_\ep, \ \lambda > 0} e^{ \lambda r \kappa (s^{-1}) \left( |\xi_0|^{\frac1s} + \ep^{\frac1s} \right)} |V_\fy u( \lambda^t x, \lambda^s (\xi_0 + \xi) )| \\
\leqs C_{ r \kappa (s^{-1}) ( |\xi_0|^{\frac1s} + \ep^{\frac1s})}
\end{align*} 
which shows that $(x_0,\xi_0) \notin \WF_s(u)$. 
\end{proof}

The following result gives a sufficient condition for the opposite inclusion. 

\begin{prop}\label{prop:WFsfreqaxis}
If $s > 1$, $t > 0$
and $u \in 
\cE_s'(\rr d) + \Sigma_t^s(\rr d)$ then 
\begin{equation}\label{eq:WFstinclusion}
\WF^{t,s}(u) \subseteq \{ 0 \} \times \pi_2 \WF_s(u). 
\end{equation}
\end{prop}

\begin{proof}
We may assume $u \in \cE_s'(\rr d) \subseteq ( \Sigma_t^s )' (\rr d)$. 
We start with the less precise inclusion 
\begin{equation}\label{eq:subinclusion1}
\WF^{t,s} (u) \subseteq \{ 0 \} \times (\rr d \setminus 0). 
\end{equation} 
Suppose $(x_0,\xi_0) \in \rr {2d}$ with $x_0 \neq 0$. 
We pick a neighborhood $U \subseteq \rr {2d}$ such that $(x_0,\xi_0) \in U$ and 
\begin{equation*}
\inf_{(x,\xi) \in U} |x| = \delta > 0. 
\end{equation*} 
If we pick $\fy \in G_c^s(\rr d) \subseteq \Sigma_t^s(\rr d)$ we have $V_\fy u(x,\xi) = 0$ if $|x| \geqs r$ for $r>0$ sufficiently large due to $u \in \cE_s'(\rr d)$. This implies that $V_\fy u(\lambda^t x,\lambda^s \xi) = 0$ if $\lambda^t \geqs r \delta^{-1}$, for all $(x,\xi) \in U$. 
Hence $(x_0, \xi_0) \notin \WF^{t,s} (u)$ and we have shown \eqref{eq:subinclusion1}.

In order to show the sharper inclusion \eqref{eq:WFstinclusion}, suppose $0 \neq (x_0,\xi_0) \notin \{ 0 \} \times \pi_2 \WF_s (u)$. 
Then either $x_0 \neq 0$ or $\xi_0 \notin \pi_2 \WF_s (u)$. If $x_0 \neq 0$ then by \eqref{eq:subinclusion1} we have $(x_0,\xi_0) \notin \WF^{t,s}(u)$. 
Therefore we may assume that $x_0=0$, $\xi_0 \notin \pi_2 \WF_s (u)$ and $\xi_0 \neq 0$, and our goal is to show $(0,\xi_0) \notin \WF^{t,s}(u)$, which will prove \eqref{eq:WFstinclusion}. 

By a slight modification to the Gevrey framework of the proof of \cite[Proposition~8.1.3]{Hormander0} we have $\pi_2 \WF_s(u) = V_s(u)$, where $V_s(u) \subseteq \rr d \setminus 0$ is a closed conic set defined as follows for $u \in \cE_s '(\rr d)$.
A point $\eta \in \rr d \setminus 0$ satisfies $\eta \notin V_s (u)$ if $\eta \in \Gamma_2$ where $\Gamma_2 \subseteq \rr d \setminus 0$ is open and conic, and 
\begin{equation}\label{eq:frequencydecay0}
\sup_{\xi \in \Gamma_2} e^{r |\xi|^{\frac1s}} |\widehat u(\xi)| < \infty \quad \forall r > 0. 
\end{equation} 

Thus we have $\xi_0 \notin V_s(u)$, so there exists an open conic set $\Gamma_2 \subseteq \rr d \setminus 0$ such that $\xi_0 \in \Gamma_2$, and \eqref{eq:frequencydecay0} holds. 
Let $\ep > 0$ be small enough so that $\xi_0 + \rB_{2 \ep} \subseteq \Gamma_2$. 
We assume $\ep \leqs \frac12 |\xi_0|$ which gives $| \xi_0  + \xi | >  \frac12 |\xi_0|$ when $|\xi| < \ep$.

We have 
\begin{equation*}
V_\fy u (x,\xi) = \widehat{u T_x \overline{\fy}} (\xi)
= (2 \pi)^{-\frac{d}2} \widehat{u} * \widehat{T_x \overline{\fy}} (\xi)
\end{equation*} 
which gives
\begin{equation}\label{eq:STFTconvolution1}
|V_\fy u (x,\xi)| 
\lesssim |\wh u| * |g| (\xi), \quad x, \ \xi \in \rr d, 
\end{equation} 
where $g (\xi)= \widehat \fy(-\xi) \in 
\Sigma_s^t (\rr d)$.  
Since 
$u \in \cE_s' (\rr d)$
we obtain from the Paley--Wiener--Schwartz theorem (Gevrey version cf. \cite{Komatsu1,Sobak1} and \cite[Theorems~1.6.1 and 1.6.7]{Rodino1}) for some $a > 0$
\begin{equation}\label{eq:PWSGevrey}
|\widehat{u} (\xi)| \lesssim e^{a |\xi|^{\frac1s}}, \quad \xi \in \rr d, 
\end{equation} 
and we have 
\begin{equation}\label{eq:PWSGevrey2}
|g(\xi)| 
\lesssim e^{-r |\xi|^{\frac1s}}, 
\quad \xi \in \rr d, \quad \forall r > 0.
\end{equation}

Let $(x,\xi) \in \rB_\ep$, $r > 0$ and $\lambda > 0$.  
We have 
\begin{align*}
e^{r \lambda} |V_\fy u ( \lambda^t x,\lambda^s (\xi_0 + \xi ))| 
\lesssim e^{r \lambda} \int_{\rr d} | \wh u ( \lambda^s (\xi_0 + \xi  - \lambda^{-s} \eta) ) | \, | g (\eta) | \, \dd \eta
= I_1 + I_2
\end{align*} 
where we split the integral into the two terms
\begin{align*}
I_1 = & e^{r \lambda} \int_{\rr {d} \setminus \Omega_\lambda} | \wh u ( \lambda^s (\xi_0 + \xi  - \lambda^{-s} \eta) ) | \, | g (\eta) | \, \dd \eta, \\
I_2 = & e^{r \lambda}  \int_{\Omega_\lambda}  | \wh u ( \lambda^s (\xi_0 + \xi  - \lambda^{-s} \eta) ) | \, | g (\eta) | \, \dd \eta
\end{align*}
where  
\begin{equation*}
\Omega_\lambda = \{ \eta \in \rr d: \,  |\eta|^{\frac1s} < \lambda \ep^{\frac1s}  \} \subseteq \rr d. 
\end{equation*}

For $I_1$ we use \eqref{eq:PWSGevrey} which together with \eqref{eq:PWSGevrey2} give for any $r_1 > 0$
\begin{align*}
I_1 \lesssim & e^{r \lambda} \int_{\rr {d} \setminus \Omega_\lambda} e^{a| \lambda^s (\xi_0 + \xi) - \eta) |^{\frac1s}} \, | g (\eta) | \, \dd \eta \\
& \leqs e^{\lambda (r + a \kappa(s^{-1}) |\xi_0 + \xi |^{\frac1s}) } \int_{\rr {d} \setminus \Omega_\lambda} e^{ \kappa(s^{-1}) a | \eta |^{\frac1s}} \, | g (\eta) | \, \dd \eta \\
& \leqs e^{\lambda (r + a \kappa(s^{-1})(|\xi_0| + \ep)^{\frac1s})}  \int_{\rr {d} \setminus \Omega_\lambda} e^{ ( \kappa(s^{-1}) a - \kappa(s^{-1}) a - r_1 - 1) | \eta |^{\frac1s}} \, \dd \eta \\
& \leqs e^{\lambda (r + a \kappa(s^{-1})(|\xi_0|+\ep)^{\frac1s} - r_1 \ep^{\frac1s})} \int_{\rr d} e^{ - | \eta |^{\frac1s}} \, \dd \eta \\
& \leqs C_r
\end{align*}
provided we pick $r_1 \geqs \ep^{-\frac1s}( r+ a \kappa(s^{-1})(|\xi_0|+\ep)^{\frac1s} )$.

It remains to estimate $I_2$. 
If $\eta \in \Omega_\lambda$ then $\lambda^{-s} | \eta| < \ep$ which implies $\xi  - \lambda^{-s} \eta \in \rB_{2 \ep}$, and thus $\xi_0 + \xi  - \lambda^{-s} \eta \in \Gamma_2$. 
Since $\Gamma_2$ is conic we have $\lambda^s (\xi_0 + \xi  - \lambda^{-s} \eta) \in \Gamma_2$. 
Thus we may use \eqref{eq:frequencydecay0}, which together with \eqref{eq:PWSGevrey2} give for any $r_1, r_2 > 0$
\begin{align*}
I_2 & \lesssim e^{r \lambda}  \int_{\Omega_\lambda}  e^{- \kappa(s^{-1}) r_1  |  \lambda^s (\xi_0 + \xi  - \lambda^{-s} \eta) |^{\frac1s} } \, | g (\eta) | \, \dd \eta \\
& \leqs e^{r \lambda}  \int_{\Omega_\lambda}  e^{- r_1  \lambda |  \xi_0 + \xi |^{\frac1s} + \kappa(s^{-1}) r_1 |\eta |^{\frac1s} } \, | g (\eta) | \, \dd \eta \\
& \leqs e^{\lambda (r - r_1 2^{- \frac1s} | \xi_0 |^{\frac1s} )}  \int_{\rr d}   e^{ \kappa(s^{-1}) r_1 |\eta |^{\frac1s} } \, | g (\eta) | \, \dd \eta \\
& \lesssim e^{\lambda (r - r_1 2^{- \frac1s} | \xi_0 |^{\frac1s} )}  \int_{\rr d}   e^{(\kappa(s^{-1}) r_1 - r_2)|\eta |^{\frac1s} }  \dd \eta \\
& \leqs C_r
\end{align*}
if we first pick $r_1 \geqs 2^{ \frac1s} | \xi_0 |^{-\frac1s}  r$ and then pick $r_2 > \kappa(s^{-1}) r_1$. 
We have shown $(0,\xi_0) \notin \WF^{t,s} (u)$. 
\end{proof}

\begin{cor}\label{cor:WFsfreqaxis}
If $t \geqs s > 1$ and $u \in 
\cE_s' (\rr d)
+ \Sigma_t^s(\rr d)$ then 
\begin{equation*}
\WF^{t,s}(u) = \{ 0 \} \times V_s(u). 
\end{equation*}
\end{cor}

The following result is a sort of converse to Corollary \ref{cor:WFsfreqaxis}. 

\begin{prop}\label{prop:WFsnotfreqaxis}
Let $s,t > 0$ satisfy $s + t > 1$, and let  
$u \in (\Sigma_t^s)'(\rr d)$. 
If 
\begin{equation*}
\WF^{t,s}(u) \cap \{ 0 \} \times (\rr d \setminus 0) = \emptyset
\end{equation*}
then $u \in C^\infty(\rr d)$ and there exist $C,r > 0$ such that 
\begin{equation}\label{eq:exponentialbound1}
|\partial^\alpha u(x)| \leqs C^{1+|\alpha|} \alpha!^s e^{r |x|^\frac1t}, \quad x \in \rr d, \quad \alpha \in \nn d. 
\end{equation}
\end{prop}

\begin{proof}
Let $\fy \in \Sigma_t^s (\rr d)$ satisfy $\| \fy \|_{L^2} = 1$. 
Using the compactness of $\sr {d-1} \subseteq \rr d$ we obtain the following conclusion from the assumption. 
There exists $\ep > 0$ such that 
\begin{equation}\label{eq:decayGamma1}
\sup_{(x,\xi) \in \rB_\ep, \ \xi_0 \in \sr {d-1}, \ \lambda > 0} e^{r \lambda} |V_\fy u (\lambda^t x, \lambda^s (\xi_0 + \xi))| < \infty \quad \forall r > 0. 
\end{equation}
Set
\begin{equation*}
\Gamma = \{ (\lambda^t x, \lambda^s (\xi_0 + \xi) ) \in \rr {2d} \setminus 0: \, \xi_0 \in \sr {d-1}, \, (x,\xi) \in \rB_\ep, \, \lambda > 0 \}. 
\end{equation*}
If $(y,\eta) \in \Gamma$ then $\eta = \lambda^s (\xi_0 + \xi)$ and $y = \lambda^t x$ for some $\xi_0 \in \sr {d-1}$, $(x,\xi) \in \rB_\ep$, and $\lambda > 0$, 
so $|\eta|^{\frac1s} = \lambda |\xi_0 + \xi|^{\frac1s} < \lambda ( 1 + \ep )^{\frac1s}$ and $|y|^\frac1t = \lambda |x|^\frac1t < \lambda \ep^\frac1t$. 
Thus from \eqref{eq:decayGamma1} it follows that we have 
\begin{equation}\label{eq:decayGamma2}
\sup_{(x,\xi) \in \Gamma} e^{r ( |x|^{\frac1t} + |\xi|^{\frac1s} )} |V_\fy u (x,\xi)| < \infty \quad \forall r > 0. 
\end{equation}

We claim that if $(y,\eta) \in \rr {2d} \setminus 0$ then
\begin{equation}\label{eq:implicationGamma}
|y|^{\frac1t} < \ep^{\frac1t} | \eta |^{\frac1s} \quad \Longrightarrow \quad (y, \eta) \in \Gamma. 
\end{equation}
In fact suppose $|y|^{\frac1t} < \ep^{\frac1t} | \eta |^{\frac1s}$. Since $\eta \neq 0$ we may define $\lambda = |\eta|^\frac1s > 0$
and $\xi_0 = \lambda^{-s} \eta \in \sr {d-1}$, whence $\eta = \lambda^s \xi_0$. Set $x = \lambda^{-t} y$ so that $y = \lambda^{t} x$. We have 
\begin{equation*}
|x|^\frac1t = |\eta|^{-\frac1s} |y|^\frac1t < \ep^{\frac1t}
\end{equation*}
so $x \in \rB_\ep$ which proves that $(y, \eta) \in \Gamma$. 

From \eqref{eq:implicationGamma} we may conclude
\begin{equation}\label{eq:cover1}
\Gamma \cup \Omega = \rr {2d} \setminus 0
\end{equation}
where 
\begin{equation*}
\Omega = \{ (y,\eta) \in \rr {2d} \setminus 0: \, | \eta |^{\frac1s} \leqs C |y|^{\frac1t}  \}
\end{equation*}
for some $C > 0$.

We use \eqref{eq:STFTinverse} for $u \in ( \Sigma_t^s )' (\rr d)$ and $\fy \in \Sigma_t^s (\rr d)$ with $\| \fy \|_{L^2} = 1$, cf. \cite{Toft1}, 
and show that the integral for $\partial^\alpha u$ is absolutely convergent for any $\alpha \in \nn d$. 
Thus we write formally
\begin{equation}\label{eq:STFTreconstruction}
\partial^\alpha u (y) = (2\pi)^{-\frac{d}{2}} \sum_{\beta \leqs \alpha} \binom{\alpha}{\beta} \int_{\rr {2d}} V_\fy u(x,\xi) \, (i\xi)^\beta e^{i \langle \xi,y \rangle} \partial^{\alpha-\beta} \fy (y-x) \, \dd x \, \dd \xi. 
\end{equation}

We will need the estimate for any $r > 0$
\begin{align*}
|\xi|^{\beta} 
& = 
\left( \frac{d s}{r} \right)^{s |\beta|} \beta!^s
\left( 
\frac{ \left( \frac{r}{d s} |\xi|^{\frac1s} \right)^{|\beta|}}{\beta!} 
\right)^s
\leqs
\left( \frac{d s}{r} \right)^{s |\beta|} \beta!^s
\left( 
\frac{ \left( \frac{r}{s} |\xi|^{\frac1s} \right)^{|\beta|}}{|\beta|!} 
\right)^s \\
& \leqs \left( \frac{d s}{r} \right)^{s |\beta|} \beta!^s e^{r |\xi|^{\frac1s}} 
\end{align*}
as well as
\begin{equation}\label{eq:Sstderivativeestimate1}
|D^\beta \fy (x)| \leqs C_{r,h} h^{|\beta|} \beta!^s e^{- r |x|^{\frac1t}}, \quad \beta \in \nn d, \quad x \in \rr d, 
\end{equation}
for any $h, r > 0$. 

In order to prove \eqref{eq:Sstderivativeestimate1} we may use the seminorms \eqref{eq:seminormSigmas} with $h^{|\alpha+\beta|}$ replaced by $h_1^{|\alpha|} h_2^{|\beta|}$ for two different arbitrary $h_1, h_2 > 0$. 
The argument is known but we repeat it for the benefit of the reader. 

If $r > 0$ then we obtain from \eqref{eq:seminormSigmas} for any $h_1, h_2 > 0$
\begin{align*}
e^{\frac{r}{t} |x|^{\frac1t}} |D^\beta \fy (x)|^{\frac1t}
& = \sum_{n=0}^\infty 2^{-n} \left( \frac{\left( \frac{2r}{t} \right)^{tn} }{n!^t} |x|^n \, |D^\beta \fy (x)| \right)^{\frac1t} \\
& \leqs 2 \left( \sup_{n \geqs 0} \frac{\left( \frac{2r}{t} \right)^{t n} d^{\frac{n}{2}}}{n!^t} \max_{|\alpha|=n} |x^{\alpha} D^\beta \fy (x)| \right)^{\frac1t} \\
& \leqs \left( C_{h_1,h_2} h_2^{|\beta|} \beta!^s
\sup_{n \geqs 0} \left( \left( \frac{2r}{t} \right)^{t} d^{\frac12} h_1 \right)^n
\right)^{\frac1t} \\
& \leqs \left( C_{h_2,r} \, h_2^{|\beta|} \beta!^s  \right)^{\frac1t}
\end{align*}
provided $h_1 \leqs \left( \frac{t}{2r} \right)^{t} d^{-\frac12}$. 
We have proved \eqref{eq:Sstderivativeestimate1} for any $h, r > 0$.

We split the integral \eqref{eq:STFTreconstruction} in two parts.
We obtain using \eqref{eq:decayGamma2} for any $r_1, r_2, r_3 > 0$ and $0 < h \leqs 1$
\begin{equation}\label{eq:integralGamma}
\begin{aligned}
& \left| \int_{\Gamma} V_\fy u(x,\xi) \, (i\xi)^\beta e^{i \langle \xi,y \rangle} \partial^{\alpha-\beta} \fy(y-x) \, \dd x \, \dd \xi \right| \\
& \leqslant \int_{\Gamma} |V_\fy u(x,\xi)| \, |\xi|^{|\beta|} \, |\partial^{\alpha-\beta} \fy(y-x)| \, \dd x \, \dd \xi \\
& \lesssim h^{|\alpha-\beta|}  \left( \frac{d s}{r_2} \right)^{s |\beta|} (\alpha-\beta)!^s \beta!^s \int_{\Gamma} e^{-r_1 (|x|^\frac1t +  |\xi|^\frac1s) + r_2 |\xi|^\frac1s  - \kappa(t^{-1}) r_3 |y-x|^\frac1t}  \,  \dd x \, \dd \xi \\
& \lesssim \left( h^{-1} \left( \frac{d s}{r_2} \right)^s \right)^{|\beta|} \alpha!^s e^{ -r_3 |y|^{\frac1t}}
\int_{\rr {2d}} e^{-r_1 (|x|^\frac1t +  |\xi|^\frac1s) + r_2 |\xi|^{\frac1s}  + \kappa(t^{-1}) r_3 |x|^{\frac1t}}  \,  \dd x \, \dd \xi \\
& \lesssim \left( h^{-1} \left( \frac{d s}{r_2} \right)^s \right)^{|\alpha|} \alpha!^s e^{ -r_3 |y|^{\frac1t}}
\end{aligned}
\end{equation}
provided $h \leqs \left( \frac{d s}{r_2} \right)^s$ and $r_1 > \max(r_2, \kappa(t^{-1}) r_3)$.

For the remaining part of the integral we may by \eqref{eq:cover1} assume that $(x,\xi) \in \Omega$. 
Using \eqref{eq:STFTGFstdistr} we obtain for some $r_1 > 0$ and any $r_2, r_3 > 0$ and $0 < h \leqs 1$
\begin{equation}\label{eq:integralOmega}
\begin{aligned}
& \left| \int_{\Omega} V_\fy u(x,\xi) \, (i\xi)^\beta e^{i \langle \xi,y \rangle} \partial^{\alpha-\beta} \fy(y-x) \, \dd x \, \dd \xi \right| \\
& \lesssim \int_{| \xi |^{\frac1s} \leqs C |x|^{\frac1t}} e^{r_1 ( |x|^\frac1t +|\xi|^\frac1s})  |\xi|^{|\beta|} \, |\partial^{\alpha-\beta} \fy(y-x)| \, \dd x \, \dd \xi \\
& \lesssim h^{|\alpha-\beta|} \left( \frac{d s}{r_2} \right)^{s |\beta|} (\alpha-\beta)!^s \beta!^s \int_{| \xi |^{\frac1s} \leqs C |x|^{\frac1t} } e^{ r_1 ( |x|^\frac1t +|\xi|^\frac1s) + r_2 |\xi|^{\frac1s} - \kappa(t^{-1}) r_3 |y-x|^\frac1t } \, \dd x \, \dd \xi \\
& \lesssim 
\left( h^{-1} \left( \frac{d s}{r_2} \right)^s \right)^{|\beta|} \alpha!^s e^{ \kappa(t^{-1}) r_3 |y|^{\frac1t} }
\int_{| \xi |^{\frac1s} \leqs C |x|^{\frac1t}} e^{ r_1 ( |x|^\frac1t +|\xi|^\frac1s) + r_2 |\xi|^{\frac1s} - r_3 |x|^\frac1t } \, \dd x \, \dd \xi \\
& \leqs
\left( h^{-1} \left( \frac{d s}{r_2} \right)^s \right)^{|\alpha|} \alpha!^s e^{ \kappa(t^{-1}) r_3 |y|^{\frac1t} }
\int_{| \xi |^{\frac1s} \leqs C |x|^{\frac1t}} e^{ - |\xi|^\frac1s +  (r_1 - r_3) |x|^\frac1t + (1 + r_1 + r_2) |\xi|^\frac1s } \, \dd x \, \dd \xi \\
& \leqs 
\left( h^{-1} \left( \frac{d s}{r_2} \right)^s \right)^{|\alpha|} \alpha!^s e^{ \kappa(t^{-1}) r_3 |y|^{\frac1t} }
\int_{| \xi |^{\frac1s} \leqs C |x|^{\frac1t}} e^{ - |\xi|^\frac1s +  (r_1 + C (1 + r_1 + r_2) - r_3) |x|^\frac1t  } \, \dd x \, \dd \xi \\
& \lesssim \left( h^{-1} \left( \frac{d s}{r_2} \right)^s \right)^{|\alpha|} \alpha!^s e^{ \kappa(t^{-1}) r_3 |y|^{\frac1t} }
\end{aligned}
\end{equation}
provided $h \leqs \left( \frac{d s}{r_2} \right)^s$ and $r_3 > r_1 + C (1 + r_1 + r_2)$. 

Combining \eqref{eq:integralGamma} and \eqref{eq:integralOmega} shows in view of \eqref{eq:STFTreconstruction} that $u \in C^\infty(\rr d)$ and the estimate \eqref{eq:exponentialbound1} follows. 
\end{proof}

\section{Microlocality}\label{sec:microlocal}

The next result concerns microlocality with respect to $\WF^{t,s}$ of pseudodifferential operators. 

We use a space of smooth symbols originally introduced in \cite[Definition~1.8]{Abdeljawad1} and denoted $\Gamma_{t,s}^{s,t; 0} (\rr {2d})$.
For $s,t > 0$ such that $s + t > 1$, $a \in \Gamma_{t,s}^{s,t; 0} (\rr {2d})$ means that $a \in C^\infty(\rr {2d})$ and 
\begin{equation}\label{eq:symbolvillkor2}
|\partial_{x}^{\alpha} \partial_{\xi}^{\beta} a(x,\xi)| \lesssim h^{|\alpha + \beta|} \alpha!^s \beta!^t e^{\mu (|x|^{\frac1t} + |\xi|^{\frac1s} )}, \quad \alpha, \beta \in \nn d, \quad x, \xi \in \rr d, 
\end{equation}
for some $\mu > 0$ and for all $h > 0$. 
The space $\Gamma_{t,s}^{s,t; 0} (\rr {2d})$ is characterized in \cite[Proposition~2.3]{Abdeljawad1} using the STFT as follows. Let $\Phi \in \Sigma_{t,s}^{s,t}(\rr {2d}) \setminus 0$ be arbitrary. 
Then $a \in \Gamma_{t,s}^{s,t; 0} (\rr {2d})$ if and only if 
\begin{equation}\label{eq:STFTvillkor1}
| V_\Phi a(z_1, z_2,\zeta_1, \zeta_2)  | \lesssim e^{\mu (|z_1|^{\frac1t} + |z_2|^{\frac1s})- b (|\zeta_1|^{\frac1s} + |\zeta_2|^{\frac1t})}, \quad z_1, z_2, \zeta_1, \zeta_2 \in \rr d, 
\end{equation}
for some $\mu > 0$ and all $b > 0$. 

If $a \in \Gamma_{t,s}^{s,t; 0} (\rr {2d})$ then $a^w(x,D): \Sigma_t^s (\rr d) \to \Sigma_t^s (\rr d)$ is continuous
and extends uniquely to a continuous operator $a^w(x,D): (\Sigma_t^s)' (\rr d) \to (\Sigma_t^s)' (\rr d)$ according to \cite[Theorem~3.15]{Abdeljawad1}. 

By the following result
it is also microlocal with respect to the $t,s$-Gelfand--Shilov wave front set. 

\begin{thm}\label{thm:microlocalWFs}
If $s,t > 0$ satisfy $s + t > 1$ 
and $a \in \Gamma_{t,s}^{s,t; 0} (\rr {2d})$ then 
\begin{equation}\label{eq:microlocal1}
\WF^{t,s}( a^w(x,D) u ) \subseteq \WF^{t,s}(u), \quad u \in (\Sigma_t^s)'(\rr d). 
\end{equation}
\end{thm}

\begin{proof}
Pick $\fy \in \Sigma_t^s(\rr d)$ such that $\| \fy \|_{L^2}=1$.
Recall the notation $\Pi(x,\xi) = M_\xi T_x$ for $(x,\xi) \in \rr {2d}$. 
Denoting the formal adjoint of $a^w(x,D)$ by $a^w(x,D)^*$, \eqref{eq:moyal} gives for $u \in (\Sigma_t^s)'(\rr d)$ and $z \in \rr {2d}$
\begin{align*}
(2 \pi)^{\frac{d}{2}}  V_\varphi (a^w(x,D) u) (z)
& = ( a^w(x,D) u, \Pi(z) \varphi ) \\
& = ( u, a^w(x,D)^* \Pi(z) \varphi ) \\
& = \int_{\rr {2d}} V_\varphi u(w) \, ( \Pi(w) \varphi,a^w(x,D)^* \Pi(z) \varphi ) \, \dd w \\
& = \int_{\rr {2d}} V_\varphi u(w) \, ( a^w(x,D) \, \Pi(w) \varphi,\Pi(z) \varphi ) \, \dd w \\
& = \int_{\rr {2d}} V_\varphi u(z-w) \, ( a^w(x,D) \, \Pi(z-w) \varphi,\Pi(z) \varphi ) \, \dd w.
\end{align*}
By e.g. \cite[Lemma 3.1]{Grochenig2}, or a direct computation involving \eqref{eq:wignerweyl}, we have 
\begin{equation*}
|( a^w(x,D) \, \Pi(z-w) \varphi,\Pi(z) \varphi )|
= \left| V_\Phi a \left( z-\frac{w}{2}, \J w \right) \right|
\end{equation*}
where $\Phi$ is the Wigner distribution $\Phi = W(\fy,\fy)$. 

We have $\Phi \in \Sigma_{t,s}^{s,t}(\rr {2d})$. 
In fact we have $\fy \otimes \overline \fy \in \Sigma_{t,t}^{s,s}(\rr {2d})$
and therefore also $(\fy \otimes \overline \fy) \circ \kappa \in \Sigma_{t,t}^{s,s}(\rr {2d})$
where $\kappa(x,y) = (x+y/2, x- y/2)$. 
Since $W(\fy,\fy) = (2 \pi)^{\frac{d}{2}} \cF_2 ( (\fy \otimes \overline \fy) \circ \kappa )$
we obtain from \cite[Proposition~1.1]{Abdeljawad1}
the conclusion $\Phi \in \Sigma_{t,s}^{s,t}(\rr {2d})$. 

Combining the preceding identities we deduce
\begin{align*}
|V_\varphi (a^w(x,D) u) (z)|
& \lesssim \int_{\rr {2d}}  |V_\varphi u(z-w)| \, \left| V_\Phi a \left( z-\frac{w}{2}, \J w \right) \right| \, \dd w.
\end{align*}

Suppose $z_0 \in \rr {2d} \setminus 0$ and $z_0 \notin \WF^{t,s}(u)$. 
There exists an open set $V$ such that $z_0 \in V$ and 
\eqref{eq:notinWFGFst2} holds. 
We pick an open set $U$ such that 
$z_0 \in U$ and $U + \rB_\ep \subseteq V$ for some $0 < \ep \leqs 1$, and we may assume 
\begin{equation}\label{eq:Ubound}
\sup_{z \in U} |z| \leqs |z_0| + 1 := \alpha. 
\end{equation}

Let $r > 0$ and $\lambda > 0$. We have
\begin{align*}
& e^{r \lambda} |V_\fy (a^w(x,D) u)  (\lambda^t x, \lambda^s \xi)| \\
& \lesssim \iint_{\rr {2d}} e^{r \lambda } |V_\fy u ( \lambda^t (x- \lambda^{-t} y), \lambda^s (\xi - \lambda^{-s} \eta))| \, \left| V_\Phi a \left( \lambda^t x-\frac{y}{2}, \lambda^s \xi-\frac{\eta}{2},  \eta, -y \right) \right| \, \dd y \, \dd \eta \\
& = I_1 + I_2
\end{align*}
where we split the integral into the two terms 
\begin{align*}
I_1 = & \iint_{\rr {2d} \setminus \Omega_\lambda} e^{r \lambda} |V_\fy u ( \lambda^t (x- \lambda^{-t} y), \lambda^s (\xi - \lambda^{-s} \eta))| \, \left| V_\Phi a \left( \lambda^t x-\frac{y}{2}, \lambda^s \xi-\frac{\eta}{2},  \eta, -y \right) \right| \, \dd y \, \dd \eta, \\
I_2 = & \iint_{\Omega_\lambda} e^{r \lambda} |V_\fy u ( \lambda^t (x- \lambda^{-t} y), \lambda^s (\xi - \lambda^{-s} \eta))| \, \left| V_\Phi a \left( \lambda^t x-\frac{y}{2}, \lambda^s \xi-\frac{\eta}{2},  \eta, -y \right) \right| \, \dd y \, \dd \eta
\end{align*}
where  
\begin{equation*}
\Omega_\lambda = \{(y,\eta) \in \rr {2d}: |y|^{\frac1t} + |\eta|^{\frac1s} < 2^{-\frac{1}{2v}} \ep^{\frac1v} \lambda \} 
\end{equation*}
with $v = \min(s,t)$. 

First we estimate $I_1$ when $(x,\xi) \in U$. 
Set $\kappa = \max(\kappa(t^{-1}), \kappa(s^{-1}))$.
From 
\eqref{eq:STFTGFstdistr}, \eqref{eq:STFTvillkor1} and \eqref{eq:Ubound}
we obtain for some $r_1, \mu > 0$ and any $b > 0$
\begin{equation}\label{eq:estimateI1b}
\begin{aligned}
I_1 
& \lesssim e^{r \lambda } \iint_{\rr {2d} \setminus \Omega_\lambda} e^{r_1 \lambda |x- \lambda^{-t} y|^{\frac1t} + r_1 \lambda |\xi- \lambda^{-s} \eta|^{\frac1s}}  \, \left| V_\Phi a \left( \lambda^t x-\frac{y}{2}, \lambda^s \xi-\frac{\eta}{2},  \eta, -y \right) \right| \, \dd y \, \dd \eta \\
& \leqs e^{ r \lambda+ \kappa r_1 \lambda |x|^{\frac1t} + \kappa r_1 \lambda |\xi|^{\frac1s} } \iint_{\rr {2d} \setminus \Omega_\lambda} e^{ r_1 \kappa (| y|^{\frac1t} + | \eta|^{\frac1s} )}  \, \left| V_\Phi a \left( \lambda^t x-\frac{y}{2}, \lambda^s \xi-\frac{\eta}{2},  \eta, -y \right) \right| \, \dd y \, \dd \eta \\
& \lesssim e^{ r \lambda+ r_1 \lambda \kappa (\alpha^{\frac1t} + \alpha^{\frac1s} ) } 
\iint_{\rr {2d} \setminus \Omega_\lambda} 
e^{ r_1 \kappa (| y|^{\frac1t} + | \eta|^{\frac1s})   
+ \mu \left( \left| \lambda^t x-\frac{y}{2}\right|^{\frac1t} + \left| \lambda^s \xi-\frac{\eta}{2}\right|^{\frac1s} \right)  
- (b+1) \left( \left| \eta \right|^{\frac1s} + \left| y \right|^{\frac1t} \right) }   \, \dd y \, \dd \eta \\
& \lesssim e^{ \lambda \left( r + (r_1+\mu) \kappa (\alpha^{\frac1t} + \alpha^{\frac1s}) \right)} \iint_{\rr {2d} \setminus \Omega_\lambda} e^{ \kappa (r_1 + 2^{-\frac1t} \mu - b) | y|^{\frac1t} + \kappa ( r_1 + 2^{-\frac1s} \mu - b) | \eta|^{\frac1s}   
-  \left( \left| \eta \right|^{\frac1s} + \left| y \right|^{\frac1t} \right) }   \, \dd y \, \dd \eta \\
& \leqs e^{ \lambda \left( r + 2 (r_1+\mu) \kappa \alpha^{\frac1v} \right)} \iint_{\rr {2d} \setminus \Omega_\lambda} e^{ \kappa (r_1 + \mu - b) (| y|^{\frac1t} +  | \eta|^{\frac1s} )   
-  \left( \left| \eta \right|^{\frac1s} + \left| y \right|^{\frac1t} \right) }   \, \dd y \, \dd \eta \\
& \leqs e^{ \lambda \left( r + 2 (r_1+\mu) \kappa \alpha^{\frac1v} 
+ \kappa (r_1 + \mu - b) 2^{-\frac{1}{2v}} \ep^{\frac1v} \right)} 
\iint_{\rr {2d}} e^{-\left( \left| \eta \right|^{\frac1t} + \left| y \right|^{\frac1s} \right) }   \, \dd y \, \dd \eta \\
& \lesssim e^{ \lambda \left( r + 2 (r_1+\mu) \kappa \alpha^{\frac1v}
+ \kappa (r_1 + \mu - b) 2^{-\frac{1}{2v}} \ep^{\frac1v} \right)} 
\leqs C_{r}
\end{aligned}
\end{equation}
for any $\lambda > 0$, 
provided we pick $b \geqs r_1 + \mu + \kappa^{-1} 2^{\frac{1}{2v}} \ep^{- \frac1v} \left( r + 2 (r_1+\mu) \kappa \alpha^{\frac1v} \right)$. 
Here $C_{r} > 0$ is a constant that depends on $r > 0$ but not on $\lambda > 0$. 
Thus we have obtained the requested estimate for $I_1$. 

It remains to estimate $I_2$. 
From $|y|^{\frac1t} + |\eta|^{\frac1s} < 2^{-\frac{1}{2v}} \ep^{\frac1v} \lambda$ 
we obtain
\begin{align*}
& \lambda^{-t} |y| < \ep^{\frac{t}{v}} \, 2^{-\frac{t}{2v}} \leqs \ep \, 2^{-\frac{1}{2}}, \\
& \lambda^{-s} |\eta| < \ep^{\frac{s}{v}} \, 2^{-\frac{s}{2v}} \leqs \ep \, 2^{-\frac{1}{2}}
\end{align*}
which gives $(\lambda^{-t} y, \lambda^{-s} \eta ) \in \rB_{\ep}$.  
Hence if $(x,\xi) \in U$ then $( x- \lambda^{-t} y, \xi - \lambda^{-s} \eta) \in V$ and we may use the estimate 
\eqref{eq:notinWFGFst2}. 
This gives for some $\mu > 0$, any $b > 0$ and a constant $C_r = C_{r,\mu,s,t} > 0$, using 
\eqref{eq:STFTvillkor1} and \eqref{eq:Ubound}
\begin{equation}\label{eq:estimateI2b}
\begin{aligned}
I_2 & = \iint_{\Omega_\lambda} e^{r \lambda} |V_\fy u ( \lambda^t (x- \lambda^{-t} y), \lambda^s (\xi - \lambda^{-s} \eta))| \, \left| V_\Phi a \left( \lambda^t x-\frac{y}{2}, \lambda^s \xi-\frac{\eta}{2},  \eta, -y \right) \right| \, \dd y \, \dd \eta \\
& = e^{-\lambda \kappa \mu 2 \alpha^{\frac1v} } \\
& \quad \times \iint_{\Omega_\lambda} e^{ (r+ \kappa \mu 2 \alpha^{\frac1v} ) \lambda} |V_\fy u ( \lambda^t (x- \lambda^{-t} y), \lambda^s (\xi - \lambda^{-s} \eta))| \, \left| V_\Phi a \left( \lambda^t x-\frac{y}{2}, \lambda^s \xi-\frac{\eta}{2},  \eta, -y \right) \right| \, \dd y \, \dd \eta \\
& \leqs C_r e^{-\lambda \kappa \mu 2 \alpha^{\frac1v}} \iint_{\Omega_\lambda} \left| V_\Phi a \left( \lambda^t x-\frac{y}{2}, \lambda^s \xi-\frac{\eta}{2},  \eta, -y \right) \right| \, \dd y \, \dd \eta \\
& \lesssim C_r e^{- \lambda \kappa \mu 2 \alpha^{\frac1v}} 
\iint_{\rr {2d}} 
e^{ \mu \left( \left| \lambda^t x-\frac{y}{2}\right|^{\frac1t} + \left| \lambda^s \xi-\frac{\eta}{2}\right|^{\frac1s} \right)  
- b \left( \left| \eta \right|^{\frac1s} + \left| y \right|^{\frac1t} \right) }
\, \dd y \, \dd \eta \\
& \leqs C_r e^{-\lambda \kappa \mu 2 \alpha^{\frac1v} + \lambda \kappa \mu 2 \alpha^{\frac1v}}
\iint_{\rr {2d}} 
e^{ \mu \kappa \left( \left|y\right|^{\frac1t} + \left| \eta \right|^{\frac1s} \right)  
- b \left( \left| \eta \right|^{\frac1s} + \left| y \right|^{\frac1t} \right) }
\, \dd y \, \dd \eta \\
& = C_r 
\iint_{\rr {2d}} 
e^{  (\kappa \mu- b) \left( \left| \eta \right|^{\frac1s} + \left| y \right|^{\frac1t} \right) }
\, \dd y \, \dd \eta \\
& \lesssim C_r
\end{aligned}
\end{equation}
provided $b > \kappa \mu$, 
for all $\lambda > 0$. 
Thus we have obtained the requested estimate for $I_2$. 
Combining \eqref{eq:estimateI1b} and \eqref{eq:estimateI2b}
we may conclude that $z_0 \notin \WF^{s,t}( a^w(x,D) u )$
and hence we have proved \eqref{eq:microlocal1}.  
\end{proof}

As a corollary we obtain the following generalization of \cite[Proposition~4.10]{Carypis1}.
Here we use a space of smooth symbols originally introduced in \cite[Definition~2.4]{Cappiello2} and denoted $\Gamma_{0,s}^\infty(\rr {2d})$, and which is identical to $\Gamma_{s,s}^{s,s; 0}(\rr {2d})$. 
For $s > \frac12$, $a \in \Gamma_{0,s}^\infty(\rr {2d})$ means that $a \in C^\infty(\rr {2d})$ and 
\begin{equation}\label{eq:symbolvillkor1}
|\pd \alpha a(z)| \lesssim h^{|\alpha|} \alpha!^s e^{\mu |z|^{\frac1s}}, \quad \alpha \in \nn {2d}, \quad z \in \rr {2d}, 
\end{equation}
for some $\mu > 0$ and for all $h > 0$. 
The space $\Gamma_{0,s}^\infty(\rr {2d})$ is characterized in \cite[Proposition~3.2]{Cappiello2} using the STFT as follows. Let $\Phi \in \Sigma_s(\rr {2d}) \setminus 0$ be arbitrary. 
Then $a \in \Gamma_{0,s}^\infty(\rr {2d})$ if and only if 
\begin{equation}\label{eq:STFTvillkor2}
| V_\Phi a(z,\zeta)  | \lesssim e^{\mu |z|^{\frac1s} - b |\zeta|^{\frac1s}}, \quad z,\zeta \in \rr {2d}, 
\end{equation}
for some $\mu > 0$ and all $b > 0$. 

If $a \in \Gamma_{0,s}^\infty(\rr {2d})$ then $a^w(x,D): \Sigma_s (\rr d) \to \Sigma_s (\rr d)$ is continuous
and extends uniquely to a continuous operator $a^w(x,D): \Sigma_s' (\rr d) \to \Sigma_s' (\rr d)$ according to \cite[Proposition~4.10]{Cappiello2}. 

\begin{cor}\label{cor:microlocalWFs}
If $s > \frac12$ and $a \in \Gamma_{0,s}^\infty(\rr {2d})$ then 
\begin{equation*}
\WF^s ( a^w(x,D) u ) \subseteq \WF^s (u), \quad u \in \Sigma_s'(\rr d). 
\end{equation*}
\end{cor}

\begin{rem}\label{rem:microlocal}
It is interesting to compare the assumption $a \in \Gamma_{0,s}^\infty (\rr {2d})$, which is equivalent to 
the STFT estimates 
\begin{equation}\label{eq:STFTvillkor1b}
| V_\Phi a(z,\zeta)  | \lesssim e^{\mu |z|^{\frac1s} - b |\zeta|^{\frac1s}}
\end{equation}
for some $\mu > 0$ and all $b > 0$, 
with the estimates 
\begin{equation}\label{eq:STFTvillkor2}
|V_\Phi a (z,\zeta) | \lesssim e^{\frac{b}{4} |z|^{\frac1s} - b |\zeta|^{\frac1s}}
\end{equation}
for all $b > 0$. 

Condition \eqref{eq:STFTvillkor2} for all $b > 0$ has been shown to imply continuity $a^w(x,D): \Sigma_s(\rr d) \to \Sigma_s(\rr d)$ \cite[Lemma~6.5 and Proposition~6.6]{Wahlberg3}, but it does not imply microlocality with respect to $\WF^s$. 
In fact microlocality for operators of this type is contradicted by \cite[p.~556]{Carypis1} with $Q = i I_{2d}$ and $t \notin \pi \zo$. 
\end{rem}

The next result is another consequence of Theorem \ref{thm:microlocalWFs}.

\begin{cor}\label{cor:translationmodulationinvar}
Suppose $s,t > 0$ satisfy $s + t > 1$. 
For any $z \in \rr {2d}$ and any $u \in (\Sigma_t^s)'(\rr d)$ we have
\begin{equation*}
\WF^{t,s}( \Pi(z) u ) = \WF^{t,s}(u). 
\end{equation*}
\end{cor}

\begin{proof}
By a calculation it is verified that $\Pi(x,\xi) = a_{x,\xi}^w(x,D)$ where
\begin{equation*}
a_{x,\xi} (y,\eta) = e^{ \frac{i}{2} \la x, \xi \ra + i \left( \la y, \xi \ra - \la x, \eta \ra \right)}, \quad (y,\eta) \in \rr {2d}. 
\end{equation*}

Using \eqref{eq:expestimate0}
we may estimate
\begin{align*}
\left| \partial_y^\alpha \partial_\eta^\beta a_{x,\xi} (y,\eta) \right|
& = |\xi^\alpha x^\beta| 
\leqs e^{s d h^{- \frac{1}{s}} + t d h^{- \frac{1}{t}}}  ( |(x,\xi)| h)^{|\alpha+\beta|} \alpha!^{s} \beta!^{t} \\
& = C_{t,s,h,d} ( |(x,\xi)| h)^{|\alpha+\beta|} \alpha!^{s} \beta!^{t}
\end{align*}
for any $h > 0$ and $\alpha, \beta \in \nn d$. 
This implies that $a_{x,\xi} \in \Gamma_{t,s}^{s,t; 0}(\rr {2d})$. 
Thus we may apply Theorem \ref{thm:microlocalWFs} which gives 
\begin{equation*}
\WF^{t,s}( \Pi(z) u ) \subseteq \WF^{t,s}(u). 
\end{equation*}
The opposite inclusion follows from $u = e^{- i \la x, \xi \ra} \Pi(-(x,\xi)) \Pi(x,\xi) u$. 
\end{proof}

\section{Global wave front sets of polynomials and generalizations}\label{sec:polynomials}

\begin{prop}\label{prop:WFstelementary}
If $s,t > 0$ satisfy $s + t > 1$ then:

\begin{enumerate}[\rm (i)]

\item for any $x \in \rr d$ and any $\alpha \in \nn d$
\begin{equation}\label{eq:diracWFst}
\WFg ( \pd \alpha \delta_x ) = \WF^{t,s}( \pd \alpha \delta_x ) =  \{ 0 \} \times ( \rr d \setminus 0 ); 
\end{equation}
\item for any $\alpha \in \nn d$
\begin{equation}\label{eq:monomialWFst}
\WFg ( x^\alpha ) = \WF^{t,s}( x^\alpha ) =  ( \rr d \setminus 0 ) \times \{ 0 \}; 
\end{equation}
\item for any $\xi \in \rr d$
\begin{equation}\label{eq:planewaveWFst}
\WFg (e^{i \la \cdot, \xi \ra} ) = \WF^{t,s}( e^{i \la \cdot, \xi \ra} ) =  ( \rr d \setminus 0 ) \times \{ 0 \}. 
\end{equation}
\end{enumerate}

\end{prop}

Proposition \ref{prop:WFstelementary} follows from the arguments in Section \ref{sec:GelfandShilovWF}, the details of the proof are left to the reader. 
We fix attention on the following generalizations 
of Proposition \ref{prop:WFstelementary}. 

Consider a polynomial on $\rr d$
\begin{equation}\label{eq:polynomial1}
p(x) = \sum_{\alpha \in \nn d, \, |\alpha| \leqs m} c _\alpha x^{\alpha}, \quad x \in \rr d, 
\end{equation}
with $c_\alpha \in \co$ and $m \in \no \setminus 0$. 

\begin{prop}\label{prop:polynomial}
Suppose $s,t > 0$ satisfy $s + t  > 1$, 
let $p$ be the polynomial \eqref{eq:polynomial1}
and define
\begin{equation*}
u = \sum_{\alpha \in \nn d, \, |\alpha| \leqs m} c _\alpha D^\alpha \delta_0 \in \cS'(\rr d).
\end{equation*}
Then 
\begin{equation}\label{eq:diracpolynomial1}
\WFg( u ) 
= \WF^{t,s}( u ) 
= \{ 0 \} \times ( \rr d \setminus 0 ) 
\end{equation}
and
\begin{equation}\label{eq:polynomial2}
\WFg( p ) 
= \WF^{t,s}( p ) 
= ( \rr d \setminus 0 ) \times \{ 0 \}. 
\end{equation}
\end{prop}

\begin{proof}
Fourier transformation gives $\widehat u = (2 \pi)^{- \frac{d}{2} } p$ so \eqref{eq:polynomial2} is a consequence of \eqref{eq:diracpolynomial1} and the Fourier invariances \eqref{eq:metaplecticWFG} and Proposition \ref{prop:WFstsymplectic} (i). 
Thus it suffices to show \eqref{eq:diracpolynomial1}. 

From Proposition \ref{prop:WFstelementary} (i) and \eqref{eq:WFstsublinear} we obtain
\begin{equation*}
\WF^{t,s}( u ) 
\subseteq \{ 0 \} \times ( \rr d \setminus 0 )
\end{equation*}
and \eqref{eq:GaborGSinclusion} gives
\begin{equation*}
\WFg( u ) 
\subseteq \WF^{v,v}( u ) 
\subseteq \{ 0 \} \times ( \rr d \setminus 0 )
\end{equation*}
where $v = \max(t,s) > \frac12$. 
Hence it suffices to show 
\begin{equation}\label{eq:WFstdiracsum}
\{ 0 \} \times ( \rr d \setminus 0 ) \subseteq \WF^{t,s}( u ) 
\end{equation}
and 
\begin{equation}\label{eq:WFgdiracsum}
\{ 0 \} \times ( \rr d \setminus 0 ) \subseteq \WFg ( u ). 
\end{equation}

Let $\fy \in \Sigma_t^s (\rr d) \setminus 0$ satisfy $\fy(0) \neq 0$. 
We have
\begin{align*} 
& V_\fy u (0,\xi) \\
& = (2 \pi)^{- \frac{d}{2}} \sum_{|\alpha| \leqs m} c _\alpha \sum_{\beta \leqs \alpha} \binom{\alpha}{\beta} \xi^\beta \overline{D^{\alpha-\beta} \fy(0)} \\
& = (2 \pi)^{- \frac{d}{2}} \left( \sum_{|\alpha| = m}  c _\alpha \xi^\alpha \overline{\fy (0)}  
+ \sum_{|\alpha| = m}  c _\alpha \sum_{\beta < \alpha} \binom{\alpha}{\beta} \xi^\beta \overline{ D^{\alpha-\beta} \fy(0)}
+ \sum_{|\alpha| < m}  c _\alpha \sum_{\beta \leqs \alpha} \binom{\alpha}{\beta} \xi^\beta \overline{ D^{\alpha-\beta} \fy(0)}
\right). 
\end{align*}

Define the principal part of $p$ as 
\begin{equation*}
p_m(x) = \sum_{|\alpha| = m} c _\alpha x^{\alpha}. 
\end{equation*}

If $\xi \in \rr d \setminus 0$, $p_m(\xi) \neq 0$ and $\lambda > 0$ then
\begin{align*} 
& (2 \pi)^{\frac{d}{2}} V_\fy u (0,\lambda^s \xi) \\
& = \lambda^{s m} p_m(\xi) \overline{\fy (0)}  
+ \underbrace{\sum_{|\alpha| = m}  c _\alpha \sum_{\beta < \alpha} \binom{\alpha}{\beta} \lambda^{s |\beta|} \xi^\beta \overline{ D^{\alpha-\beta} \fy(0)}
+ \sum_{|\alpha| < m}  c _\alpha \sum_{\beta \leqs \alpha} \binom{\alpha}{\beta} \lambda^{s |\beta|} \xi^\beta  \overline{ D^{\alpha-\beta} \fy(0)}}_{:=R}. 
\end{align*}
Since $R$ contains terms $\lambda^{s k}$ where $k < m$
this implies that $(0,\xi) \in \WFg( u)$ and $(0,\xi) \in \WF^{t,s}( u)$. 

If instead $\xi \in \rr d \setminus 0$ and $p_m(\xi) = 0$, then for any $\ep > 0$ the ball $\rB_\ep(\xi)$ contains $\eta \in \rr d \setminus 0$ such that $p_m(\eta) \neq 0$. 
In fact $p_m$ extends to an entire function on $\cc d$ whose zeros are isolated. 
From the argument above it follows that $(0,\eta) \in \WFg( u)$ and $(0,\eta) \in \WF^{t,s}( u)$. 
It follows that 
$(0,\xi) \in \WFg( u)$ and $(0,\xi) \in \WF^{t,s}( u)$. 
We have now shown \eqref{eq:WFstdiracsum} and \eqref{eq:WFgdiracsum}. 
\end{proof}

In order to generalize Propositions \ref{prop:WFstelementary} and \ref{prop:polynomial} we would like to study series of the form
\begin{equation}\label{eq:diracseries1}
u = \sum_{\alpha \in \nn d} c _\alpha D^\alpha \delta_0 
\end{equation}
containing infinitely many nonzero terms $c_\alpha \in \co$, 
and the corresponding power series
\begin{equation}\label{eq:maclaurinseries1}
f(x) = \sum_{\alpha \in \nn d} c _\alpha x^\alpha,
\end{equation}
under suitable hypotheses on the coefficients $c_\alpha \in \co$.

First we note that $u \notin \cS'(\rr d)$. 
In fact we have for $\fy \in \cS(\rr d)$
\begin{equation}\label{eq:series1}
(u,\fy) = \sum_{\alpha \in \nn d} c _\alpha i^{|\alpha|} \pd \alpha \overline{\fy (0)}
\end{equation}
and it is known that a smooth function $\fy$ may have arbitrary growth of $\alpha \mapsto \pd \alpha \fy (0)$ 
(Borel's lemma \cite[Theorem~1.2.6]{Hormander0}). 
Thus the sum \eqref{eq:series1} is not guaranteed to converge for $\fy \in \cS(\rr d)$, unless the series is finite. 
The series \eqref{eq:diracseries1} does not converge in $\cS'(\rr d)$, and $u \notin \cS'(\rr d)$ if the series is infinite. 
For the same reason \eqref{eq:maclaurinseries1} does not converge in $\cS'(\rr d)$, and $f \notin \cS'(\rr d)$. 
(Note that $\wh u = (2 \pi)^{-\frac{d}{2}} f$ when the series is finite.)

Nevertheless it is possible to state conditions on $\{ c_\alpha \}_{\alpha \in \nn d}$ that are sufficient for $u \in (\Sigma_t^s)'(\rr d)$. 
Suppose $s > 0$ and 
\begin{equation}\label{eq:seriescondition1}
\sum_{\alpha \in \nn d} |c _\alpha | \, r^{|\alpha|} \alpha!^s < \infty
\end{equation}
for some $r > 0$. 
Then for $t > 0$ such that $s + t > 1$, and $\fy \in \Sigma_t^s(\rr d)$, we have 
\begin{align*}
|(u,\fy)| \leqs \sum_{\alpha \in \nn d} |c _\alpha | \, |\pd \alpha \fy (0)|
\leqs \nm \fy{\mathcal S_{t,h}^s} \sum_{\alpha \in \nn d} |c _\alpha | \alpha!^s h^{|\alpha|}
\lesssim \nm \fy{\mathcal S_{t,h}^s}
\end{align*}
provided $h \leqs r$. Thus the series \eqref{eq:diracseries1} converges in $(\Sigma_t^s)'(\rr d)$ and $u \in (\Sigma_t^s)'(\rr d)$.  
We may also conclude that \eqref{eq:maclaurinseries1} converges in $(\Sigma_s^t)'(\rr d)$, $f \in (\Sigma_s^t)'(\rr d)$, and 
the Fourier transform acts termwise as $\widehat u = (2 \pi)^{-\frac{d}{2}} f \in (\Sigma_s^t)'(\rr d)$.

We may distinguish two rather different situations 
under condition \eqref{eq:seriescondition1}. 
Namely, if $s > 1$ then $u \in \cE_s'(\rr d)$, with support in the origin, cf. \cite[Example~1.5.3 and 1.6.5]{Rodino1}. 
The absolutely convergent series $f$ satisfies 
\begin{equation*}
|f(x)| \lesssim e^{a |x|^{\frac1s}}, \quad x \in \rr d, 
\end{equation*}
for some $a > 0$ in $\rr d$, cf. \eqref{eq:PWSGevrey}, 
and more precise bounds in $\cc d$ can be deduced from the Paley--Wiener--Schwartz theorem in $\cE_s'(\rr d)$,
cf. \cite[Theorem~1.6.7]{Rodino1}, \cite{Komatsu1,Sobak1}. 

If instead $0 < s \leqs 1$
the series \eqref{eq:maclaurinseries1} also converges absolutely for any $x \in \rr d$, 
and is an entire function. 
In fact
\begin{align*}
\sum_{\alpha \in \nn d} \left| c _\alpha x^\alpha \right|
& \leqs \sum_{\alpha \in \nn d} | c _\alpha| \, r^{|\alpha|} \alpha!^s \left(\frac{ (r^{-1} |x|)^{\frac{|\alpha|}{s}} }{\alpha!} \right)^s \\
& \leqs \sum_{\alpha \in \nn d} | c _\alpha| \, r^{|\alpha|} \alpha!^s \left(\frac{ \left( d (r^{-1} |x|)^{\frac{1}{s}} \right)^{|\alpha|}}{|\alpha|!} \right)^s \\
& \leqs e^{  s d r^{-\frac1s} |x|^{\frac1s} } \sum_{\alpha \in \nn d} | c _\alpha| \, r^{|\alpha|} \alpha!^s \\
& \lesssim e^{  s d r^{-\frac1s} |x|^{\frac1s} }
\end{align*}
which also reveals the growth bound 
\begin{equation*}
|f(x)| \lesssim e^{  s d r^{-\frac1s} |x|^{\frac1s} }, \quad x \in \rr d. 
\end{equation*}

But the definition of support of $u \in (\Sigma_t^s)'(\rr d)$ breaks down if $s \leqs 1$. 
Consider as an example for $z \in \cc d$
\begin{equation*}
u = \sum_{\alpha \in \nn d} \frac{(-\overline{z})^\alpha}{\alpha!} D^\alpha \delta_0. 
\end{equation*}
Condition \eqref{eq:seriescondition1} is satisfied if $r < |z|^{-1}$, and thus $u \in (\Sigma_t^s)'(\rr d)$. 
The corresponding test functions $\fy \in \Sigma_t^s(\rr d)$ extend to entire functions on $\cc d$. 
From Maclaurin expansion we have 
\begin{equation*}
(u, \fy) 
= \sum_{\alpha \in \nn d} \frac{(-\overline{z})^\alpha}{\alpha!} \overline{D^\alpha \fy (0)}
= \overline{ \sum_{\alpha \in \nn d} \frac{(iz)^\alpha}{\alpha!} \pd \alpha \fy (0)}
= \overline{\fy(iz)}.
\end{equation*}
Thus $u$ may be regarded as a delta distribution at the point $i z \in \cc d$. 

In the following result we require that \eqref{eq:seriescondition1} holds for all $r > 0$ 
which precludes the preceding example.

\begin{prop}\label{prop:diracseries1}
Let $s,t > 0$ satisfy $s + t > 1$, 
suppose that \eqref{eq:seriescondition1} holds for all $r > 0$, 
and define $u \in (\Sigma_t^s)'(\rr d)$ and $f \in (\Sigma_s^t)'(\rr d)$ by \eqref{eq:diracseries1} and \eqref{eq:maclaurinseries1} respectively. 
Then 
\begin{equation}\label{eq:diracseries2}
\WF^{t,s}( u ) 
\subseteq \{ 0 \} \times ( \rr d \setminus 0 )
\end{equation}
and
\begin{equation}\label{eq:taylorseries2}
\WF^{s,t}( f ) 
\subseteq ( \rr d \setminus 0 ) \times \{ 0 \}. 
\end{equation}
\end{prop}

\begin{proof}
Since $\widehat u = (2 \pi)^{-\frac{d}{2}} f \in (\Sigma_s^t)'(\rr d)$ it again suffices to show \eqref{eq:diracseries2} by the Fourier invariance Proposition \ref{prop:WFstsymplectic} (i). 
If $s > 1$ the result follows from Proposition \ref{prop:WFsfreqaxis}, cf. \eqref{eq:subinclusion1}.  
Consider the general case $s > 0$. 

Let $\fy \in \Sigma_t^s(\rr d) \setminus 0$,  
let $(x_0,\xi_0) \in T^* \rr d \setminus 0$ satisfy $x_0 \neq 0$, and let $(x_0,\xi_0) \in U$ where $U \subseteq \rr {2d}$ is open and satisfies 
\begin{equation*}
\sup_{(x,\xi) \in U} |\xi| \leqs |\xi_0| + 1 := a, \quad \inf_{(x,\xi) \in U} |x| \geqs \ep > 0. 
\end{equation*}
If $(x,\xi) \in U$ then we obtain,
using the estimates \eqref{eq:Sstderivativeestimate1}, 
for any $h, r, \lambda > 0$
\begin{align*}
(2 \pi)^{\frac{d}{2}} |V_\fy u( \lambda^t x, \lambda^s \xi )|
& = \left| \sum_{\alpha \in \nn d} c _\alpha  \sum_{\beta \leqs \alpha} \binom{\alpha}{\beta} \lambda^{s |\beta|} \xi^\beta \overline{D^{\alpha-\beta} \fy(- \lambda^t x)} \right| \\
& \leqs C_{r,h} \sum_{\alpha \in \nn d} | c _\alpha | \sum_{\beta \leqs \alpha} \binom{\alpha}{\beta} \lambda^{s |\beta|} |\xi|^{|\beta|} 
h^{|\alpha-\beta|} (\alpha-\beta)!^s e^{- 2 r \ep^{-\frac1t}  \lambda |x|^{\frac1t}} \\
& \leqs C_{r,h} e^{- 2 r \lambda} \sum_{\alpha \in \nn d} | c _\alpha | \, h^{|\alpha|} \alpha!^s \sum_{\beta \leqs \alpha} \binom{\alpha}{\beta} a^{|\beta|} 
\left( \frac{ \left( \lambda h^{-\frac1s} \right)^{|\beta|}}{\beta!} \right)^{s}  \\
& \leqs C_{r,h} e^{- 2 r \lambda} \sum_{\alpha \in \nn d} | c _\alpha | h^{|\alpha|} \alpha!^s \sum_{\beta \leqs \alpha} \binom{\alpha}{\beta} a^{|\beta|} 
\left( \frac{ \left( d \lambda h^{-\frac1s} \right)^{|\beta|}}{|\beta|!} \right)^{s}  \\
& \leqs C_{r,h} e^{- 2 \lambda r  + \lambda s d h^{-\frac1s} } \sum_{\alpha \in \nn d} | c _\alpha | h^{|\alpha|} \alpha!^s \sum_{\beta \leqs \alpha} \binom{\alpha}{\beta} a^{|\beta|}  \\
& = C_{r,h} e^{- 2 \lambda r  + \lambda s d h^{-\frac1s} } \sum_{\alpha \in \nn d} | c _\alpha | ( (a+1) h) ^{|\alpha|} \alpha!^s. 
\end{align*}

If we pick $h = r^{-s} s^{s} d^{s}$ and use \eqref{eq:seriescondition1} then 
\begin{align*}
|V_\fy u( \lambda^t x, \lambda^s \xi )|
& \leqs C_{r} e^{- \lambda r } \sum_{\alpha \in \nn d} | c _\alpha | ( (a+1) h) ^{|\alpha|} \alpha!^s \\
& \leqs C_{r}' e^{- \lambda r }
\end{align*}
for a new constant $C_r' > 0$. Since $(x,\xi) \in U$ and $r > 0$ are arbitrary we have shown $(x_0,\xi_0) \notin \WF^{t,s} (u)$ which proves \eqref{eq:diracseries2}. 
\end{proof}

\begin{rem}\label{rem:suffequality1}
In dimension $d=1$ we can state conditions that are sufficient for equality in \eqref{eq:diracseries2} and 
\eqref{eq:taylorseries2}. 
In fact suppose 
\begin{equation*}
u = \sum_{k = 0}^\infty c _k D^k \delta_0 
\end{equation*}
where \eqref{eq:seriescondition1} is satisfied for all $r > 0$, and 
either $c_{2k} = 0$ for all $k \geqs 0$ or $c_{2k+1} = 0$ for all $k \geqs 0$. 
Then for $\fy \in \Sigma_t^s(\ro)$ 
\begin{equation*}
(\check u, \fy) 
= \sum_{k = 0}^\infty c _k ( D^k \delta_0, \check \fy)
= \sum_{k = 0}^\infty c _k (-1)^k ( D^k \delta_0, \fy)
= \pm (u,\fy)
\end{equation*}
which means that $u$ is either even or odd. 
By \eqref{eq:evensteven1} we have $\WF^{t,s}(u) = - \WF^{t,s}(u)$, and since $\WF^{t,s}(u) \neq \emptyset$ due to $u \notin \Sigma_t^s(\ro)$, we must have 
\begin{equation*}
\WF^{t,s}( u ) 
= \{ 0 \} \times ( \ro \setminus 0 ). 
\end{equation*}
Equality in \eqref{eq:taylorseries2} follows. 
\end{rem}

We can also get equalities for $\WF^{t,s}(u)$ and $\WF^{s,t}(f)$ in terms of the subset 
$V_s(u)$ defined in \eqref{eq:frequencydecay0}. 
Using $\widehat u = (2 \pi)^{-\frac{d}{2}} f \in (\Sigma_s^t)'(\rr d)$ we may rephrase \eqref{eq:frequencydecay0} as follows: 
$x_0 \in \rr d \setminus 0$ satisfies $x_0 \notin V_s(u)$ if there exists an open set $U \subseteq \rr d \setminus 0$ such that $x_0 \in U$ and 
\begin{equation}\label{eq:Scomplement}
\sup_{x \in U, \ \lambda > 0} e^{ r \lambda } |f( \lambda^s x)| < \infty \quad \forall r > 0.
\end{equation}

Thus $V_s(u)$ consists of the directions in $\rr d \setminus 0$ in which $\wh u( x)$ does not decay like $e^{- r |x|^{\frac1s}}$ for all $r > 0$. 
Note that we assume $s > 1$ in the following result. This depends on the fact that we need a window function with certain properties. 

\begin{prop}\label{prop:diracseries2}
Let $s > 1$ and $t > 0$. 
Suppose that \eqref{eq:seriescondition1} holds for all $r > 0$
and define $u \in (\Sigma_t^s)'(\rr d)$ and $f \in (\Sigma_s^t)'(\rr d)$ by \eqref{eq:diracseries1} and \eqref{eq:maclaurinseries1} respectively. 
Then 
\begin{equation}\label{eq:diracseries3}
\WF^{t,s}( u ) 
= \{ 0 \} \times V_s(u)
\end{equation}
and
\begin{equation}\label{eq:taylorseries3}
\WF^{s,t}( f )
= V_s(u) \times \{ 0 \}. 
\end{equation}
\end{prop}

\begin{proof}
Again Fourier transformation gives $\widehat u = (2 \pi)^{-\frac{d}{2}} f$ so again by Proposition \ref{prop:WFstsymplectic} (i) it suffices to show \eqref{eq:diracseries3}. 

As for the inclusion
\begin{equation*}
\WF^{t,s}( u )  \subseteq \{ 0 \} \times V_s(u)
\end{equation*}
it is a direct consequence of Proposition \ref{prop:WFsfreqaxis} since $V_s(u) = \pi_2 \WF_s(u)$.  

The opposite inclusion cannot be deduced from Proposition \ref{prop:WFGevreyWFst} and Corollary \ref{cor:WFsfreqaxis}, 
because of the restrictive assumption $t \geqs s$ there. 
Instead we argue as follows. 
We know from Proposition \ref{prop:diracseries1} that 
$\WF^{t,s}( u ) \subseteq \{ 0 \} \times ( \rr d \setminus 0 )$. 
Assume $\xi_0 \in \rr d \setminus 0$ and $(0,\xi_0) \notin \WF^{t,s}( u )$. 
Let $\fy \in \Sigma_t^s (\rr d)$ satisfy $\fy(0) = 1$ and $\pd \alpha \fy(0) = 0$ for all $\alpha \neq 0$, 
which is possible since $s > 1$. 
If we fix $x = 0$ in \eqref{eq:notinWFGFst1} and assume there $U = A \times B \subseteq \rr {2d}$ where $A \subseteq \rr d$ is a neighborhood of $0$ and $B \subseteq \rr d$ is a neighborhood of $\xi_0$, we obtain
\begin{equation*}
\sup_{\lambda > 0, \ \xi \in B}
e^{r \lambda} |V_\fy u(0, \lambda^s \xi)| < + \infty \quad \forall r > 0. 
\end{equation*}

Since 
\begin{align*}
V_\fy u (0,\xi) 
& = (2 \pi)^{- \frac{d}{2}} \sum_{\alpha \in \nn d} c_\alpha \sum_{\beta \leqs \alpha} \binom{\alpha}{\beta} \xi^\beta \overline{ D^{\alpha-\beta} \fy (0)} \\
& = (2 \pi)^{- \frac{d}{2}} \sum_{\alpha \in \nn d} c_\alpha \xi^\alpha = (2 \pi)^{- \frac{d}{2}} f(\xi),
\end{align*}
\eqref{eq:Scomplement} is satisfied with $U = B$
and we conclude $\xi_0 \notin V_s(u)$. 
Thus $\{ 0 \} \times V_s(u) \subseteq \WF^{t,s}( u )$. 
\end{proof}

\section{The $t,s$-Gelfand--Shilov wave front set of an exponential function}\label{sec:exponential}

For $z \in \cc d$ fixed consider the exponential function $\rr d \ni x \mapsto a(x) = e^{\la x, z \ra}$. 
If $s > 0$, $0 < t \leqs 1$, $s + t > 1$ and $\fy \in \Sigma_t^s(\rr d)$ then by \eqref{eq:Sstderivativeestimate1}
we have for some $h > 0$
\begin{equation*}
\left| \int_{\rr d} a(x) \overline{ \fy (x)} \dd x \right|
\leqs \| \fy \|_{\mathcal S_{t,h}^s} \int_{\rr d} e^{|z| |x| - (|z|+1) |x|^{\frac{1}{t}}} \dd x
\lesssim \| \fy \|_{\mathcal S_{t,h}^s} 
\end{equation*}
which implies $a \in (\Sigma_t^s)'(\rr d)$. 
We consider $a$ as the multiplier operator $T f = a f$. 
Then $T = a^w(x,D)$ with $a(x,\xi) = a(x) = e^{\la x, z \ra}$. 
From \eqref{eq:expestimate0} for any $h > 0$
we obtain for any $\alpha \in \nn d$
\begin{equation*}
|\pd \alpha a(x)| = |z^\alpha | e^{\re \la x,z \ra}
\leqs C_{s,d,h} ( h |z|)^{|\alpha|} \alpha!^s e^{|\re z| |x|}. 
\end{equation*}
This means that $a \in \Gamma_{t,s}^{s,t; 0} (\rr {2d})$ for all $0 < t \leqs 1$, $s > 0$, $s + t > 1$. 

Theorem \ref{thm:microlocalWFs} 
combined with Proposition \ref{prop:WFstelementary}
now gives 
\begin{equation}\label{eq:exponentialWF1}
\WF^{t,s}(  e^{\la \cdot , z \ra} ) \subseteq ( \rr d \setminus 0 ) \times \{ 0 \} 
\end{equation}
for any $z \in \cc d$.

By considering the operator 
$T^{-1}$ with symbol $e^{-\la x, z \ra} \in \Gamma_{t,s}^{s,t; 0} (\rr {2d})$, 
so that 
$T^{-1} ( e^{ \la \cdot, z \ra} ) = 1$, 
we deduce the opposite inclusion. 
We have obtained: 

\begin{prop}\label{prop:exponentialWFs}
If $0 < t \leqs 1$, $s > 0$, $s + t > 1$, and $z \in \cc d$ then
\begin{equation*}
\WF^{t,s} ( e^{\la \cdot ,z \ra} ) = ( \rr d \setminus 0 )  \times \{ 0 \}. 
\end{equation*}
\end{prop}

\begin{cor}\label{cor:diracseriesexp}
If $0 < t \leqs 1$, $s > 0$, $s + t > 1$, 
$z \in \cc d$ and
\begin{equation*}
u = \sum_{\alpha \in \nn d} \frac{z^\alpha}{\alpha!} (-D)^\alpha \delta_0 
\end{equation*}
then $u \in (\Sigma_s^t)'(\rr d)$ and 
\begin{equation*}
\WF^{s,t} ( u ) =  \{ 0 \} \times ( \rr d \setminus 0 ) . 
\end{equation*}
\end{cor}

\begin{proof}
We have the Maclaurin series 
\begin{equation}\label{eq:maclaurinseries0}
f(x) = e^{\la x ,z \ra} = \sum_{\alpha \in \nn d} \frac{z^\alpha}{\alpha!} x^\alpha, \quad x \in \rr d, 
\end{equation}
which converges in $(\Sigma_t^s)'(\rr d)$ to $f \in (\Sigma_t^s)'(\rr d)$. 
We apply the Fourier transform termwise with convergence in $(\Sigma_s^t)'(\rr d)$ which gives
\begin{align*}
\wh f & = \sum_{\alpha \in \nn d} \frac{z^\alpha}{\alpha!} \cF (x^\alpha) \\
& = (2 \pi)^{\frac{d}{2}} \sum_{\alpha \in \nn d} \frac{z^\alpha}{\alpha!} (-D)^\alpha \delta_0 
= (2 \pi)^{\frac{d}{2}} u \in (\Sigma_s^t)'(\rr d). 
\end{align*}

Proposition \ref{prop:WFstsymplectic} (i)
and Proposition \ref{prop:exponentialWFs} 
now give
\begin{equation*}
\WF^{s,t}(u) = \J \WF^{t,s}(f) = \{ 0 \} \times ( \rr d \setminus 0 ) . 
\end{equation*}
\end{proof}

\begin{rem}\label{rem:fourierexponential}
Note that $u$ may be considered as the Dirac distribution
\begin{equation*}
(u,\fy) = \overline{\fy( i \overline z)}, \quad \fy \in \Sigma_s^t (\rr d), 
\end{equation*}
which makes sense since $\fy$ extends to an entire function on $\cc d$ as $t \leqs 1$. 
If $z = i \xi$ with $\xi \in \rr d$ we recapture the well known identity
\begin{equation*}
\cF (  e^{i \la \cdot, \xi \ra} ) = \wh f = (2 \pi)^{\frac{d}{2}} u = (2 \pi)^{\frac{d}{2}} \delta_{\xi} \in \cD'(\rr d). 
\end{equation*}
\end{rem}

\begin{rem}\label{rem:exponential}
Proposition \ref{prop:diracseries1} contains the ``$\subseteq$'' inclusion of 
Proposition \ref{prop:exponentialWFs} and Corollary \ref{cor:diracseriesexp}, under the restriction $0 < t < 1$ (that is avoiding $t=1$), as a particular case. 
In fact comparing \eqref{eq:maclaurinseries0} with \eqref{eq:maclaurinseries1} we can identify the Maclaurin coefficients for $f = e^{\la \cdot, z \ra}$ where $z \in \cc d$. They are $c_\alpha = z^\alpha/\alpha!$. 
If $0 < s < 1$ we have for any $r > 0$, and $0 < a < 1$
\begin{align*}
\sum_{\alpha \in \nn d} |c_\alpha| \, r^{|\alpha|} \alpha!^{s}
& \leqs \sum_{\alpha \in \nn d} (|z| r)^{|\alpha|} \alpha!^{s-1}
= \sum_{\alpha \in \nn d} a^{|\alpha|} \left( \frac{(|z| r a^{-1})^{\frac{|\alpha|}{1-s}}}{\alpha!} \right)^{1-s} \\
& \leqs \sum_{\alpha \in \nn d} a^{|\alpha|} 
\left( \frac{ \left( d (|z| r a^{-1})^{\frac{1}{1-s}} \right)^{|\alpha|} }{|\alpha|!} \right)^{1-s} \\
& \leqs (1 -a )^{-d} \exp \left( (1-s) d (|z| r a^{-1})^{\frac{1}{1-s}} \right). 
\end{align*}
By Proposition \ref{prop:diracseries1} we may conclude
\begin{equation*}
\WF^{s,t}( e^{\la \cdot, z \ra} ) 
\subseteq ( \rr d \setminus 0 ) \times \{ 0 \}
\end{equation*}
and
\begin{equation*}
\WF^{t,s}( u ) 
\subseteq \{ 0 \} \times ( \rr d \setminus 0 )
\end{equation*}
where 
\begin{equation*}
u = \sum_{\alpha \in \nn d} \frac{z^\alpha}{\alpha!} (- D)^\alpha \delta_0. 
\end{equation*}
\end{rem}

By combining with the results of Section \ref{sec:chirp} we finally consider in dimension $d = 1$
\begin{equation}\label{eq:expchirp}
v (x)  = e^{ z x + i c x^{m} }
\end{equation}
with $z \in \co$, $c \in \ro \setminus 0$, $m \in \no$, $m \geqs 2$. 
Then $v \in (\Sigma_t^s)'(\rr d)$ if $0 < t \leqs 1$, $s > 0$, $s + t > 1$.

\begin{prop}\label{prop:expchirp}
If $\frac1{m-1} < t \leqs 1$ then for $v$ defined by \eqref{eq:expchirp} we have 
\begin{equation}\label{eq:WFexpchirp}
\WF^{t,t(m-1)} (v) = \{ (x, c m x^{m-1}) \in \rr 2, \ x \neq 0  \}. 
\end{equation}

\end{prop}

\begin{proof}
As before define $T = a^w(x,D)$ with $a(x,\xi) = e^{zx}$ regarded as a symbol in
$\Gamma_{t,s}^{s,t;0}(\rr 2)$, for any $s > 0$ such that $s + t > 1$ and $t \leqs 1$.
Set 
\begin{equation*}
w (x)  = e^{i c x^{m} } \in \cS'(\ro) \subseteq (\Sigma_t^s)'(\ro). 
\end{equation*}
We have $v = T w$. 
From Theorem \ref{thm:microlocalWFs} we deduce 
\begin{equation*}
\WF^{t,s}( v ) \subseteq \WF^{t,s}(w). 
\end{equation*}
By considering the operator $T^{-1}$ we deduce the opposite inclusion, 
hence $\WF^{t,s}( v ) = \WF^{t,s}(w)$. 
Under the assumption $t > \frac1{m-1}$ we may apply 
Theorem \ref{thm:chirpWFst}, and obtain \eqref{eq:WFexpchirp}. 
\end{proof}

\section*{Acknowledgment}
Work partially supported by the MIUR project ``Dipartimenti di Eccellenza 2018-2022'' (CUP E11G18000350001).



\begin{thebibliography}{2000}

%
\bibitem{Abdeljawad1}
A.~Abdeljawad, M.~Cappiello and J.~Toft, \textit{Pseudo-differential operators in anisotropic Gelfand--Shilov setting}, 
Integr. Equ. Oper. Theory \textbf{91} 26  (2019).
%
\bibitem{Boiti1}
C.~Boiti, D.~Jornet and A.~Oliaro, \textit{The Gabor wave front set in spaces of ultradifferentiable functions}, 
Monatsh. Math. \textbf{188} (2019), 199--246. 
%
\bibitem{Cappiello0a}
M.~Cappiello, T.~Gramchev and L.~Rodino, \textit{Semilinear pseudo-differential equations and travelling waves}, 
Fields Institute Communications \textbf{52} (2007), 213--238. 
%
\bibitem{Cappiello0b}
M.~Cappiello, T.~Gramchev and L.~Rodino, \textit{Entire extensions and exponential decay for semilinear elliptic equations}, 
J. Anal. Math. \textbf{111} (2010), 339--367. 
%
\bibitem{Cappiello1}
M.~Cappiello and R.~Schulz, \textit{Microlocal analysis of quasianalytic Gelfand--Shilov type ultradistributions}, 
Compl. Var. Elliptic Equ. \textbf{61} (4) (2016), 538--561.
%
\bibitem{Cappiello2}
M.~Cappiello and J.~Toft, \textit{Pseudo-differential operators in a  Gelfand--Shilov setting}, 
Math. Nachr. \textbf{290} (5--6) (2017), 738--755.
%
\bibitem{Carypis1}
E.~Carypis and P.~Wahlberg, \textit{Propagation of exponential phase space singularities for Schr\"odinger equations with quadratic Hamiltonians}, J. Fourier Anal. Appl. \textbf{23} (3) (2017), 530--571. Correction: \textbf{27}:35 (2021). 
%
\bibitem{Cordero1}
E.~Cordero and L.~Rodino, \textit{Time-Frequency Analysis of Operators}, 
De Gruyter Studies in Mathematics \textbf{75}, De Gruyter, Berlin, 2020. 
%
\bibitem{Debrouwere1}
A.~Debrouwere and J.~Vindas, \textit{Topological properties of convolutor spaces via the short-time Fourier transform}, 
Trans. Amer. Math. Soc. \textbf{374} (2021), 829--861. 
%
\bibitem{Folland1}
G.~B.~Folland, \textit{Harmonic Analysis in Phase Space}, Princeton University Press, 1989.
%
\bibitem{Garello1}
G.~Garello and A.~Morando, \textit{$m$-Microlocal elliptic pseudodifferential operators acting on $L_{\rm loc}^p (\Omega)$}, 
Math. Nachr. \textbf{289} (2016), 1820--1837. 
%
\bibitem{Gelfand2} 
I.~M.~Gel'fand and G.~E.~Shilov, 
\emph{Generalized Functions}, Vol. II,
Academic Press, New York London, 1968.
%
\bibitem{Grochenig1}
K.~Gr\" ochenig, \textit{Foundations of Time-Frequency Analysis}, Birkh\" auser, Boston, 2001.
%
\bibitem{Grochenig2}
\bysame, \textit{Time-frequency analysis of Sj\"ostrand's class}, Rev. Mat. Iberoamer. \textbf{22} (2), 703--724, 2006.
%
\bibitem{Hormander0}
L.~H\"ormander,
\textit{The Analysis of Linear Partial Differential Operators}, Vol. I, III, IV,
Springer, Berlin, 1990.
%
\bibitem{Hormander1}
L.~H\"ormander, \textit{Quadratic hyperbolic operators}, Microlocal Analysis and Applications, Lecture Notes in Math. \textbf{1495}, Eds. L. Cattabriga, L. Rodino, pp. 118--160, Springer, 1991.
%
\bibitem{Hormander2}
L.~H\"ormander, \textit{Symplectic classification of quadratic forms, and general Mehler formulas}, Math. Z. \textbf{219} (1995), 413--449.  
%
\bibitem{Kaneko1}
A.~Kaneko, \textit{Introduction to hyperfunctions}, 
Kluwer Acad. Publ., 1988. 
%
\bibitem{Kaneko2}
A.~Kaneko, \textit{On the flabbiness of the sheaf of Fourier hyperfunctions}, 
Sci. Pap. Coll. Gen. Educ. Univ. Tokyo \textbf{36} (1986), 1--14. 
%
\bibitem{Komatsu1}
H.~Komatsu, \textit{Ultradistributions, II: the kernel theorem and ultradistributions with support in a submanifold}, 
J. Fac. Sci. Univ. Tokyo Sect. IA Math. \textbf{24} (3) (1977), 607--628. 
%
\bibitem{Lascar1} 
R.~Lascar, \textit{Propagation des singularit\'es d'\'equations pseudodiff\'erentielles quasi homog\`enes}, 
Ann. Inst. Fourier Grenoble \textbf{27} (1977), 79--123. 
%
\bibitem{Morimoto1} 
Y.~Morimoto, N.~Lerner, K.~Pravda-Starov and C.-J.~Xu, \textit{Gelfand--Shilov and Gevrey smoothing effect for the spatially inhomogeneous non-cutoff Kac equation}, 
J. Funct. Anal. \textbf{269} (2015), 459--535.  
%
\bibitem{Morimoto2} 
Y.~Morimoto, N.~Lerner, K.~Pravda-Starov and C.-J.~Xu, \textit{Gelfand--Shilov smoothing properties of the radially symmetric spatially homogeneous Boltzmann equation without angular cut-off}, 
J. Differ. Equations \textbf{256} (2014), 797--831. 
%
\bibitem{Nicola1}
F.~Nicola and L.~Rodino, \textit{Global Pseudo-Differential Calculus on Euclidean Spaces}, Birkh\"auser, Basel, 2010.
%
\bibitem{Parenti1}
C.~Parenti and L.~Rodino, \textit{Parametrices for a class of pseudo-differential operators I, II}, 
Ann. Mat. Pura Appl. \textbf{125} (1980), 221--254 and 255--278. 
%
\bibitem{Petersson1}
A.~Petersson, \textit{Fourier characterizations and non-triviality of Gelfand-Shilov spaces, with applications to Toeplitz operators}, 
accepted for publication, J. Fourier Anal. Appl. (2023). 
%
\bibitem{Pilipovic1}
S.~Pilipovi{\'c}, B. Prangoski and J.~Vindas, \textit{On quasianalytic classes of Gelfand--Shilov type}, 
J. Math. Pures Appl. \textbf{116} (2018), 174--210. 
%
\bibitem{PRW1}
K. Pravda-Starov, L.~Rodino and P.~Wahlberg, \textit{Propagation of Gabor singularities for Schr\"odinger equations with quadratic Hamiltonians}, Math. Nachr. \textbf{291} (1) (2018), 128--159. 
%
\bibitem{Rodino1}
L. Rodino, \textit{Linear Partial Differential Operators in Gevrey Spaces}, World Scientific, Singapore New Jersey London Hong Kong, 1993. 
%
\bibitem{Rodino2}
L. Rodino and P. Wahlberg, \textit{The Gabor wave front set}, Monaths. Math. \textbf{173} (4) (2014), 625--655. 
%
\bibitem{Rodino3}
L. Rodino and P. Wahlberg, \textit{Anisotropic global microlocal analysis for tempered distributions}, 
Monatsh. Math., available online (2022). 
%
\bibitem{Schaefer1}
H.~H.~Schaefer and M.~P.~Wolff, \textit{Topological Vector Spaces}, 
Springer-Verlag, New York, 1999. 
%
\bibitem{Schulz1}
R.~Schulz and P.~Wahlberg, \textit{Equality of the homogeneous and the Gabor wave front set}, Comm. PDE \textbf{42} (5) (2017), 703--730. 
%
\bibitem{Shubin1}
M.~A.~Shubin, \textit{Pseudodifferential Operators and Spectral Theory}, Springer, 2001.
%
\bibitem{Sobak1}
M.~Sobak, \textit{The Paley--Wiener theorems for Gevrey functions and ultradistributions}, 
Bachelor thesis, Linn\ae us University, 2018. 
%
\bibitem{Teofanov1}
N.~Teofanov, \textit{The Grossmann--Royer transform, Gelfand--Shilov spaces, and continuity properties of localization operators on modulation spaces}, 
in: Mathematical Analysis and Applications -- Plenary Lectures, L. Rodino, J. Toft (eds.), 
ISAAC 2017, Springer Proceedings in Mathematics \& Statistics \textbf{262}, Springer, Cham, 2018. 
%
\bibitem{Toft1}
J.~Toft, \textit{The Bargmann transform on modulation and Gelfand--Shilov spaces, with applications to Toeplitz and pseudo-differential operators}, 
J. Pseudo-Differ. Oper. Appl. \textbf{3} (2) (2012), 145--227. 
%
\bibitem{Wahlberg2}
P.~Wahlberg, \textit{The Gabor wave front set of compactly supported distributions}, Advances in Microlocal and Time-Frequency Analysis, 
P. Boggiatto, M. Cappiello, E. Cordero, S. Coriasco, G. Garello, A. Oliaro, J. Seiler (Eds.)
Birkh\"auser Verlag, pp. 507--520, 2020. 
%
\bibitem{Wahlberg3}
P.~Wahlberg, \textit{Semigroups for quadratic evolution equations acting on Shubin--Sobolev and Gelfand--Shilov spaces}, 
Ann. Fenn. Math. \textbf{47} (2) (2022), 821--853. 
%
\bibitem{Wahlberg4}
P.~Wahlberg,  \emph{Propagation of anisotropic Gelfand--Shilov wave front sets}, 
J. Pseudo-Differ. Oper. Appl. \textbf{14} (7), 2023.
%
\bibitem{Wahlberg5}
P.~Wahlberg,  \emph{Propagation of anisotropic Gabor wave front sets}, 
arXiv:2301.03190 [math.AP] (2023).
%

\end{thebibliography}
\end{document}